\documentclass[11pt]{article}

\usepackage{graphicx, amsmath,amsthm,amssymb,url}
\usepackage{empheq}

\usepackage[colorlinks=true,linkcolor=blue,urlcolor=black,citecolor=blue,breaklinks]{hyperref}
\usepackage{algorithm,algcompatible,mathtools}
\usepackage[multidot]{grffile}
\usepackage[normalem]{ulem} 
\usepackage{enumerate}
\usepackage{thmtools}
\usepackage{thm-restate}
\usepackage{subfigure}
\usepackage{fullpage,etoolbox}
\usepackage{color}
\usepackage[draft]{commands}
\usepackage{mathrsfs}
\DeclareMathAlphabet{\mathpzc}{OT1}{pzc}{m}{it}

\numberwithin{equation}{section}

\title{System Level Synthesis}
\author{James Anderson, John C. Doyle, Steven Low, Nikolai Matni}
\date{\today}

\begin{document}
\maketitle

\begin{abstract}
This article surveys the \emph{System Level Synthesis} framework, which presents a novel perspective on constrained robust and optimal controller synthesis for linear systems.  We show how SLS shifts the controller synthesis task from the design of a controller to the design of the entire closed loop system, and highlight the benefits of this approach in terms of scalability and transparency.  We emphasize two particular applications of SLS, namely large-scale distributed optimal control and robust control.  In the case of distributed control, we show how SLS allows for localized controllers to be computed, extending robust and optimal control methods to large-scale systems under practical and realistic assumptions.  In the case of robust control, we show how SLS allows for novel design methodologies that, for the first time, quantify the degradation in performance of a robust controller due to model uncertainty -- such transparency is key in allowing robust control methods to interact, in a principled way, with modern techniques from machine learning and statistical inference.  Throughout, we emphasize practical and efficient computational solutions, and demonstrate our methods on easy to understand case studies.
\end{abstract}

\section{Introduction}
\label{sec:intro}
The systems we seek to design and control are becoming increasingly complex, be it in their dynamics, their scale, or their interaction with the environment.  To address this complexity, control methods must be able to accommodate the large-scale distributed nature of many of these systems (e.g.,  smart-grids, intelligent transportation systems, software defined networks).  But this in itself is often not sufficient, as certain components of the system, or the environment in which it must interact, are too difficult or impossible to model from first principles, thus motivating a principled integration of data-driven techniques.  While there has been a tremendous amount of work in each of these areas, usually targeting a specific problem setting, what is lacking is a unified framework that allows for all of these issues to be simultaneously addressed in a principled manner.  This paper summarizes our progress towards developing such a unified framework in the context of linear optimal control.  As we describe in more detail below, by moving from controller synthesis to \emph{System Level Synthesis}, wherein the closed loop behavior and properties of the system are directly optimized over, we are able to make significant improvements over the state of the art in terms of computational scalability, richness of problems that can be tackled, and the principled integration of models produced by data-driven methods, among other things.

We begin by providing a brief overview of relevant results from classical optimal control, robust control, and distributed optimal control.  We then highlight our main contributions, and end this section by providing a roadmap to the reader for the rest of the paper.  Our goal is to convey the core concepts of SLS to as broad an audience as possible, while also providing deep dives into specific domains (such as robust and distributed control), in hopes of appealing to both the interested graduate student (be they specializing in control theory, optimization, machine learning, or artificial intelligence), as well as the domain experts. 

\subsection{Prior Work}
\label{sec:prior}
\subsubsection{Classical Controller Parameterizations and Synthesis Results}
The foundation of many optimal controller synthesis procedures is a parameterization of all internally stabilizing controllers, and the responses that they achieve, over which relevant performance measures can be easily optimized.  \revsecond{For finite dimensional linear-time-invariant (LTI) systems, the class of internally stabilizing LTI feedback controllers is characterized by the celebrated Youla parameterization \cite{1976_Youla_parameterization} and the closely related factorization approach \cite{Factorization}. In \cite{1976_Youla_parameterization}, the authors showed that the Youla parameterization defines an isomorphism between a stabilizing controller and the resulting closed loop system response from sensors to actuators -- therefore rather than synthesizing the controller itself, this system response (or Youla parameter) could be directly optimized. This allowed for the incorporation of customized design specifications on the closed loop system into the controller design process via convex optimization \cite{Boyd_closed} or interpolation \cite{Dahleh}. Subsequently, analogous parameterizations of stabilizing controllers for more general classes of systems were developed: notable examples include the polynomial approach \cite{Polynomial} for generalized Rosenbrock systems \cite{Rosenbrock} and the behavioral approach \cite{Willems,Behavior_I, Behavior_II, Behavior_para} for linear differential systems. These results illustrate the power and generality of the Youla parameterization and factorization approaches to optimal control in the centralized setting.} Together with state-space methods, they played a major role in shifting controller synthesis from an ad hoc, loop-at-a-time tuning process to a principled one with well defined notions of optimality, and in the LTI setting, paved the way for the foundational results of robust and optimal control that would follow \cite{1989_DGKF}.

\subsubsection{Robust Control}
Building on these classical results, a rich line of work looked to explicitly account for modeling uncertainty in the design procedure.  In particular, systems were typically viewed as being composed of a nominal system estimate and a set of unknown but bounded model errors.
It was therefore necessary to ensure that the computed controller ensured stability
and performance guarantees for any such admissible realization.  When modeling errors to the nominal system are allowed to be arbitrary
norm-bounded linear time-invariant (LTI) operators in feedback with the nominal plant, traditional
small-gain theorems and robust synthesis techniques can be applied to exactly
solve the problem \cite{dahleh1994control,ZDGBook}.  However, when the
errors are known to have more structure there are more sophisticated techniques based on structured
singular values and corresponding $\mu$-synthesis techniques
\cite{doyle-ssv,mu-mixed,mu-complex,mu-real} or integral quadratic constraints
(IQCs) \cite{megretski1997system}.  

While theoretically appealing and much less conservative than traditional small-gain approaches,
the resulting synthesis methods are both computationally intractable (although
effective heuristics do exist) and difficult to interpret analytically.  In
particular, we know of no results in the literature that bound the degradation in performance of controlling an uncertain system in terms of the size of the perturbations affecting it.  While traditional results in robust control considered fixed model uncertainty, in the data-driven and learning based control setting, analytic interpretability of the effects uncertainty size are essential in obtaining sample-complexity bounds for these methods \cite{learning-lqr, dean18safe, dean18regret,boczar2018finite}.

\subsubsection{Distributed Control}
\rev{However, as control engineers shifted their attention from centralized to distributed optimal control, it was observed that the parameterization approaches that were so fruitful in the centralized setting were no longer directly applicable.} In contrast to centralized systems, modern cyber-physical systems (CPS) are large-scale, physically distributed, and interconnected.  Rather than a logically centralized controller, these systems are composed of several sub-controllers, each equipped with their own sensors and actuators -- these sub-controllers  exchange locally available information (such as sensor measurements or applied control actions) via a communication network.  \revsecond{These information sharing constraints make the corresponding distributed optimal controller synthesis problem challenging to solve \cite{1972_Ho_info,2012_Mahajan_Info_survey,2006_Rotkowitz_QI_TAC,2002_Bamieh_spatially_invariant,2005_Bamieh_spatially_invariant,2013_Nayyar_common_info}. In particular, imposing such structural constraints on the controller can lead to optimal control problems that are NP-hard \cite{1968_Witsenhausen_counterexample,1984_Tsitsiklis_NP_hard}.}

Despite these technical and conceptual challenges, a body of work \cite{2002_Bamieh_spatially_invariant, 2004_Dullerud, 2005_Bamieh_spatially_invariant, 2006_Rotkowitz_QI_TAC, 2012_Mahajan_Info_survey, 2013_Nayyar_common_info} that began in the early 2000s, and that culminated with the introduction of ``quadratic invariance'' (QI) in the seminal paper \cite{2006_Rotkowitz_QI_TAC}, showed that for a large class of practically relevant LTI systems, such internal structure could be integrated with the Youla parameterization and still preserve the convexity of the optimal controller synthesis task.  Informally, a system is quadratically invariant if sub-controllers are able to exchange information with each other faster than their control actions propagate through the CPS \cite{2010_Rotkowitz_QI_delay}.  Even more remarkable is that this condition is tight, in the sense that QI is a necessary \cite{2014_Lessard_convexity} and sufficient \cite{2006_Rotkowitz_QI_TAC} condition for subspace constraints (defined by, for example, communication delays) on the controller to be enforceable via convex constraints on the Youla parameter. \revsecond{The identification of QI triggered an explosion of results in distributed optimal controller synthesis \cite{2012_Lessard_two_player,2010_Shah_H2_poset,2013_Lamperski_H2,2013_Lessard_structure,2013_Scherer_Hinf,2014_Lessard_Hinf,2014_Matni_Hinf,2014_Tanaka_Triangular,2014_Lamperski_state} -- these results showed that the robust and optimal control methods that proved so powerful for centralized systems could be ported to the distributed setting.} \rev{As far as we are aware, no such results exist for the more general classes of systems considered in \cite{Polynomial,Willems,Behavior_I,Behavior_II,Behavior_para}.}

\rev{However, a fact that is not emphasized in the distributed optimal control literature is that distributed controllers are actually more complex to synthesize and implement than their centralized counterparts.\footnote{For example, see the solutions presented in \cite{2012_Lessard_two_player,2010_Shah_H2_poset,2013_Lamperski_H2,2013_Lessard_structure,2013_Scherer_Hinf,2014_Lessard_Hinf,2014_Matni_Hinf,2014_Tanaka_Triangular,2014_Lamperski_state} and the message passing implementation suggested in \cite{2014_Lamperski_state}.} } \rev{In particular, a major limitation of the QI framework is that, for strongly connected systems,\footnote{We say that a plant is strongly connected if the state of any subsystem can eventually alter the state of all other subsystems.} it cannot provide a convex characterization of localized controllers, in which local sub-controllers only access a subset of system-wide measurements. This need for global exchange of information between sub-controllers is a limiting factor in the scalability of the synthesis and implementation of these distributed optimal controllers.}

\subsection{Main Contributions}
Motivated by the issues raised above, we propose a  novel parameterization of internally stabilizing controllers and the closed loop responses that they achieve, providing an alternative approach to constrained optimal and robust controller synthesis. Specifically, rather than directly designing only the feedback loop between sensors and actuators, as in the Youla framework, we propose directly designing the entire closed loop response of the system, as captured by the maps from process and measurement disturbances to control actions and states.  As such, we call the proposed method a System Level Approach (SLA) to controller synthesis, which is composed of three elements: System Level Parameterizations (SLPs), System Level Constraints (SLCs) and System Level Synthesis (SLS) problems.  Further, in contrast to the QI framework, which seeks to impose structure on the input/output map between sensor measurements and control actions, the SLA imposes structural constraints on the system response itself, and shows that this structure carries over to the \emph{internal realization} of the corresponding controller.  It is this conceptual shift from structure on the input/output map to the internal realization of the controller that allows us to expand the class of structured controllers that admit a convex characterization, and in doing so, vastly increase the scalability of distributed optimal control methods.  Further, by moving our focus from controllers to system responses, we are able to make transparent the effects of modeling uncertainty and approximations on controller synthesis and performance.  We summarize our main contributions below. 

This paper summarizes, and in some places extends, our recent theoretical and computational contributions \cite{2014_Wang_ACC,2014_Wang_CDC,2014_Wang_Allerton,2015_Wang_H2,2015_Wang_Reg,2015_Wang_LDKF,WMD_ACC17, Virtual, SLStutorial, Thesis, anderson17, SysLevelSyn1,SysLevelSyn2, learning-lqr} to the area of constrained optimal and robust controller synthesis. In particular, we
\begin{itemize}
\item define and analyze the system level approach to controller synthesis, which is built around novel SLPs of all stabilizing controllers and the closed loop responses that they achieve;
\item in the state-feedback setting, we show that the above parameterizations are stable under small perturbations, leading to novel robust control problems that, among other features, allow us to bound performance degradation as a function of modeling uncertainty size;
\item show that SLPs allow us to constrain the closed loop response of the system to lie in arbitrary sets: we call such constraints on the system SLCs.  If these SLCs admit a convex representation, then the resulting set of constrained system responses admits a convex representation as well;
\item show that such constrained system responses can be used to directly implement a controller achieving them -- in particular, any SLC imposed on the system response imposes a corresponding SLC on the internal structure of the resulting controller;
\item show that the set of constrained stabilizing controllers that admit a convex parameterization using SLPs and SLCs is a strict superset of those that can be parameterized using quadratic invariance -- hence we provide a generalization of the QI framework, characterizing the broadest known class of constrained controllers that admit a convex parameterization;
\item formulate and analyze the SLS problem, which exploits SLPs and SLCs to define the broadest known class of constrained optimal control problems that can be solved using convex programming.  We show that the optimal control problems considered in the QI literature \cite{2012_Mahajan_Info_survey}, as well as the recently defined localized optimal control framework \cite{2015_Wang_H2} are all special cases of SLS problems;
\item show that if \emph{locality} constraints are imposed on the system response, and the SLCs admit a certain (partial) decomposability property, the computational complexity of controller synthesis and implementation can be made to scale as $O(1)$ relative the size of the full system, assuming sufficient parallel computation.
\end{itemize}

\subsection{Paper Structure}
{The paper is organized as follows. We begin with a gentle warmup in Section~\ref{sec:sys_resp_intro}, wherein we consider full-state-feedback problems over a finite time horizon -- by restricting ourselves to this simplified setting, we are able to build up nearly all of the core machinery of SLS with minimal technical overhead.  We show how several known optimal and robust control synthesis problems can be posed in the SLS framework, essentially turning (robust) optimal control problems into (robust) optimization problems over system responses. 

In Section~\ref{sec:LTIinfinite_horizon} we formally present the general optimal control problem formulation, and briefly review classical solution approaches. With this information in hand, the finite time horizon results of Section~\ref{sec:sys_resp_intro} are extended to the infinite horizon setting in Section~\ref{sec:state_feedback_SLS}. The notion of locality is introduced in Section~\ref{sec:locality} and then used to show how to decompose the synthesis problem into uncoupled subproblems of  reduced dimension in Section~\ref{sec:scalability}.  We end with a robust extension of the previous results, and show how it can be used to guarantee near-optimal performance of computationally tractable solutions to SLS problems.  Finally, in Section \ref{sec:output}, we extend these results to the output-feedback setting.  We end with conclusions in Section \ref{sec:conclusion}.
}

\section{Working with System Responses}\label{sec:sys_resp_intro}
To introduce the notion of a system response, and to illustrate some of its benefits, we begin by considering linear optimal control problems over a finite time horizon with full state-feedback.  We show how the corresponding controller synthesis task can be cast as an optimization over the closed loop behavior of the system, i.e., over system responses, as opposed to over the controller itself.  By initially restricting ourselves to a finite time horizon, we are not burdened by the technical overhead associated with system stability, controller internal stability, or infinite dimensional optimization problems.  In subsequent sections, we build upon the material presented here and extend these methods to infinite horizon problems, distributed control problems, and to the output-feedback setting.

\subsection{Finite-horizon System Level Synthesis}
\label{sec:time}
Consider the linear-time-varying (LTV) system
\begin{equation}
x_{t+1} = A_tx_t + B_tu_t + w_t,
\label{eq:dynamics}
\end{equation}
where $x_t \in \R^n$ is the state, $u_t \in \R^p$ is the control input, and $w_t \in \R^n$ is an exogenous disturbance process.  Our goal is to study the behavior of this system over the horizon $t=0,\dots,T$ when the control input $u_t$ is a causal LTV state-feedback controller, i.e., $u_t = K_t(x_0,x_1,\dots,x_t)$ for some linear map $K_t$.  To do so, we introduce the following signals and block-lower-triangular matrices:

\begin{equation}
\tf x = \begin{bmatrix} x_0 \\ x_1 \\ \vdots \\ x_{T} \end{bmatrix}, \ \tf u = \begin{bmatrix} u_0 \\ u_1 \\ \vdots \\ u_{T} \end{bmatrix} \ \tf w = \begin{bmatrix} x_0  \\ w_0 \\ w_1 \\ \vdots \\ w_{T-1}\end{bmatrix}, \ \K = \begin{bmatrix} K^{0,0} & & & \\ K^{1,1} & K^{1,0} & & \\ \vdots & \ddots & \ddots & \\  K^{T,T} & \cdots & K^{T,1} & K^{T,0} \end{bmatrix} \label{eq:definitions},
\end{equation}
where here $\K$ should be thought of as the matrix representation of the convolution operator induced by the linear-time-varying controller $K_t$ such that $u_t = \K^{t,\star} \tf x$, for $\K^{t,\star}$ the $t$-th row of $\K$.

Further, we let $Z$ be the block-downshift operator, i.e., a matrix with identity matrices along its first block sub-diagonal and zeros elsewhere, and define
\begin{equation*}
\A := \mathrm{blkdiag}\left(A_0,A_1,\dots,A_{T-1}, 0\right), \ \B := \mathrm{blkdiag}\left(B_0,B_1,\dots,B_{T-1},0\right).
\end{equation*}
Note that in the linear-time-invariant (LTI) setting, $\A$ and $\B$ are block diagonal matrices with the same matrix being repeated along the diagonal. Using these definitions allows us to compactly rewrite the behavior of the system \eqref{eq:dynamics} over the horizon $t=0,\dots,T$ as
\begin{equation}
\begin{array}{rcl}
\tf x &=& Z\A\tf x + Z \B \tf u + \tf w \\
&=& Z(\A+\B\K)\tf x + \tf w.
\end{array}
\end{equation}

In particular, the closed loop behavior of our system under the feedback law $\K$ can be entirely characterized by the following maps,
\begin{equation}
\begin{array}{rcl}
\tf x &=& (I - Z(\A+\B\K))^{-1}\tf w\\
\tf u &=& \K(I - Z(\A+\B\K))^{-1}\tf w,
\end{array}
\label{eq:full_responses}
\end{equation}
which describe the closed loop transfer functions from the exogenous disturbance $\tf w$ to the state $\tf x$ and control input $\tf u$, respectively.  These maps are what we call \emph{system responses}, and we now show how optimal controller synthesis can be performed by directly optimizing over these system responses, as opposed to the controller map $\tf K$ itself.

To that end, we introduce two additional lower-block-triangular operators
\begin{equation}
\Phix = \begin{bmatrix} \Phi_x^{0,0} & & & \\\Phi_x^{1,1} & \Phi_x^{1,0} & & \\ \vdots & \ddots & \ddots & \\  \Phi_x^{T,T} & \cdots & \Phi_x^{T,1} & \Phi_x^{T,0} \end{bmatrix}, \ \Phiu = \begin{bmatrix} \Phi_u^{0,0} & & & \\ \Phi_u^{1,1} & \Phi_u^{1,0} & & \\ \vdots & \ddots & \ddots & \\  \Phi_u^{T,T} & \cdots & \Phi_u^{T,1} & \Phi_u^{T,0} \end{bmatrix}.
\end{equation}
Here we use $\Phix$ to denote the system response describing the map from $\tf w \to \tf x$, and similarly, $\Phiu$ to denote the system response describing the map from $\tf w \to \tf u$ -- this allows us to compactly rewrite equation \eqref{eq:full_responses} as
\begin{equation}
\begin{array}{rcl}
\begin{bmatrix}\tf x \\ \tf u \end{bmatrix} &=& \begin{bmatrix}\Phix  \\ \Phiu \end{bmatrix} \tf w.
\label{eq:toep_resp}
\end{array}
\end{equation}

Our objective is to optimize directly over these system responses $\SFpair$: in order to do so, we must ensure that there exists a linear-time-varying controller $\K$ such that equations \eqref{eq:full_responses} and \eqref{eq:toep_resp} are equivalent.  The next theorem shows that enforcing the system responses to obey the system dynamics is necessary and sufficient to ensure the existence of such a corresponding controller, i.e., to ensure that the pair $\SFpair$ is \emph{achievable}.

\begin{theorem}\label{thm:finite_horizon}
Over the horizon $t=0,\hdots, T$, the  system dynamics \eqref{eq:dynamics}  with block-lower-triangular state feedback law $\K$ defining the control action as $\tf u = \tf K \tf x$, the following are true
\begin{enumerate}
\item the affine subspace defined by
\begin{equation}
\begin{bmatrix} I - Z\A & -Z\B \end{bmatrix}\begin{bmatrix} \Phix \\ \Phiu \end{bmatrix} = I,
\label{eq:subspace}
\end{equation}
parameterizes all possible system responses \eqref{eq:toep_resp}.
\item for any block-lower-triangular matrices $\SFpair$ satisfying \eqref{eq:subspace}, the controller $\tf K = \Phiu \Phix^{-1}$ achieves the desired response.
\end{enumerate}
\label{thm:ft-param}
\end{theorem}
\begin{proof}
Proof of 1.: Let $\tf K$ be any block-lower-triangular operator, and $\tf u = \tf K \tf x$.  Then as argued above, we have that 
\begin{equation*}
\begin{array}{rcl}
\tf x &=& (I-Z(\A+\B\tf K))^{-1}\tf w \\
\tf u &=& \tf K(I-Z(\A+\B\tf K))^{-1}\tf w.
\end{array}
\end{equation*}
It is then easily seen that
\begin{equation*}
\begin{bmatrix} I - Z\A \ -Z\B\end{bmatrix}\begin{bmatrix}(I-Z(\A+\B\tf K))^{-1} \\ \tf K (I-Z(\A+\B\tf K))^{-1} \end{bmatrix} = (I - Z\A -Z\B\tf K)(I-Z(\A+\B\tf K))^{-1} = I.
\end{equation*}
Proof of 2.:  First notice that the affine constraint \eqref{eq:subspace} implies that $\Phix^{t,0} = I$ for $t=0,\dots,T$, and hence $\Phix^{-1}$ exists. Let $\tf K = \Phiu \Phix^{-1}$ and notice that
\begin{equation*}
\tf x = (I-Z(\A+\B\Phiu \Phix^{-1}))^{-1}\tf w.
\end{equation*}
But we then have that
\begin{equation*}
(I-Z(\A+\B\Phiu \Phix^{-1}))^{-1} = ((I-Z\A)\Phix - Z\B \Phiu)\Phix^{-1})^{-1} = \Phix ((I-Z\A)\tf \Phix - Z\B \Phiu)^{-1}  = \Phix
\end{equation*}
where the last equality follows form the fact that $\SFpair$ satisfy \eqref{eq:subspace}.  Similarly we have that
\begin{equation*}
\tf u = \Phiu \Phix^{-1}\tf x = \Phiu \Phix^{-1} \Phix \tf w = \Phiu \tf w,
\end{equation*}
where the second equality follows from the fact that $\tf x = \Phix \tf w$.
\end{proof}

At this point we emphasize that there is nothing ``distributed'' about the formulation or Theorem~\ref{thm:finite_horizon}. As we will see in later sections, the implications of Theorem~\ref{thm:finite_horizon} and its various generalizations will provide the building blocks for a unified and scalable distributed control framework.

Finally, before providing examples of how this parameterization can be used to formulate SLS problems, we make a small comment on controller realization.  Even in the finite horizon setting, where no issues relating internal stability arise, one may wish to avoid inverting the matrix $\Phix$ to compute the map $\tf K$ -- reasons may include numerical ill-condition and prohibitive computational cost for larger horizons.  We show here that one can instead implement this inverse using a particular feedback interconnection -- in particular, one can show that the following implementation:
 \begin{equation}
 \begin{array}{rcl}
 u_t & = & \sum_{\tau = 1}^T \Phi_u(t,\tau)\hat{w}_{t-\tau} \\
 \hat{x}_{t+1} & = & \sum_{\tau = 1}^{T-1}\Phi_x(t+1,\tau+1)\hat{w}_{t-\tau} \\
 \hat{w}_t & = & x_{t+1} - \hat{x}_{t+1},
 \end{array}
 \end{equation}
 leads to the desired mapping $\tf K = \Phiu\Phix^{-1}$ from $\tf x \to \tf u$.  An added benefit of this realization, which we exploit heavily in the sequel, is that any structure imposed on the system level response $\SFpair$ carries over directly to the internal structure of the controller realization, allowing a natural and transparent way to impose distributed constraints on the controller realization.

\subsection{System Level Synthesis}
The benefit of the previous parameterization is that the system behavior is now entirely defined in terms of \emph{affine} functions of the disturbance process (i.e., \eqref{eq:subspace}), and further to ensure achievability it suffices to constrain the maps $\SFpair$ to lie in an affine space.  We now show how this can be used to turn standard optimal control problems in to convex optimization problems.  The process of converting an optimal control problem to one over system responses consists of three steps:
\begin{enumerate}
\item Rewrite the problem using the vector notation \eqref{eq:definitions}.
\item Set $\tf x = \Phix \tf w$, and $\tf u = \Phiu \tf w$, and constrain $\SFpair$ to lie in the affine space \eqref{eq:subspace}.
\item Use knowledge about the disturbance process $\tf w$ to derive an appropriate objective function (see below for LQR, $\mathcal{H}_\infty$, and $\mathcal{L}_1$ optimal control).
\end{enumerate}
We now provide some worked examples of classical centralized control problem posed in the SLS framework.

\subsubsection{LQR with no driving noise and known $x_0$}\label{sec:LQRx0}
Here we let 
\begin{equation}
\tf w = \begin{bmatrix} x_0 \\ 0 \\ \vdots \\ 0\end{bmatrix}.
\end{equation}

For simplicity of notation, we assume that $u_{T}$ is included in the cost functional, i.e., that the optimal control problem that we seek to solve is given by
\begin{equation}
\begin{array}{rl}
\min_{x_t,u_t} & \sum_{t=0}^{T} x_t^\tp Q_t x_t + u_t^\tp R_tu_t\\
\st & x_{t+1} = Ax_t + Bu_t, \ t = 0, \dots, T-1, \\
& x_0 \text{ known}.
\end{array}
\label{eq:lqr_toep}
\end{equation}

Using vector notation, this problem can be written as 
\begin{equation}
\begin{array}{rl}
\min_{\tf x,\tf u} & \bignorm{\begin{bmatrix} \Qq^{\frac{1}{2}} & 0 \\ 0 & \Rr^{\frac{1}{2}}\end{bmatrix}\begin{bmatrix}\tf x \\ \tf u\end{bmatrix}}_F^2\\
\st & \tf x = Z\A\tf x + Z\B\tf u,
\end{array}
\label{eq:lqr_toep}
\end{equation}
where
\begin{equation}
\Qq := \mathrm{blkdiag}(Q_0,Q_1, \dots, Q_{T}), \ \Rr := \mathrm{blkdiag}(R_0,R_1, \dots, R_{T}).
\end{equation}

Now, we use the fact that $\tf x = \Phix \tf w = \Phix(:,0)x_0$ and $\tf u = \Phiu \tf w = \Phiu(:,0)x_0$, where we use $\tf\Phi(:,0)$ to denote the first block column of a matrix $\tf\Phi$, and Theorem \ref{thm:ft-param}, to cast this problem as an optimization over system responses:
\begin{equation}
\begin{array}{rl}
\min_{\Phix,\Phiu} & \bignorm{\begin{bmatrix} \Qq^{\frac{1}{2}} & 0 \\ 0 & \Rr^{\frac{1}{2}}\end{bmatrix}\begin{bmatrix}\Phix(:,0) \\ \Phiu(:,0)\end{bmatrix}x_0}_F^2 \\
\st & \begin{bmatrix} I - Z\A & -Z\B \end{bmatrix}\begin{bmatrix} \Phix(:,0) \\ \Phiu(:,0) \end{bmatrix} = I,
\end{array}
\label{eq:lqr_compact}
\end{equation}
which can be expressed as a convex quadratic program.

\subsubsection{LQR with driving noise and random $x_0$}
\label{sec:lqr-noise}
Here we let $\tf w \iid \mathcal{N}(0,\Sigma_w)$.  For simplicity of notation, we assume that $u_{T}$ is included in the cost functional, i.e., that the optimal control problem that we seek to solve is given by
\begin{equation}
\begin{array}{rl}
\min_{x_t,u_t} & \sum_{t=0}^{T} \E\left[x_t^\tp Q_t x_t + u_t^\tp R_tu_t\right] \\
\st & x_{t+1} = Ax_t + Bu_t + w_t.
\end{array}
\label{eq:lqr_toep}
\end{equation}

Following the same steps as the previous subsection, this problem can be rewritten as
\begin{equation}
\begin{array}{rl}
\min_{\Phix,\Phiu} & \bignorm{\begin{bmatrix} \Qq^{\frac{1}{2}} & 0 \\ 0 & \Rr^{\frac{1}{2}}\end{bmatrix}\begin{bmatrix}\Phix \\ \Phiu\end{bmatrix}\Sigma_w^{1/2}}_F^2 \\
\st & \begin{bmatrix} I - Z\A & -Z\B \end{bmatrix}\begin{bmatrix} \Phix \\ \Phiu \end{bmatrix} = I,
\end{array}
\label{eq:lqr_compact}
\end{equation}
where $\Qq$ and $\Rr$ are as defined above. Once again, this problem can be expressed as a convex quadratic program.

\subsubsection{$\mathcal{H}_\infty$ optimal control}
Consider the following $\mathcal{H}_\infty$ optimal control problem
\begin{equation}
\begin{array}{rl}
\min_{\tf x,\tf u} \max_{\norm{\tf w}_{2}\leq 1}&  \sum_{t=0}^T x_t^\top Q_tx_t + u_t^\top R_tu_t\\
\st & x_{t+1} = Ax_t + Bu_t + w_t.
\end{array}
\label{eq:hinf0}
\end{equation},

which can be rewritten in vector notation as

\begin{equation}
\begin{array}{rl}
\min_{\tf x ,\tf u} \max_{\norm{\tf w}_2\leq 1}&  \bignorm{\begin{bmatrix} \Qq^{\frac{1}{2}} & 0 \\ 0 & \Rr^{\frac{1}{2}}\end{bmatrix}\begin{bmatrix}\tf x \\ \tf u\end{bmatrix}}^2_2\\
\st & \tf x = Z\A\tf x + Z\B\tf u + \tf w.
\end{array}
\label{eq:hinf1}
\end{equation}

Now, using that $\tf x = \Phix \tf w$ and $\tf u = \Phiu \tf w$, we can rewrite our control problem in terms of system responses as follows

\begin{equation}
\begin{array}{rl}
\min_{\Phix ,\Phiu} \max_{\norm{\tf w}_2\leq 1}&  \bignorm{\begin{bmatrix} \Qq^{\frac{1}{2}} & 0 \\ 0 & \Rr^{\frac{1}{2}}\end{bmatrix}\begin{bmatrix}\Phix \\ \Phiu \end{bmatrix}\tf w}^2_2\\
\st & \begin{bmatrix} I - Z\A & -Z\B \end{bmatrix}\begin{bmatrix} \Phix \\ \Phiu \end{bmatrix} = I.
\end{array}
\label{eq:hinf1}
\end{equation}

Now, it suffices to notice that the inner maximization over $\tf w$ is simply the definition of the spectral norm of the weighted system responses, leading to the final optimization problem over $\SFpair$ 

\begin{equation}
\begin{array}{rl}
\min_{\Phix ,\Phiu} &  \bignorm{\begin{bmatrix} \Qq^{\frac{1}{2}} & 0 \\ 0 & \Rr^{\frac{1}{2}}\end{bmatrix}\begin{bmatrix}\Phix \\ \Phiu \end{bmatrix}}^2_{2\to 2}\\
\st & \begin{bmatrix} I - Z\A & -Z\B \end{bmatrix}\begin{bmatrix} \Phix \\ \Phiu \end{bmatrix} = I,
\end{array}
\label{eq:hinf}
\end{equation}

which can be expressed as a semidefinite program.
\subsubsection{$\mathcal{L}_1$ optimal control}\label{sec:L1opt}
The previous subsection optimized the worst case $\ell_2\to \ell_2$ gain of our system -- a similar derivation can be used to obtain a problem that optimizes the worst case $\ell_\infty \to \ell_\infty$ gain of our system.  In particular, we skip a step and simply begin with the following problem in vector notation
\begin{equation}
\begin{array}{rl}
\min_{\tf x ,\tf u} \max_{\norm{\tf w}_\infty\leq 1}&  \bignorm{\begin{bmatrix} \Qq^{\frac{1}{2}} & 0 \\ 0 & \Rr^{\frac{1}{2}}\end{bmatrix}\begin{bmatrix}\tf x \\ \tf u\end{bmatrix}}_\infty\\
\st & \tf x = Z\A\tf x + Z\B\tf u + \tf w.
\end{array}
\label{eq:linf1}
\end{equation}

Once again, using that $\tf x = \Phix \tf w$ and $\tf u = \Phiu \tf w$, we can rewrite our control problem in terms of system responses as follows

\begin{equation}
\begin{array}{rl}
\min_{\Phix ,\Phiu} \max_{\norm{\tf w}_\infty\leq 1}&  \bignorm{\begin{bmatrix} \Qq^{\frac{1}{2}} & 0 \\ 0 & \Rr^{\frac{1}{2}}\end{bmatrix}\begin{bmatrix}\Phix \\ \Phiu \end{bmatrix}\tf w}_\infty\\
\st & \begin{bmatrix} I - Z\A & -Z\B \end{bmatrix}\begin{bmatrix} \Phix \\ \Phiu \end{bmatrix} = I.
\end{array}
\label{eq:hinf1}
\end{equation}

Now, it suffices to notice that the inner maximization over $\tf w$ is simply the definition of the $\mathcal{L}_1$ norm (i.e., the matrix $\ell_\infty \to \ell_\infty$ induced norm) of the weighted system responses, leading to the final optimization problem over $\SFpair$ 

\begin{equation}
\begin{array}{rl}
\min_{\Phix ,\Phiu} &  \bignorm{\begin{bmatrix} \Qq^{\frac{1}{2}} & 0 \\ 0 & \Rr^{\frac{1}{2}}\end{bmatrix}\begin{bmatrix}\Phix \\ \Phiu \end{bmatrix}}_{\infty\to\infty}\\
\st & \begin{bmatrix} I - Z\A & -Z\B \end{bmatrix}\begin{bmatrix} \Phix \\ \Phiu \end{bmatrix} = I,
\end{array}
\label{eq:hinf}
\end{equation}
which can be expressed as a linear program.

\subsubsection{General Constrained Optimal Control}
The four example problems described above can be viewed as specific instantiations of a generic problem structure. They each contain the System Level Parameterization (SLP)~\eqref{eq:subspace} as a constraint, and the cost-function is some appropriately function of the system responses $\SFpair$.  Although not included in the examples above, we could also additionally impose that the system responses lie in some sets, i.e., that $\SFpair \in \Ss$, or that the state and input trajectories also lie in some sets, i.e., that  $\Phix\tf w \in \mathcal X$ and $\Phiu \tf w \in \mathcal U$ for all $\tf w \in \mathcal W$. Recent work in this direction is presented in~\cite{CheA19, DeaTMR18}.

Combining all of these elements together, we obtain a System Level Synthesis (SLS) problem:
\begin{subequations}\label{eq:SLSsyn}
\begin{align}
 \min_{\Phix, \Phiu}\quad &  g(\Phix,\Phiu)  \label{eq:cost_SLS}\\
\st \quad &  \begin{bmatrix} I - Z\A & -Z\B \end{bmatrix}\begin{bmatrix} \Phix \\ \Phiu \end{bmatrix} = I \label{eq:subspace_SLS} \\
\quad & \SFpair \in \Ss,  \\
\quad & \Phix \tf w \in \mathcal X, \, \Phiu \tf w \in \mathcal U \ \text{for all $\tf w \in \mathcal W$}. \label{eq:SLCfullmpc}
\end{align}
\end{subequations}
Clearly, if the objective functions and constraints are chosen to be convex functions, then the SLS problem is itself convex, as is the case for all of the examples given above. 

To reiterate, the SLP is the affine constraint~\eqref{eq:subspace_SLS} that ensures the system is achievable over the horizon $t=0,\dots,T$, i.e. it ensures that there exists a controller $\tf K$ such that system response $\SFpair$ described by~\eqref{eq:toep_resp} and the closed-loops maps $\tf w \rightarrow \tf x$ and $\tf w \rightarrow \tf u$ as given by~\eqref{eq:full_responses} are equivalent. System Level Constraints (SLCs) modeled by~\eqref{eq:SLCfullmpc} force the closed loop response to lie in a pre-described set $\Ss$. Typical choices for additional constraints $\Ss$ imposed on the system responses would include additional performance requirements (resulting in multi-objective control problems,) and structural constraints (as imposed by distributed constraints), or sparsity surrogate constraints (as imposed by architectural constraints).  In Section~\ref{sec:locality} we will describe a particularly useful class of SLCs that define spatiotemporal locality constraints. 
Additionally, the constraints $\Phix\tf w \in \mathcal X$ and $\Phiu \tf w \in \mathcal U$ can be used to incorporate state and input constraints.  Typical choices for constraints $\mathcal X \times \mathcal U$ on the state and input trajectories include, for example, polytopic or ellipsoidal constraints.   At this point, the connection between the finite time horizon SLS problem \eqref{eq:SLSsyn} and the subproblems solved in model predictive control should be apparent.



\subsection{A new perspective on robustness}
\label{sec:robust-time}
Thus far, we have only shown that SLS can be used to cast standard optimal control problems as optimizations over system responses.  As a preview of some of the additional benefits of working with system responses, we present here a novel robust control synthesis method that allows us to bound the performance degradation incurred by using approximate system responses that do not exactly satisfy the achievability constraints \eqref{eq:subspace}.  In later sections, we extend these robustness results to the infinite horizon setting.

Our first result is a robust variant of Theorem \ref{thm:finite_horizon}, which characterizes the behavior achieved by a controller built from approximate system responses.  Before stating the result, we let
\begin{equation}
\tf \Delta = \begin{bmatrix} \Delta^{0,0} & & & \\ \Delta^{1,1} & \Delta^{1,0} & & \\ \vdots & \ddots & \ddots & \\  \Delta^{T,T} & \cdots & \Delta^{T,1} & \Delta^{T,0 }\end{bmatrix}
\label{eq:delta}
\end{equation}
be an arbitrary block-lower-triangular matrix.

\begin{theorem}
Let $\tf \Delta$ be defined as in equation \eqref{eq:delta}, and suppose that  $\{\Phix, \Phiu\}$ satisfy
\begin{equation}
\begin{bmatrix}I - Z\A & -Z\B \end{bmatrix}\begin{bmatrix} \Phix \\ \Phiu \end{bmatrix} = I + \tf \Delta.
\label{eq:robust-param}
\end{equation}

If $(I+\Delta^{i,0})^{-1}$ exists for $i = 1,\dots,T$, then the controller $\tf K = \Phiu \Phix^{-1}$ achieves the system response
\begin{equation}
\begin{bmatrix} \tf x \\ \tf u \end{bmatrix} = \begin{bmatrix}
\Phix \\ \Phiu \end{bmatrix} (I + \tf \Delta)^{-1}\tf w
\label{eq:robust_response}
\end{equation}
\label{thm:ft-robust}
\end{theorem}
\begin{proof}
If $(I+\Delta^{i,0})^{-1}$ exists for $i=1,\dots,T$, then so does $(I + \tf \Delta)^{-1}$.  Therefore the constraint \eqref{eq:robust-param} is equivalent to
\begin{equation}
\begin{bmatrix}I - Z\A & -Z\B \end{bmatrix}\begin{bmatrix} \Phix \\ \Phiu \end{bmatrix}(I + \tf \Delta)^{-1} = I.
\end{equation}

Then noting that $\tf K = \Phiu \Phix^{-1} = \Phiu (I + \tf \Delta)^{-1} (\Phix (I + \tf \Delta)^{-1})^{-1} $, we have, by Theorem \ref{thm:finite_horizon}, that the controller $\tf K =  \Phiu \Phix^{-1} $ achieves the system responses \eqref{eq:robust_response}.
\end{proof}

We now present a simple example of how this theorem can be used to compute a controller that is robust to parametric uncertainty in the system dynamics $(\A,\B)$ -- we return to this theorem and its infinite horizon analog later in the text, at which point we present several other use-cases.

Suppose that rather than having access to the true system dynamics $(\A,\B)$, we instead only have access to estimates $(\Aahat,\Bbhat)$, and further let $\{\Phixh, \Phiuh\}$ satisfy the achievability constraints defined by $(\Aahat,\Bbhat)$.  Then simple algebra shows that the approximate system responses $\{\Phixh, \Phiuh\}$ satisfy
\begin{equation}
\begin{bmatrix} I - Z \A & -Z \B \end{bmatrix}\begin{bmatrix} \Phixh \\ \Phiuh \end{bmatrix} = I + Z\begin{bmatrix}\DA & \DB \end{bmatrix}\begin{bmatrix} \Phixh \\ \Phiuh \end{bmatrix},
\end{equation}
where $\DA:= \Aahat- \A$ and $\DB := \Bbhat - \B$.  We can then invoke Theorem \ref{thm:ft-robust} with
\[
\tf \Delta := Z\begin{bmatrix}\DA & \DB \end{bmatrix}\begin{bmatrix} \Phixh \\ \Phiuh \end{bmatrix} \]
to conclude that the controller $\hat\K = \Phiuh\Phixh^{-1}$, computed using only the system estimates $(\Aahat,\Bbhat)$, achieves the following response on the actual system $(\A,\B)$:
\[
\begin{bmatrix} \tf x \\ \tf u \end{bmatrix} = \begin{bmatrix} \Phixh \\ \Phiuh \end{bmatrix} \left(I + \tf\Delta\right)^{-1}\tf w.
\]

If we further suppose that we have known bounds on our modeling error, i.e., that $\norm{\DA}_{2\to 2} \leq \epsilon_A$ and $\norm{\DB}_{2 \to 2} \leq \epsilon_B$, then we can use this insight to formulate a robust SLS problem. Recall that for matrices, the induced $2\to 2$ norm is simply the maximum singular value. For illustrative purposes, we formulate a robust LQR problem, but the general approach is applicable to any performance metrics that have a submultiplicative property.  In particular, consider the problem presented in Section \ref{sec:lqr-noise}, and assume for simplicity that $\Sigma_w = I$.  Then, assuming we only have access to the approximate dynamics $(\Aahat,\Bbhat)$, we can pose the following robust optimization problem

\begin{equation}
\begin{array}{rl}
\displaystyle\min_{\Phix,\Phiu} \max_{\norm{\DA}\leq \epsilon_A, \ \norm{\DB}\leq \epsilon_B}& \bignorm{\begin{bmatrix} \Qq^{\frac{1}{2}} & 0 \\ 0 & \Rr^{\frac{1}{2}}\end{bmatrix}\begin{bmatrix}\Phix \\ \Phiu\end{bmatrix}\left(I + \tf \Delta \right)^{-1} }_F^2 \\
\st & \begin{bmatrix} I - Z\Aahat & -Z\Bbhat \end{bmatrix}\begin{bmatrix} \Phix \\ \Phiu \end{bmatrix} = I.
\end{array}
\label{eq:robust-nc}
\end{equation}

By considering a performance metric that maximizes over the admissible model uncertainty, we are able to provide robust performance guarantees over the uncertainty set.  However, as posed, this problem is non-convex and does not admit a computationally efficient solution.  

We can however exploit the structure of the problem to derive a tractable upper bound.  In particular, notice that the effect of the model uncertainty, as captured by the $(I + \tf \Delta)^{-1}$ term, can be isolated by applying the submultiplicative property of the spectral and frobenius norms, and the strictly-lower-block-triangular structure of $\tf \Delta$ can be further exploited to compute an upper bound on its effect.  Specifically, we have that
\begin{align*}
\bignorm{\begin{bmatrix} \Qq^{\frac{1}{2}} & 0 \\ 0 & \Rr^{\frac{1}{2}}\end{bmatrix}\begin{bmatrix}\Phix \\ \Phiu\end{bmatrix}\left(I + \tf \Delta \right)^{-1} }_F^2 &\leq \bignorm{\begin{bmatrix} \Qq^{\frac{1}{2}} & 0 \\ 0 & \Rr^{\frac{1}{2}}\end{bmatrix}\begin{bmatrix}\Phix \\ \Phiu\end{bmatrix}}_F^2\bignorm{\left(I + \tf \Delta \right)^{-1} }_{2\to 2}^2\\
&\leq \bignorm{\begin{bmatrix} \Qq^{\frac{1}{2}} & 0 \\ 0 & \Rr^{\frac{1}{2}}\end{bmatrix}\begin{bmatrix}\Phix \\ \Phiu\end{bmatrix}}_F^2\left(\sum_{t=0}^{T-1}\bignorm{\tf \Delta }_{2\to2}^{t}\right)^2,
\end{align*}
where the first inequality follows from the submultiplicative property of the spectral and Frobenius norms, and the second inequality by noticing that because $\tf \Delta$ is strictly-lower-block-triangular, the inverse can be expressed as a finite series and applying the submultiplicative property and the triangle inequality.

Finally, we note that using Proposition 3.5 from \cite{learning-lqr}, we can upper bound the spectral norm of $\tf \Delta$ as follows:
\begin{align*}
\bignorm{\tf \Delta}_{2\to 2} = \bignorm{Z\begin{bmatrix}\DA & \DB \end{bmatrix}\begin{bmatrix} \Phixh \\ \Phiuh \end{bmatrix}} _{2\to 2}\leq \sqrt{2}\bignorm{\begin{matrix} \epsilon_A \Phix \\ \epsilon_B \Phiu\end{matrix}}_{2\to 2}.
\end{align*}

Leveraging these two inequalities, we can now write a tractable (quasi-convex) optimization problem that provides an upper-bound to our original robust optimization problem \eqref{eq:robust-nc}.

\begin{equation}
\begin{array}{rl}
\displaystyle\min_{\tau}\left(\sum_{t=0}^{T-1}\tau^t\right)^2 \min_{\Phix,\Phiu} & \bignorm{\begin{bmatrix} \Qq^{\frac{1}{2}} & 0 \\ 0 & \Rr^{\frac{1}{2}}\end{bmatrix}\begin{bmatrix}\Phix \\ \Phiu\end{bmatrix} }_F^2 \\
\st & \begin{bmatrix} I - Z\Aahat & -Z\Bbhat \end{bmatrix}\begin{bmatrix} \Phix \\ \Phiu \end{bmatrix} = I \\
& \sqrt{2}\bignorm{\begin{matrix} \epsilon_A \Phix \\ \epsilon_B \Phiu\end{matrix}}_{2\to 2} \leq \tau.
\end{array}
\label{eq:robust}
\end{equation}

Thus, by leveraging the transparent mapping of system uncertainty to system performance, we were able to derive a tractable upper bound that produces a controller that is both guaranteed to be robustly stabilizing and performant.  Further, as we show in Section \ref{sec:robust}, this transparency also allows us to derive bounds on performance degradation as a function of the uncertainty sizes $\epsilon_A$ and $\epsilon_B$.  Although we defer presenting this result to the infinite horizon setting, as the computations become less cumbersome in that case, we note that analogous, albeit messier, degradation bounds can be derived for the robust optimization problem \eqref{eq:robust}.

\section{Preliminaries and Notation}\label{sec:LTIinfinite_horizon}
We now formally introduce the infinite horizon optimal control problems that we solve in this paper.

\subsection{Notation}
We use lower and upper case Latin letters such as $x$ and $A$ to denote vectors and matrices, respectively, and lower and upper case boldface Latin letters such as $\tf x$ and $\tf G$ to denote signals and transfer matrices, respectively.  We use calligraphic letters such as $\s$ to denote sets. In the interest of clarity, we work with discrete time linear time invariant systems, but unless stated otherwise, all results extend naturally to the continuous time setting.  We use standard definitions of the Hardy spaces $\mathcal{H}_2$ and $\mathcal{H}_\infty$, and denote their restriction to the set of real-rational proper transfer matrices by $\mathcal{RH}_2$ and $\RHinf$.  We use $G(i)$ to denote the $i$th spectral component of a transfer function $\tf G$, i.e., $\tf G(z) = \sum_{i=0}^{\infty} \frac{1}{z^i} G(i)$ for $| z | > 1$.    We use $\FT$ to denote the space of finite impulse response (FIR) transfer matrices with horizon $T$, i.e., $\FT := \{ \tf G \in \RHinf \, | \, \tf G = \sum_{i=0}^T\frac{1}{z^i}G(i)\}$. We frequently use the notation $\tf G \in \frac{1}{z}\RHinf$ to denote that $\tf G$ is strictly proper. Informally this can be parsed as $z\tf G \in \RHinf$. Finally, we use $\tf G(T_1:T_2)$ to denote the projection of $\tf G$ onto $\mathcal{F}_{T_1}^\perp \cap \mathcal{F}_{T_2}$, i.e., $\tf G(T_1:T_2) = \sum_{i=T_1}^{T_2}\frac{1}{z^i}G(i)$. 

\subsection{System Model}
We consider discrete time linear time invariant (LTI) systems of the form
\begin{subequations} \label{eq:sys-dynamics}
\begin{align}
x[t+1] &= A x[t] + B_1 w[t] + B_2 u[t] \label{eq:sys_x} \\
\bar{z}[t] &= C_1 x[t] + D_{11} w[t] + D_{12} u[t] \label{eq:sys_z} \\
y[t] &= C_2 x[t] + D_{21} w[t] + D_{22} u[t] \label{eq:sys_y}
\end{align}
\end{subequations}
where $x$, $u$, $w$, $y$, $\bar{z}$ are the state vector, control action, external disturbance, measurement, and regulated output, respectively. 
Equation \eqref{eq:sys-dynamics} can be written in state space form as
\begin{equation}
\tf P = \left[ \begin{array}{c|cc} A & B_1 & B_2 \\ \hline C_1 & D_{11} & D_{12} \\ C_2 & D_{21} & D_{22} \end{array} \right] = \begin{bmatrix} \tf P_{11} & \tf P_{12} \\ \tf P_{21} & \tf P_{22} \end{bmatrix} \nonumber
\end{equation}
where $\tf P_{ij} = C_i(zI-A)^{-1}B_j + D_{ij}$. We refer to $\tf P$ as the open loop plant model. 

Consider a dynamic output feedback control law $\tf u = \tf K \tf y$. The controller $\tf K$ is assumed to have the state space realization
\begin{subequations} \label{eq:Kss}
\begin{align}
\xi[t+1] &= A_k \xi[t] + B_k y[t] \label{eq:Kss1} \\
u[t] &= C_k \xi[t] + D_k y[t], \label{eq:Kss2}
\end{align}
\end{subequations}
where $\xi$ is the internal state of the controller. We have $\tf K = C_k(zI-A_k)^{-1}B_k + D_k$. A schematic diagram of the interconnection of the plant $\tf P$ and the controller $\tf K$ is shown in Figure \ref{fig:pk}.

\begin{figure}[h]
      \centering
      \includegraphics[width=0.25\textwidth]{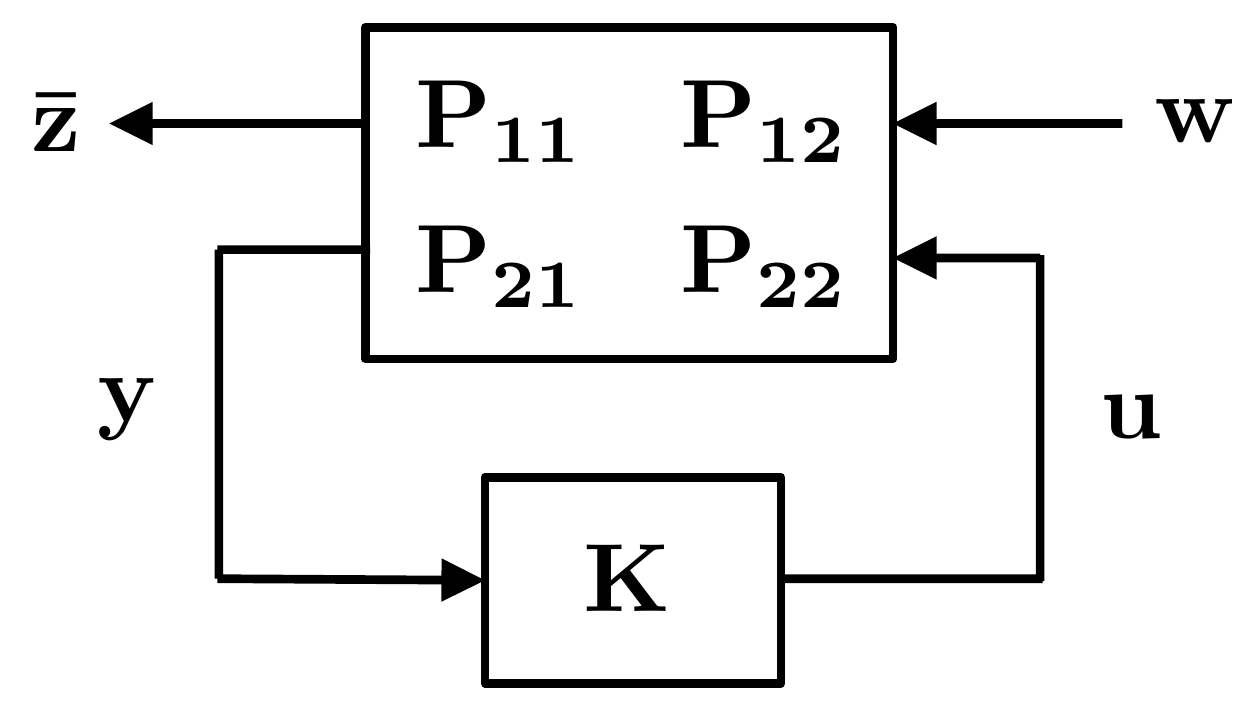}
      \caption{Interconnection of the plant $\tf P$ and controller $\tf K$.}
      \label{fig:pk}
\end{figure}

The following assumptions are made throughout the paper.
\begin{assumption}
The interconnection in Figure \ref{fig:pk} is well-posed -- the matrix $(I-D_{22}D_k)$ is invertible.
\end{assumption}
\begin{assumption}\label{as:stab}
Both the plant and the controller realizations are stabilizable and detectable; i.e., $(A,B_2)$ and $(A_k, B_k)$ are stabilizable, and $(A,C_2)$ and $(A_k,C_k)$ are detectable.
\end{assumption}
Note that assumption~\ref{es:stab}
\rev{The goal of the optimal control problem is to find a controller $\tf K$ to stabilize the plant $\tf P$ and minimize a suitably chosen norm\footnote{\rev{Typical choices for the norm include $\mathcal{H}_2$ and $\mathcal{H}_\infty$.}} of the closed loop transfer matrix from external disturbance $\tf w$ to regulated output $ \tf{\bar{z}}$. This leads to the following centralized optimal control formulation:}
\rev{
\begin{eqnarray}
\underset{\tf K}{\text{minimize  }} &&|| \tf P_{11} + \tf P_{12} \tf K(I- \tf P_{22} \tf K)^{-1} \tf P_{21} || \nonumber\\
\text{subject to } && \tf K \text{ internally stabilizes } \tf P. \label{eq:centralized-add}
\end{eqnarray}
}
\vspace{-5mm}

\subsection{Youla Parameterization} \label{sec:move_Youla}
\revsecond{
A common technique to solve the optimal control problem \eqref{eq:centralized-add} is via the Youla parameterization, which is based on a doubly co-prime factorization of the plant, defined as follows.
\begin{definition}
A collection of stable transfer matrices, $\tf U_r$, $\tf V_r$, $\tf X_r$, $\tf Y_r$, $\tf U_l$, $\tf V_l$, $\tf X_l$, $\tf Y_l$ $\in \RHinf$ defines a doubly co-prime factorization of $\tf P_{22}$ if $\tf P_{22} = \tf V_r \tf U_r^{-1} = \tf U_l^{-1} \tf V_l$ and
\begin{equation}
\begin{bmatrix} \tf X_l & \tf{-Y}_l \\ \tf{-V}_l & \tf U_l \end{bmatrix} \begin{bmatrix} \tf U_r & \tf Y_r \\ \tf V_r & \tf X_r \end{bmatrix} = I. \nonumber
\end{equation}
\end{definition}
Such doubly co-prime factorizations can always be computed if $\tf P_{22}$ is stabilizable and detectable \cite{Zhou1996robust}.

Let $\tf Q$ be the Youla parameter. From \cite{Zhou1996robust}, problem \eqref{eq:centralized-add} can be reformulated in terms of the Youla parameter as 
\begin{eqnarray}
\underset{\tf Q}{\text{minimize  }} &&|| \tf T_{11} + \tf T_{12} \tf Q \tf T_{21}|| \nonumber\\
\text{subject to } && \tf Q \in \RHinf \label{eq:trad_youla-add}
\end{eqnarray}
with $\tf T_{11} = \tf P_{11} + \tf P_{12} \tf Y_r \tf U_l \tf P_{21}$, $\tf T_{12} = - \tf P_{12} \tf U_r$, and $\tf T_{21} = \tf U_l \tf P_{21}$. 

The benefit of optimizing over the Youla parameter $\tf Q$, rather than the controller $\tf K$, is that \eqref{eq:trad_youla-add} is convex with respect to the Youla parameter. One can then incorporate various convex design specifications \cite{Boyd_closed} in \eqref{eq:trad_youla-add} to customize the controller synthesis task.
Once the optimal Youla parameter $\tf Q$, or a suitable approximation thereof, is found in \eqref{eq:trad_youla-add}, we reconstruct the controller by setting $\tf K =  (\tf Y_r - \tf U_r \tf Q)(\tf X_r - \tf V_r \tf Q)^{-1}$. 

}

\subsection{Structured Controller Synthesis and QI} \label{sec:structured}
\rev{We now move our discussion to the distributed optimal control problem.}
We follow the paradigm adopted in \cite{2006_Rotkowitz_QI_TAC,2012_Lessard_two_player,2010_Shah_H2_poset,2013_Lamperski_H2,2013_Lessard_structure,2013_Scherer_Hinf,2014_Lessard_Hinf,2014_Matni_Hinf,2014_Tanaka_Triangular}, and focus on information asymmetry introduced by delays in the communication network -- this is a reasonable modeling assumption when one has dedicated physical communication channels (e.g., fiber optic channels), but may not be valid under wireless settings.  In the references cited above, locally acquired measurements are exchanged between sub-controllers subject to delays imposed by the communication network,\footnote{Note that this delay may range from 0, modeling instantaneous communication between sub-controllers, to infinite, modeling no communication between sub-controllers.} which manifest as subspace constraints on the controller itself.\rev{\footnote{\rev{For continuous time systems, the delays can be encoded via subspaces that may reside within $\mathcal{H}_\infty$ as opposed $\RHinf$.}}}

\rev{Let $\mathcal{C}$ be a subspace enforcing the information sharing constraints imposed on the controller $\tf K$. A distributed optimal control problem can then be formulated as \cite{2006_Rotkowitz_QI_TAC,2011_QIN, 2014_Lessard_convexity, 2014_Sabau_QI}:}
\begin{equation}
\begin{array}{rl}
\underset{\tf K}{\text{minimize  }} &\| \tf P_{11} + \tf P_{12} \tf K(I- \tf P_{22} \tf K)^{-1} \tf P_{21} \| \\
\text{subject to } & \tf K \text{ internally stabilizes } \tf P, \,\, \tf K \in \mathcal{C}.
\end{array}
\label{eq:decentralized}
\end{equation}

\rev{A summary of the main results from the distributed optimal control literature \cite{2006_Rotkowitz_QI_TAC,2012_Lessard_two_player,2010_Shah_H2_poset,2013_Lamperski_H2,2013_Lessard_structure,2013_Scherer_Hinf,2014_Lessard_Hinf,2014_Matni_Hinf,2014_Tanaka_Triangular} can be given as follows:} if the subspace $\mathcal{C}$ is quadratically invariant with respect to $\tf P_{22}$ \revsecond{(i.e., $\tf K \tf P_{22} \tf K \in \mathcal{C}, \,\, \forall \tf K \in \mathcal{C}$)} \cite{2006_Rotkowitz_QI_TAC}, then the set of all stabilizing controllers lying in subspace $\mathcal{C}$ can be parameterized by those stable transfer matrices $\tf Q \in \RHinf$ satisfying $\mathfrak{M}(\tf Q)\in \mathcal{C}$, \rev{for $\mathfrak{M}(\tf Q) := \tf K(I- \tf P_{22} \tf K)^{-1} = (\tf Y_r - \tf U_r \tf Q) \tf U_l$.\footnote{\rev{By definition, we have $\tf P_{22} = \tf V_r \tf U_r^{-1} = \tf U_l^{-1} \tf V_l$. This implies that the transfer matrices $\tf U_r$ and $\tf U_l$ are both invertible. Therefore, $\mathfrak{M}$ is an invertible affine map of the Youla parameter $\tf Q$.}}}
Further, these conditions can be viewed as tight, in the sense that quadratic invariance is also a necessary condition \cite{2011_QIN, 2014_Lessard_convexity} for a subspace constraint $\mathcal{C}$ on the controller $\tf K$ to be enforced via a convex constraint on the Youla parameter $\tf Q$.

This allows the optimal control problem \eqref{eq:decentralized} to be recast as the following convex model matching problem:
\begin{equation} 
\begin{array}{rl}
\underset{\tf Q}{\text{minimize}} & \| \tf T_{11} + \tf T_{12} \tf Q \tf T_{21}\| \\
\text{subject to} & \tf Q \in \RHinf, \,\, \mathfrak{M}(\tf Q) \in \mathcal{C}.
\end{array}
\label{eq:trad_youla2}
\end{equation}

\subsection{\revsecond{QI imposes limitations on controller sparsity}} \label{sec:limitation}
\revsecond{When working with large-scale systems, it is natural to impose that sub-controllers only collect information from a local subset of all other sub-controllers. This can be enforced by setting the subspace constraint $\mathcal{C}$ in problem \eqref{eq:decentralized} to encode a suitable sparsity pattern $\tf K_{ij} = 0$,\footnote{$\tf K_{ij}$ denotes the $(i,j)$-entry of the transfer matrix $\tf K$.} for some $i, j$.  However, if the plant $\tf P_{22}$ is dense (i.e., if the underlying system is strongly connected), which may occur even if the system matrices $(A,B_2,C_2)$ are sparse, then \emph{any} such sparsity constraint is not QI condition with respect to the plant $\tf P_{22}$: this follows immediately from the algebraic definition of QI $\tf K \tf P_{22} \tf K \in \mathcal{C}, \,\, \forall \tf K \in \mathcal{C}$.  As QI is a necessary and sufficient condition for the subspace constraint $\tf K \in \mathcal{C}$ to be enforced via a convex constraint on the Youla parameter $\tf Q$, we conclude that for strongly connected systems, any sparsity constraint imposed on the controller $\tf K$ can only be enforced via a non-convex constraint on Youla parameter.  A major motivation for SLS was to circumvent this limitation of the QI framework.}

\section{State-Feedback System Level Synthesis}\label{sec:state_feedback_SLS}
In this section, we extend the results of Section \ref{sec:sys_resp_intro} to the infinite horizon setting, and propose a novel parameterization of internally stabilizing state-feedback controllers centered around \emph{system responses}, which are defined by the closed loop maps from process disturbances to state and control action. We show that for a given system, the set of stable \revsecond{closed loop system responses that are achievable by an internally stabilizing LTI controller} is an affine subspace of $\RHinf$, and that the corresponding internally stabilizing controller achieving the desired system response admits a particularly simple and transparent realization. 

\revsecond{Consider a state feedback model given by}
\begin{equation}
\tf P = \left[ \begin{array}{c|cc} A & B_1 & B_2 \\ \hline C_1 & D_{11} & D_{12} \\ I & 0 & 0 \end{array} \right]. \label{eq:sfplant}
\end{equation}

In this case, the state dynamics described in \eqref{eq:sys-dynamics} reduce to
\begin{equation}
x[t+1] = Ax[t] +B_2u[t] + B_1 w[t].
\label{eq:LTI}
\end{equation}

The $z$-transform of these state dynamics is then given by
\begin{equation}
(zI - A) \tf x = B_2 \tf u + \ttf{\delta_x}, \label{eq:z_state}
\end{equation}
where we let $ \ttf{\delta_x} := B_1\tf w$ denote the disturbance affecting the state.

{
%

Consider now the (dynamic) state-feedback control policy $\tf u = \tf K \tf x$ and substitute this into~\eqref{eq:z_state}. After a simple rearrangement the maps $\Phix : \ttf{\delta_x} \rightarrow \tf x$ and $\Phiu : \ttf{\delta_x} \rightarrow \tf u$ are seen to be
\begin{equation}
\begin{array}{rcl}
{\tf{\Phi}}_x &=&(zI - A - B_2\tf K)^{-1} \\
{\tf{\Phi}}_u &=& \tf K(zI-A-B_2 \tf K)^{-1}.
\end{array} \label{eq:full_responses_inf}
\end{equation}
The resemblance between~\eqref{eq:full_responses} and~\eqref{eq:full_responses_inf} is immediate. 

\begin{figure}[ht!]
      \centering
      \includegraphics[width=0.4\textwidth]{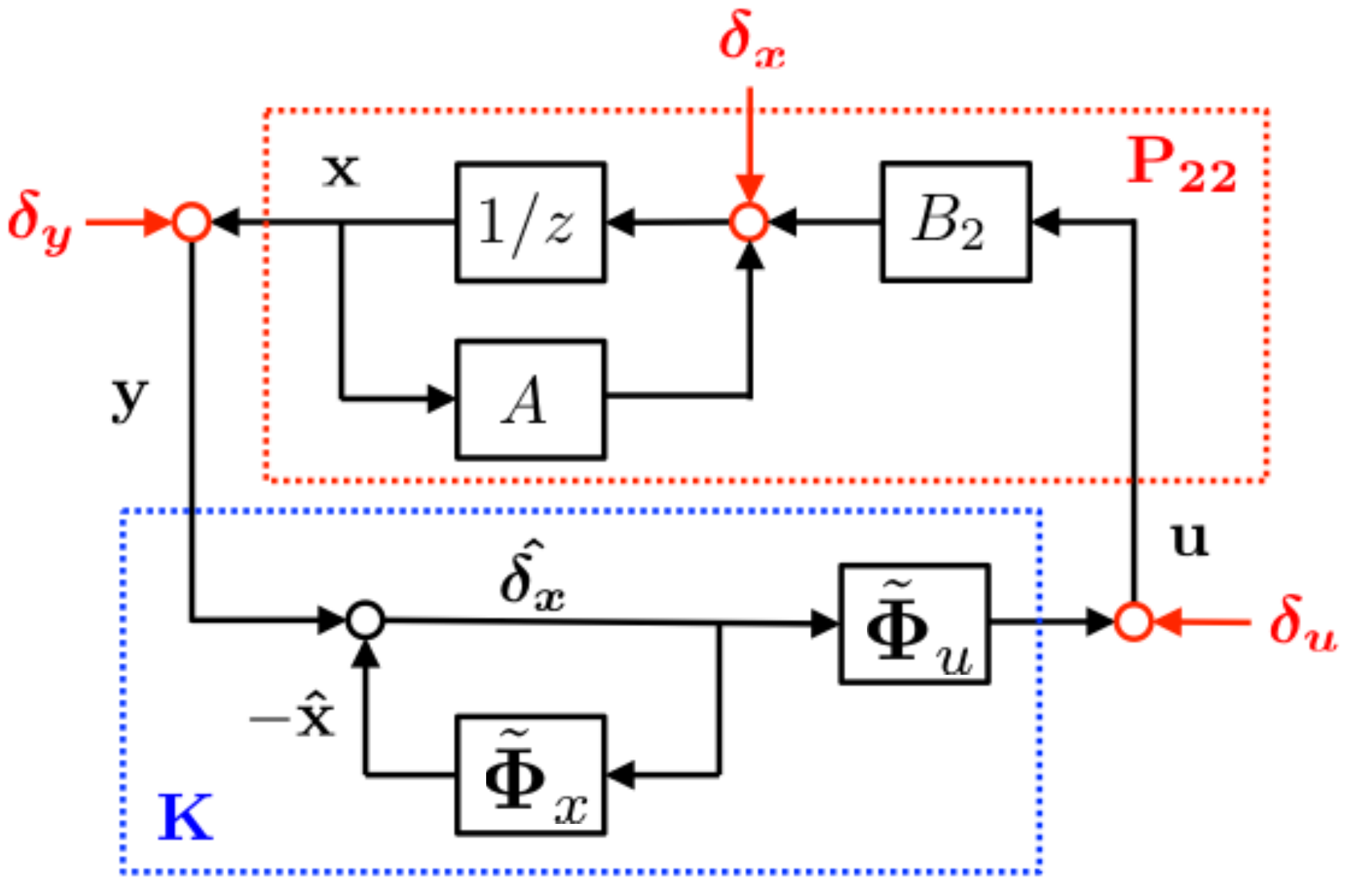}
      \caption{The proposed state feedback controller structure, with $\tilde{\Phix} = I - z \Phix $ and $\tilde{\Phiu} = z \Phiu$.}
      \label{fig:sf_controller}
\end{figure}

We now provide an infinite horizon analogue to Theorem~\ref{thm:finite_horizon}, that ensures that the system response $\SFpair$ are achievable by the proposed state-feedback policy $\tf u = \tf K \tf x$.  As we are now operating over an infinite horizon, we must take care to provide an internally stable realization for the controller.
\begin{theorem}\label{thm:param}
Consider the LTI system~\eqref{eq:LTI}, evolving under a dynamic state-feedback control policy $\tf u = \tf K \tf x$. The following statements are true:
\begin{enumerate}
\item The affine subspace defined by
\begin{equation}\label{eq:affine_cons}
 \begin{bmatrix} zI - A & -B_2 \end{bmatrix}\begin{bmatrix} \Phix \\ \Phiu \end{bmatrix} = I, \quad \Phix, \Phiu \in \frac{1}{z}\RHinf
\end{equation}
parameterizes all system responses from $\ttf{\delta_x}$ to $(\tf x,\tf u)$ as defined in~\eqref{eq:full_responses_inf}, achievable by an internally stabilizing state feedback controller $\tf K$.
\item For any transfer matrices $\SFpair$ satisfying~\eqref{eq:affine_cons}, the controller $\tf K = \Phiu \Phix^{-1}$, implemented as shown in Figure \ref{fig:sf_controller} via the following equations:
\begin{equation*}
\begin{array}{rcl}
\tf u &=& z\Phiu \ttf{\hat{\delta}_x} \\
\ttf{\hat{\delta}_x} &=& \tf x + (I-z\Phix)\ttf{\hat{\delta}_x},
\end{array}
\end{equation*}
 is internally stabilizing and achieves the desired system response~\eqref{eq:full_responses_inf}.
\end{enumerate}
\end{theorem}

The similarities between the finite horizon case given in Theorem~\ref{thm:ft-param} and those presented in Theorem~\ref{thm:param} are striking -- this should not be surprising, as the lower-block-triangular operators we previously considered can be viewed as truncations of the matrix representations of the transfer matrices $\SFpair$.

At this point the observant reader may have noticed that Theorem~\ref{thm:param} circumvents the standard assumptions of \emph{stabilizability} of the pair $(A,B_2)$. However this not the case: 
\begin{lemma}
The constraint~\eqref{eq:affine_cons} is feasible if and only if $(A,B_2)$ is stabilizable.
\label{lem:stabilizable}
\end{lemma}
\begin{proof}
We first show that the stabilizability of $(A,B_2)$ implies that there exist transfer matrices $\Phix, \Phiu\in\frac{1}{z}\RHinf$ satisfying equation \eqref{eq:affine_cons}. From the definition of stabilizability, there exists a matrix $F$ such that $A+B_2 F$ is a stable matrix. Substituting the state feedback control law $u = F x$ into \eqref{eq:z_state}, we have $\tf x = (zI-A-B_2 F)^{-1}  \ttf{\delta_x}$ and $\tf u = F(zI-A-B_2 F)^{-1}  \ttf{\delta_x}$. 
The system response is given by $\Phix = (zI-A-B_2 F)^{-1}$ and $\Phiu = F(zI-A-B_2 F)^{-1}$, which lie in $\frac{1}{z}\RHinf$ and are a solution to \eqref{eq:affine_cons}.

For the opposite direction, we note that $\Phix, \Phiu \in \RHinf$ implies that these transfer matrices do not have poles outside the unit circle $| z | \geq 1$. From \eqref{eq:z_state}, we further observe that $\begin{bmatrix} zI - A & -B_2 \end{bmatrix}$ is right invertible in the region where $\Phix$ and $\Phiu$ do not have poles, with $\begin{bmatrix} \Phix^\top & \Phiu^\top \end{bmatrix}^\top$ being its right inverse. This then implies that $\begin{bmatrix} zI - A & -B_2 \end{bmatrix}$ has full row rank for all $| z | \geq 1$. This is equivalent to the PBH test \cite{ZDGBook} for stabilizability, proving the claim.
\end{proof}
%
%

This lemma therefore provides a sanity check: if the system is not stabilizable, then there do not exist any stable system responses, and hence the constraints are infeasible.  We now prove Theorem \ref{thm:param}.

\begin{proof}[Theorem \ref{thm:param}]
This proof is nearly identical to that of Theorem \ref{thm:ft-param}, but with the added complication of ensuring internal stability of the resulting controller.

Proof of 1.: Let $\tf K$ be any internally stabilizing controller, and set $\tf u = \tf K \tf x$.  Then as argued above, we have that
\begin{equation}
\begin{array}{rcl}
\tf x &=& (zI-(A+B_2\tf K))^{-1}\ttf{\delta_x} \\
\tf u &=& \tf K(zI-(A+B_2\tf K))^{-1}\ttf{\delta_x}.
\end{array}
\end{equation}
It is then easily seen that
\begin{equation}
\begin{bmatrix} zI - A \ -B_2\end{bmatrix}\begin{bmatrix}(zI-(A+B_2\tf K))^{-1} \\ \tf K (zI-(A+B_2\tf K))^{-1} \end{bmatrix} = (zI - A -B_2\tf K)(zI-(A+B_2\tf K))^{-1} = I.
\end{equation}

Proof of 2.: First notice that the affine constraint \eqref{eq:affine_cons} implies that $\Phix(1) = I$, and hence $\Phix^{-1}$ exists.  We begin by showing that the controller $\tf K =\Phiu\Phix^{-1}$ achieves the desired response, and then show that the implementation suggested in Figure \ref{fig:sf_controller} is internally stabilizing.  Let $\tf K = \Phiu\Phix^{-1}$ and notice that
\begin{equation}
\tf x = (zI-(A+B_2\Phiu \Phix^{-1}))^{-1}\ttf{\delta_x}.
\end{equation}
But we then have that
\begin{equation}
(zI-(A+B_2\Phiu \Phix^{-1}))^{-1} = ((zI-A)\Phix - B_2 \Phiu)\Phix^{-1})^{-1} = \Phix ((zI-A)\tf \Phix - B_2 \Phiu)^{-1}  = \Phix
\end{equation}
where the last equality follows form the fact that $\SFpair$ satisfy \eqref{eq:affine_cons}.  Similarly we have that
\begin{equation}
\tf u = \Phiu \Phix^{-1}\tf x = \Phiu \Phix^{-1} \Phix \ttf{\delta_x} = \Phiu \ttf{\delta_x},
\end{equation}
where the second equality follows from the fact that $\tf x = \Phix \ttf{\delta_x}$.

Finally, we turn to proving the internal stability of the proposed controller realization shown in Figure \ref{fig:sf_controller}, where here $\Phixtil = I - z \Phix$ and $\Phiutil= z \Phiu$.  We first note that from Figure \ref{fig:sf_controller}, we can express the state feedback controller $\tf K$ as $\tf K = \Phiutil (I - {\Phixtil})^{-1} = (z \Phiu)(z \Phix)^{-1} = \Phiu \Phix^{-1}$. 

It can be checked that $z\Phixtil,\Phiutil \in \RHinf$, and hence the internal feedback loop between $\ttf{\hat{\delta}_x}$ and the reference state trajectory $\tf{\hat{x}}$ is well defined.  As is standard, we introduce external perturbations $\delta_x, \delta_y$, and $\delta_u$ into the system and note that the perturbations entering other links of the block diagram can be expressed as a combination of $(\ttf{\delta_x}, \ttf{\delta_y}, \ttf{\delta_u})$ being acted upon by some stable transfer matrices.\footnote{The matrix $A$ may define an unstable system, but viewed as an element of $\mathcal{F}_0$, defines a stable (FIR) transfer matrix.}  Hence the standard definition of internal stability applies, and we can use a bounded-input bounded-output argument (e.g., Lemma $5.3$ in \cite{Zhou1996robust}) to conclude that it suffices to check the stability of the nine closed loop transfer matrices from perturbations $(\ttf{\delta_x}, \ttf{\delta_y}, \ttf{\delta_u})$ to the internal variables $(\tf x, \tf u, \ttf{\hat{\delta_x}})$ to determine the internal stability of the structure as a whole.

Routine calculations show that the closed loop transfer matrices from $(\ttf{\delta_x}, \ttf{\delta_y}, \ttf{\delta_u})$ to $(\tf x, \tf u, \ttf{\hat{\delta}_x})$ are given by
\begin{equation}
\begin{bmatrix} \tf x \\ \tf u \\ \ttf{\hat{\delta}_x} \end{bmatrix} = \begin{bmatrix} \Phix & -\Phixtil - \Phix A & \Phix B_2 \\ \Phiu & \Phiutil - \Phiu A & I + \Phiu B_2 \\ \frac{1}{z}I & I-\frac{1}{z}A & \frac{1}{z}B_2 \end{bmatrix} \begin{bmatrix} \ttf{\delta_x} \\ \ttf{\delta_y} \\ \ttf{\delta_u} \end{bmatrix}. \label{eq:sf_cl}
\end{equation}
As all nine transfer matrices in \eqref{eq:sf_cl} are stable, the implementation in Figure \ref{fig:sf_controller} is internally stable. Furthermore, the desired system response $\{\Phix, \Phiu\}$, from $\ttf{\delta_x}$ to $(\tf x, \tf u)$, is achieved.
\end{proof}

With Theorem~\ref{thm:param} in hand we can now proceed as in the finite horizon case and define the System Level Synthesis problem. Let us assume that a suitable cost functional $g$ has been chosen -- suitable examples include $\mathcal{L}_1, \hinf$, or $\mathcal{H}_2$ objective functions \cite{Zhou1996robust}. The general form of an SLS problem takes the form
\begin{equation}
\begin{array}{rl}
 \min_{\Phix, \Phiu}&  g(\Phix,\Phiu) \\
\st &  \begin{bmatrix} zI-A & - B_2 \end{bmatrix}\begin{bmatrix} \Phix \\ \Phiu \end{bmatrix} = I \\
& \Phix, \Phiu \in \frac{1}{z}\mathcal{R}\hinf \cap \Ss  .
\end{array}\label{eq:SLSsyn_inf}
\end{equation}
If the objective function $g(\cdot)$ and constraint set $\mathcal{S}$ are convex, then optimization problem \eqref{eq:SLSsyn_inf} is a convex but infinite-dimensional problem.  The constraint set $\Ss$ can be used to capture additional constraints to be imposed on the the system response, which in turn translate to constraints on the internal realization of the controller as illustrated in Figure \ref{fig:sf_controller}.  We note here that if the system is controllable, then a useful SLC to add is that the system responses are FIR, i.e., that $\SFpair \in \fir$ for some suitably chosen horizon $T$.  In that case, the optimization problem \eqref{eq:SLSsyn_inf} immediately becomes finite dimensional (assuming that the remaining constraints admit finite dimensional representations), and further, the controller implementation shown in Figure \ref{fig:sf_controller} reduces to an interconnection of FIR filter banks, greatly simplifying controller deployment.  Finally, we show in Section \ref{sec:robust} that, modulo some minor technical assumptions and slight modifications to problem \eqref{eq:SLSsyn_inf}, taking a sufficiently long FIR horizon $T$ results in near optimal performance.
}

%

\subsection{Distributed Control, Locality, and Scalability}
The key insight that we exploit in this section is that distributed constraints (in the form of sparsity and delay constraints) can be imposed on the internal blocks of the controller realization shown in Figure \ref{fig:sf_controller} via subspace constraints on the system responses.  This idea of imposing structure on the internal controller realization (i.e., on $\SFpair$), instead of on the controller itself (i.e., on the map $\tf K = \Phiu\Phix^{-1}$), is what allows us to circumvent the limitations of existing results in the distributed control literature.  This insight, coupled with the observation that many optimal control problems of interest can be decomposed into parallel subproblems, allows us to scale out distributed control methods to the large-scale setting with ease.

Before delving into the technical details, we present a simple example that illustrates some of the shortcomings in the distributed controls literature that we are able to overcome.

\begin{example}
\label{ex:motivating}
\rev{
Consider the optimal control problem:
\begin{equation}
\begin{array}{rl}
\minimize_{~u} & \lim_{T\to \infty}\frac{1}{T}\sum_{t=0}^T \mathbb{E}\|x[t]\|_2^2 \\
\text{subject to} & x[t+1]=Ax[t]+u[t]+w[t],
\end{array}
\label{eq:opt_ctrl_example}
\end{equation}
with disturbance $w[t]\overset{\mathrm{i.i.d}}{\sim{}}\mathcal{N}(0,I)$.  We assume full state-feedback, i.e., the control action at time $t$ can be expressed as $u[t]=f(x[0:t])$ for some function $f$.  An optimal control policy $u^\star$ for this LQR problem is easily seen to be given by $u^\star[t]=-Ax[t]$. } 

\rev{Further suppose that the state matrix $A$ is sparse and let its support define the adjacency matrix of a graph $\mathcal{G}$ for which we identify the $i$th node with the corresponding state/control pair $(x_i,u_i)$.  
In this case, we have that the optimal control policy $u^\star$ can be implemented in a \emph{localized} manner. In particular, in order to implement the state feedback policy for the $i$th actuator $u_i$, only those states $x_j$ for which $A_{ij}\neq 0$ need to be collected -- thus only those states corresponding to immediate neighbors of node $i$ in the graph $\mathcal{G}$, i.e., only \emph{local} states, need to be collected to compute the corresponding control action, leading to a localized implementation.  As we discuss below, the idea of locality is essential to allowing controller synthesis and implementation to scale to arbitrarily large systems, and hence such a structured controller is desirable.  }

\rev{Now suppose that we naively attempt to solve optimal control problem \eqref{eq:opt_ctrl_example} by converting it to its equivalent $\mathcal{H}_2$ model matching problem \rev{\eqref{eq:decentralized}} and constraining the controller $\tf K$ to have the same support as $A$, i.e., $\tf K = \sum_{t=0}^\infty \frac{1}{z^t}K[t]$, $\supp{K[t]} \subset \supp {A}$.} 
\revsecond{If the graph $\mathcal{G}$ is strongly connected, then \emph{any} sparsity constraint in the form of $\tf{K}_{ij} = 0$ is not QI with respect to the plant $\tf P_{22} = (zI-A)^{-1}$.  To see this, note that if the graph $\mathcal{G}$ is strongly connected, then $\tf P_{22}$ is a dense transfer function: it then follows immediately that any subspace $\mathcal{C}$ enforcing sparsity constraints on $\tf K$ fails to satisfy $\tf K \tf P_{22} \tf K \in \mathcal{C}, \,\, \forall \tf K \in \mathcal{C}$, and hence is not QI with respect to $P_{22}$.
The results of \cite{2014_Lessard_convexity} further allow us to conclude that computing such a structured controller can never be done using convex programming when using the Youla parameterization.}

\rev{In contrast, in the case of a full control ($B_2=I$) problem, the condition \eqref{eq:affine_cons} simplifies to $(zI-A)\Phix - \Phiu = I$, $\Phix, \Phiu \in \frac{1}{z}\RHinf$.   Again, suppose that we wish to synthesize an optimal controller that has a communication topology given by the support of $A$ -- from the above implementation, it suffices to constrain the support of transfer matrices $\Phix$ and $\Phiu$ to be a subset of that of $A$. It can be checked that $\Phix = \frac{1}{z}I$, and $\Phiu = -\frac{1}{z}A$ satisfy the above constraints, and recover the globally optimal controller $\tf K = -A$. } 
\end{example}

\subsection{Spatial and Temporal Locality}\label{sec:locality}
As mentioned above, the constraint set $\Ss$ can be used to impose distributed constraints on the controller realization -- in this section we describe a useful such class of constraints that enforce \emph{locality} in the system responses.  In particular, in addition to the FIR constraints described above, which can be interpreted as localizing the system responses in time, we also ask that the system response be spatially localized.

The notion of temporal locality ensures that the closed-loop disturbance response has a finite-impulse response. Furthermore, it ensures that the resulting System Level Synthesis problem is a finite-dimensional mathematical program. Spatial locality ensures that when a disturbance enters the system, its effect is felt only in a local neighborhood of the origin. Spatiotemporal locality, as well as being desirable from a performance perspective, will also make computation tractable. We begin by introducing spatial locality.

The network structure of the system is naturally modeled by a graph $\cG(V,E)$. In this setting the vertices $v_i\in V$ denote subsystems and the unweighted edges, $e_{ij}:=(v_i,v_j)\in E\subset V\times V$ with $v_i,v_j \in V$, indicate that subsystem $j$ affects subsystem $i$. As a distance metric on $\cG$ we use the length of the shortest path from $v_i$ to $v_j$ which we denote by $\dist(v_i\rightarrow v_j)$. 

\begin{defn}
For the unweighted graph $\cG$, the $d$-outgoing set of subsystem $i$ is  $\out_{v_i}(d):= \{v_j~|~ \dist(v_i \rightarrow v_j)\le d\}$. Analogously, the $d$-incoming set of subsystem $i$ is $\inset_{v_i}:=\{v_j~|~ \dist(v_i \rightarrow v_j)\le d\}$.
\end{defn}
The definition of locality we use is built directly on the notion of the incoming and outgoing sets. Informally, the system is $d$-localizable if every subsystem in the network can synthesize a controller using a plant model which is limited to subsystems contained in its $d$-outgoing set and implemented using signals coming only from subsystems in its $d$-incoming set. Note that we choose $d$ to be significantly smaller than the longest path in $\cG$.

\begin{defn}
Given the state dynamics~\eqref{eq:LTI} and the map $x = \Phix \tf w$. Denote the transfer matrix describing the perturbation from signal $w_j$ at subsystem $j$ to the state of subsystem $i$, $x_i$, by $[\Phix]_{ij}$. Then the map $\Phix$ is $d$-localizable if and only if for every subsystem $j$, $[\Phix]_{ij}= 0$ for all $i \notin \out_j(d)$. The definition for $d$-localizability of $\Phiu$ where $u = \Phiu \tf w$, follows in the same manner but with perturbations to control action $u_i$ at subsystem $i$. 
\end{defn}

\begin{defn}
The subspace $\cL_d$ is said to be a $d$-localized constraint if it constrains $\Phix$ to be $d$-localized and $\Phiu$ to be $(d+1)$-localized.
\end{defn}

We now introduce temporal-locality: The set $\fir$ is a convex \emph{System Level Constraint} (SLC) that  ensures the system response $\SFpair$ has a finite impulse response of horizon  length $T$. Consider a transfer matrix $\tf G$, then:
\begin{equation*}
\fir:= \left\{\tf G \in \mathcal{R}\hinf ~|~\tf G = \sum_{i=0}^T\frac{1}{z^i}\tf G[i] \right\},
\end{equation*}
We make an important observation regarding computation here. When the SLC set $\Ss$ intersects $\fir$, the resulting SLS problem~\eqref{eq:SLSsyn_inf} becomes finite-dimensional. We now have all the ingredients necessary to define a broad class of locality constraints.

\begin{defn}
The system given by the state equations~\eqref{eq:LTI} is $(d,T)$-localizable if~\eqref{eq:SLSsyn_inf} is feasible with $\Ss = \cL_d \cap \fir $ for some ($d,T)\in \tf Z_+ \times \tf Z_+$.
\end{defn}

The final class we consider describe  delay constraints. We allow for three types of delays in the network; actuation delay, communication delay, and sensing delay denoted by $k_a, k_c$, and $k_s$ respectively. Actuation delay describes the time taken for control action $u_i$ to affect state $x_i$ of it's subsystem. Communication delay is the amount of time it takes for a subsystem-controller to transmit a signal to its immediate neighbors. Finally, the sensing delay is the time taken for a subsystem-controller $u_i$ to access its state variable $x_i$. For brevity it is assumed that $(k_a,k_c,k_s)$ are integers and normalized with respect to the sampling period of the system.

We are now in position where we can combine the $(d,T)$-localization subspace constraint and the communication subspace constraints as constraints on the system response $\{\Phix, \Phiu\}$. The most transparent way to implement these constraints is to enforce sparsity constraints on the entries of the system response spectral factors. Recall that $\Phix = \sum_{k=0}^{\infty}  z^{-k}\Phi_x[k]$ and $\Phiu = \sum_{k=0}^{\infty}  z^{-k}\Phi_u[k]$, then the combined $(d,T)$-localization and communication subspace constraints can be encoded by specifying the support of the spectral factors. Define $\alpha = \floor{\frac{k-k_a-k_s}{k_c}}$, then we have
\begin{equation*}
\begin{array}{l}
\supp(\Phi_x[k]) \subseteq \supp\left((A+I)^{\min(d,\alpha)} \right)\\
\\
\supp(\Phi_u[k]) \subseteq  \supp\left((A+I)^{\min(d+1,\alpha)} \right)
\end{array} \quad \text{for } k=1\hdots T
\end{equation*}
and $\Phi_x[k] = \tf 0$,  $\Phi_u[k] = \tf 0 $ for all $k>T$. The reader is referred to~\cite{MickeyThesis} for a detailed description of the derivation of $\alpha$ and for some explicit cases where combinations of $d$ and the delay parameters are guaranteed to lead to feasible solutions. The delay constraints which are encoded as sparsity constraints on the spectral factors are convex constraints. For notational convenience we will use the set $\cD$ to denote these sparsity constraints. Thus a $(d,T)$-locality constraint with delays is given by $\Ss = \cL_d \cap \fir \cap \cD$.

\subsection{Scalability via Decomposition}\label{sec:scalability}
We now introduce the concept of column-separability for cost functionals and constraints. It is this column-wise separability that will allow us to achieve $O(1)$ synthesis complexity with respect to the state dimension of~\eqref{eq:LTI}.

For compactness of notation, in this section we define the transfer matrix $\tf \Phi = \left[ \Phix^T ~ \Phiu^T \right]^T$ and assume it dimension $n_r \times n_c$. Let $\{c_1, \hdots , c_p \}$ partition the set $\{1,\hdots, n_c\}$. Let us partition $\tf \Phi$ as $\{\tf \Phi(:,c_1)\hdots, \tf \Phi(:,c_p)\}$. This idea of partitioning a decision variable by its columns leads to the notion of column-wise separability:
\begin{defn}\label{def:cols}
The functional $g(\tf \Phi)$ is column-wise separable with respect to the partition $c_1,\hdots,c_p$ if it can be written as
\begin{equation*}
g(\tf \Phi) = \sum_{j=1}^pg_j(\tf \Phi(:,c_j))
\end{equation*}
for some functionals $g_j$ for $j=1,\hdots,p$.
\end{defn}
Some examples of column-wise separable norms are described below. We note that the $\hinf$-norm is \emph{not} column- (nor row-) wise separable.
{
\begin{example} [Frobenius norm]
The square of the Frobenius norm of a $m \times n$ matrix $\Phi$ is given by
\begin{equation}
\| \Phi \|_{F}^2 = \sum_{i=1}^m \sum_{j=1}^n \Phi_{ij}^2. \nonumber
\end{equation}
This objective function is column-wise separable with respect to arbitrary column-wise partition. Consider the explicit example of
\begin{equation*}
\Phi = \left[ \begin{array}{ccc} 1 & 2 & 3 \\ 4 & 5 & 6\end{array}  \right]
\end{equation*}
partitioned as $\{1,3\},\{2\}$, then 
\begin{equation*}
\|\Phi \|_F^2 =  \left\| \left[ \begin{array}{cc} 1 & 3  \\ 4 & 6\end{array}  \right]  \right\|_F^2 + \left\| \left[ \begin{array}{c} 2 \\  5 \end{array} \right] \right\|_F^2.
\end{equation*}
\end{example}
\begin{example} [$\mathcal{H}_2$ norm]
The square of the $\mathcal{H}_2$ norm of a transfer matrix $\ttf{\Phi}$ is given by
\begin{equation}
\| \ttf{\Phi} \|_{\mathcal{H}_2}^2 = \sum_{t=0}^\infty \| \Phi[t] \|_{F}^2, \label{eq:h2_comp}
\end{equation}
which is column-wise separable with respect to arbitrary column-wise partition. Now consider the $\htwo$-norm of a transfer matrix realization of a system with a finite impulse response of horizon $T$, i.e. $\tf \Phi \in \fir$. In this case the summation in~\eqref{eq:h2_comp} is finite.
\end{example}
\begin{example}[Element-wise $\ell_1$] \label{ex:l1}
Motivated by the separability of the $\mathcal{H}_2$ norm, we define the element-wise $\ell_1$ norm (denoted by $\mathcal{E}_1$) of a transfer matrix $\tf \Phi \in \FT$ as
\begin{equation}
|| \tf \tf \Phi ||_{\mathcal{E}_1}  =  \sum_i \sum_j \sum_{t=0}^T | \Phi_{ij}[t] |. \nonumber
\end{equation}
\end{example}

}

\begin{defn}
The constraint-set $\Ss$ in~\eqref{eq:SLSsyn_inf} is column-wise separable with respect to the partition $c_1,\hdots,c_p$ when the condition
\begin{equation*}
\tf \Phi \in \Ss \Leftrightarrow \tf \Phi(:,c_j)\in \Ss_j ~\text{for } j=1,\hdots,p
\end{equation*}
is satisfied for some sets $\Ss_j$ for $j =1,\hdots,p$, where $\Ss_j = \cL(:,c_j)\cap \cX_j \cap \fir$.
\end{defn}
{
\begin{example} [Affine Subspace]\label{ex:ZAB}
The affine subspace constraint
\begin{equation}
\tf G \ttf{\Phi} = \tf H \nonumber
\end{equation}
is column-wise separable with respect to arbitrary column-wise partition. Specifically, we have
\begin{equation}
\tf G \Sub{\ttf{\Phi}}{ : }{c_j} = \Sub{\tf{H}}{ : }{c_j} \nonumber
\end{equation}
for $c_j$ any subset of $\{1, \dots, n\}$. An important special case of column-wise  separable affine subspace constraints is the system level parametrization constraint~\eqref{eq:SLSsyn_inf}, here we have 
\begin{equation*}
\tf G = \begin{bmatrix} zI-A & - B_2 \end{bmatrix} ,\quad  \tf \Phi = \left[ \begin{array}{c} \Phix \\ \Phiu \end{array} \right], \text{and} \quad \tf H = I.
\end{equation*}
\end{example}
\begin{example} [Locality and FIR Constraints]
The locality and FIR constraints introduced in Section \ref{sec:locality}
\begin{equation}
\ttf{\Phi} \in \Ell \cap \FT \nonumber
\end{equation}
are column-wise separable with respect to arbitrary column-wise partition.  This follows from the fact that both locality and FIR constraints can be encoded via sparsity structure: the resulting linear subspace constraint is trivially column-wise separable.
\end{example}

With the definition of column-wise separability firmly established the following chain of problems will make clear how functions and constraint sets that satisfy these definitions provide computational tractability to large-scale problems. First, consider the generic optimization problem 
\begin{subequations} \label{eq:gopt}
\begin{align} 
\underset{\ttf{\Phi}}{\text{minimize }} \quad & g(\ttf{\Phi}) \label{eq:gopt-1}\\
\text{subject to } \quad & \ttf{\Phi} \in \s, \label{eq:gopt-2}
\end{align}
\end{subequations}
where the decision variable is $\ttf{\Phi}$, an $m \times n$ transfer matrix, $g(\cdot)$ a functional objective, and $\s$ a set constraint. Assume now that the objective function \eqref{eq:gopt-1} and the set constraint \eqref{eq:gopt-2} are both column-wise separable with respect to a column-wise partition $\{ c_1, \dots, c_p\}$. When the objective function and constraints are both column-wise separable we are said to have a  \emph{column-wise separable problem}. Specifically, \eqref{eq:gopt} can be partitioned into $p$ parallel subproblems as
\begin{subequations} \label{eq:gopt-decom}
\begin{align} 
\underset{\Sub{\ttf{\Phi}}{ : }{c_j}}{\text{minimize }} \quad & g_j(\Sub{\ttf{\Phi}}{ : }{c_j}) \label{eq:gopt-decom-1} \\
\text{subject to } \quad & \Sub{\ttf{\Phi}}{ : }{c_j} \in \s_j \label{eq:gopt-decom-3}
\end{align}
\end{subequations}
for $j = 1, \dots, p$. 
We will now show that System Level Synthesis problems of the form~\eqref{eq:SLSsyn_inf} under separability conditions can be written in the form of~\eqref{eq:gopt-decom} and explicitly form the problem. First, let us define
\begin{equation*}
\tf Z_{AB} = \begin{bmatrix} zI-A & -B_2 \end{bmatrix},
\end{equation*}
set $\tf \Phi$ as defined in example~\ref{ex:ZAB}, and let the the system level constraint be given by $\s = \Ell \cap \FT \cap \Sother$. The state-feedback SLS problem~\eqref{eq:SLSsyn_inf} can then be written as 
\begin{subequations}\label{eq:optSLSdecomp}
\begin{align}
 \min_{\tf \Phi}\quad &  g(\tf \Phi) \label{eq:SLSobj}\\
\st\quad &  \tf Z_{AB} = I \label{eq:Zab_aff}\\
& \tf \Phi\in \Ss  .\label{eq:set_cons}
\end{align}
\end{subequations}
As we have already shown, the affine subspace constraint~\eqref{eq:Zab_aff} is column-wise separable with respect to any column partitioning. If $g(\cdot)$ is chosen to be column-wise separable, for example it is given by the $\htwo-$ or  $\mathcal{E}_1$-norms, then~\eqref{eq:optSLSdecomp} is a column-wise separable optimization problem if the SLC~\eqref{eq:set_cons} is column-wise separable. Given that we have shown $\fir \cap \Ell $ to be separable, all that is needed is for $\Ss$ to be column-wise separable is for $\cX$ to also be separable. When this is true, we  can express the set constraint $\s_j$ in \eqref{eq:gopt-decom-3} as $\s_j = \Sub{\Ell}{:}{c_j} \cap \FT \cap \Sother_j$ for some $\Sother_j$ for each $j$. Finally, using the column partitions, \eqref{eq:optSLSdecomp} can be partitioned into $p$ parallel subproblems of the form
\begin{subequations} \label{eq:decom}
\begin{align} 
\underset{\Sub{\ttf{\Phi}}{ : }{\colps_j}}{\text{minimize }} \quad & g_j(\Sub{\ttf{\Phi}}{ : }{c_j}) \label{eq:decom-1}\\
\text{subject to } \quad & \tf Z_{AB} \Sub{\ttf{\Phi}}{ : }{c_j} = \Sub{I}{ : }{c_j} \label{eq:decom-2} \\
& \Sub{\ttf{\Phi}}{ : }{c_j} \in \Sub{\Ell}{:}{c_j} \cap \FT \cap \Sother_j \label{eq:decom-3}
\end{align}
\end{subequations}
for $j = 1, \dots, p$. 
The column-wise separable SLS problem~\eqref{eq:decom} is exactly in the form of the generic column-wise separable optimization problem~\eqref{eq:gopt-decom}.

 Observe that the $p$ sub-problems are completely decoupled and can be solved in parallel, moreover, because of the constraint $\cL$ the dimension of each the sub-problem will be smaller than that of the global problem~\eqref{eq:SLSsyn_inf} -- this is because the decomposed constraint $\cL(:,c_j)$ has sparse support and can be permuted into partition of zero and non-zero elements, the zero elements can then be dropped. Effectively, this means that each sub-problem only has to work with a local system model. Subsystem $i$ comprising $(A_{ii}, A_{ij}, B_{ii})$ is a constructed from suitably defined sub-blocks of $(A,B_2)$ as defined in~\eqref{eq:LTI}.
 
 The following simplified example should provide some additional intuition into the role of spatio-temporal locality.
 \begin{example}[Symmetric Chain]
 Consider the network dynamical system with the dynamics at node $i$ given by
 \begin{align*}
 x_i[t+1] &= A_{ii}x_i[t] + \sum_{j\in \mathcal{N}_i}A_{ij}x_j[t] + B_{ii}u_i[t] + \delta_{x_i}[t].
 \end{align*}
 The set $\mathcal{N}_i$ denotes the incoming set of node $i$, in this case the system graph is a bi-directional chain, i.e. the adjacency matrix has zeros everywhere apart from at the sub- and super-diagonal entries. Furthermore, we assume that the control inputs and disturbances only directly affect  local subsystems. In Figure~\ref{fig:time_space_cone}  a diagram of the system response to a particular disturbance $(\ttf{\delta_x})_i$ is depicted. The horizontal-axis denotes time and the vertical-axis indicates position in the chain. At time $t$ an impulse hits node $i$ (chosen as the center of the chain). The figure shows  the communication delay imposed on the controller via the QI subspace SLC, the deadbeat response of the system to the disturbance imposed by the FIR SLC, and the localized region affected by the disturbance $(\ttf{\delta_x})_i$ imposed by the localized SLC. 
 
 \begin{figure}[ht!]
      \centering
      \includegraphics[width=0.6\textwidth]{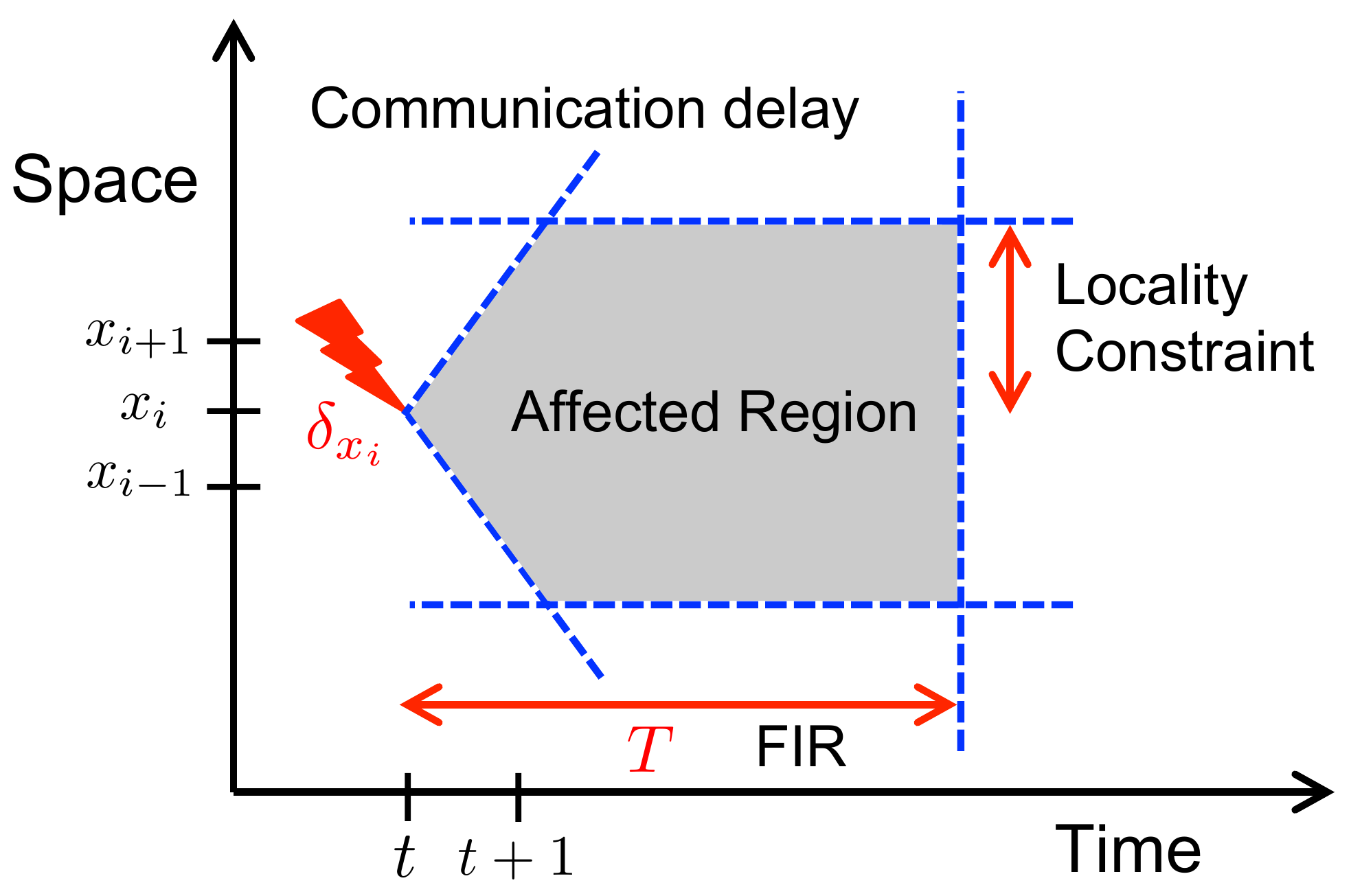}
      \caption{The state response to a disturbance hitting the chain at time $t$.}
      \label{fig:time_space_cone}
\end{figure}

\end{example}

\subsubsection{Dimension Reduction}
In the previous section, the idea of column-wise separability for functions and constraints was introduced. It was shown that SLS problems~\eqref{eq:SLSsyn_inf} defined by column-wise separable components could be easily solved by a set of parallel sub-problems as described by~\eqref{eq:decom}. In this subsection, we take advantage of the fact that locality constraints $\Ell(:,c_j)$ and network structure likely enforce a sparsity pattern on the data and decision variables. Such sparsity leads to many of the entries in the columns of $\tf \Phi (:,c_j)$ being constrained to be zero, thus many of the elements of $\tf Z_{AB}$ are not needed - we now proceed to demonstrate how this can capitalized on.

Let $s_j$ and $t_j$ be sets of positive integers, and assume that we have a column-wise separable SLS problem~\eqref{eq:decom} with column partitions $\{c_1,\hdots, c_p\}$. Later in the section we will describe how to construct the sets $s_j$ and $t_j$, for now we focus on the intuition behind the algorithm. It can be taken for granted that the cardinality of the set $s_j$ is less than the number of rows in $\tf \Phi(:,c_j)$. Likewise, the cardinality of $t_j$ will be less than the number of rows in $\tf Z_{AB}$. Consider the $j^{\text{th}}$ subproblem from~\eqref{eq:decom}, it will be shown that it is equivalent to the following optimization problem:
\begin{subequations} \label{eq:decom40}
\begin{align} 
\underset{\Sub{\ttf{\Phi}}{\rone_j}{c_j}}{\text{minimize }} \quad & \bar{g}_j(\Sub{\ttf{\Phi}}{\rone_j}{c_j}) \label{eq:decom40-1}\\
\text{subject to } \quad & \Sub{\tf Z_{AB}}{\rtwo_j}{\rone_j} \Sub{\ttf{\Phi}}{\rone_j}{c_j} = \Sub{I}{\rtwo_j}{c_j} \label{eq:decom40-2} \\
& \Sub{\ttf{\Phi}}{\rone_j}{c_j} \in \Sub{\Ell}{\rone_j}{c_j} \cap \FT \cap \bar{\Sother}_j \label{eq:decom40-3},
\end{align}
\end{subequations}
where $\tf \Phi(s_j,c_j)$ corresponds to a sub-matrix of $\tf \Phi$ composed of the rows and columns of $\tf \Phi$ specified by the sets $s_j$ and $c_j$ respectively. The cost functional $\bar{g}_j$  is simply the functional $g_j$ restricted to the reduced set of variables, likewise $ \bar{\Sother}_j$ is simply $\cX_j$ restricted to the appropriate subset. Roughly speaking, the set $s_j$ is the collection of optimization variables contained within the localized region specified by $\Sub{\Ell}{:}{c_j}$, and the set $t_j$ is the collection of states that are directly affected by the optimization variables in $s_j$.

The computational complexity of solving~\eqref{eq:decom40} depends on the choice of cost functional $g$ (and hence $\bar{g}_j$) and the number of variables in the problem. It should be clear by looking at~\eqref{eq:decom40} that the cardinality of the sets $c_j,s_j,t_j$ determine the computational complexity of the resulting reduced-order SLP. For example, the cardinality of the set $s_j$ is equal to the number of nonzero rows of the locality constraint $\Sub{\Ell}{:}{\colps_j}$. The sparsity of the system matrices is also central to the cardinality of theses sets. When the locality constraint and the system matrices are suitably sparse, it is possible to make the size of these sets much smaller than the size of the global network. In this case, the global optimization subproblem \eqref{eq:decom} reduces to a local optimization subproblem \eqref{eq:decom40} which depends only on the local plant model $\Sub{\tf Z_{AB}}{t_j}{s_j}$. 

Combining column-wise separability with the dimension-reduction procedure described above, we have the following algorithm: 
\\
\noindent
\\
\textbf{Input: }$g(\cdot), \Ss,$ and $\tf{Z}_{AB}$
\begin{enumerate}
\item Find the column-wise partitioning $\{c_1,\hdots ,c_p\}$ such that~\eqref{eq:optSLSdecomp} decomposes into~\eqref{eq:decom}.
\item \textbf{For} $j=1,\hdots, p$

Construct the sets $s_j, t_j$ such that~\eqref{eq:decom} reduces to ~\eqref{eq:decom40}.

Solve~\eqref{eq:decom40}.
\\
\noindent
\\
\textbf{Output: }$\tf \Phi$

\textbf{End For}
\item From sub-solutions $\tf \Phi(s_j,c_j)$ with $j=1,\hdots,p$, construct $\tf \Phi$.
\end{enumerate}

The above algorithm embodies the idea of locality. Not only does locality allow for decoupling of the system level synthesis problem, it also provides a natural set of reduced order models where each variable in the smaller problem corresponds to a ``local'' patch of the system model. The reader is referred to \cite{SysLevelSyn2} for the details of how to construct minimal sets $s_j$ and $t_j$.

}

\subsection{Localized LQR Optimal Control}
\label{sec:llqr}
Here we provide an example of a classical distributed control problem, namely that of distributed $\mathcal{H}_2$, or LQR, optimal control, that can be made scalable by imposing locality constraints on the system responses $\SFpair$, and exploiting the separability of the resulting optimization problem.  We call the resulting optimization problem the Localized LQR (LLQR) problem.

In \cite{2014_Wang_CDC}, we formulate the LLQR problem with uncorrelated process noise (i.e., $B_1 = I$) as
\begin{subequations} \label{eq:llqr}
\begin{align}
\underset{\{\Phix, \Phiu\}}{\text{minimize}} \quad & || \begin{bmatrix} C_1 & D_{12} \end{bmatrix} \begin{bmatrix} \Phix \\ \Phiu \end{bmatrix} ||_{\mathcal{H}_2}^2 \label{eq:llqr-1}\\
\text{subject to} \quad & \begin{bmatrix} zI - A & -B_2 \end{bmatrix} \begin{bmatrix} \Phix \\ \Phiu \end{bmatrix} = I \label{eq:llqr-2} \\
& \begin{bmatrix} \Phix \\ \Phiu \end{bmatrix} \in \Ell \cap \FT \cap \frac{1}{z} \mathcal{RH}_\infty. \label{eq:llqr-3}
\end{align}
\end{subequations}
The (system level objective) SLO \eqref{eq:llqr-1} and the SLC \eqref{eq:llqr-3} are both column-wise separable with respect to arbitrary column-wise partitions. The separable property of the SLO is implied by the separability of the $\mathcal{H}_2$ norm and the assumption that the process noise is pairwise uncorrelated. The separability of the constraints \eqref{eq:llqr-2} - \eqref{eq:llqr-3} follows from the above discussion pertaining to affine subspaces. The physical interpretation of the column-wise separable property is that we can analyze the system response of each local process disturbance $\ttf{\delta_{x_j}}$ in an independent and parallel way, and then exploit the superposition principle satisfied by LTI systems to reconstruct the full solution to the LLQR problem.  We also note that removing the locality and FIR constraints does not affect column-wise separability, and hence the standard LQR problem with uncorrelated noise is also column-wise separable.
In general, imposing FIR or locality constraints degrades the transient performance of the controller, so our controller is not $\mathcal{H}_2$ optimal. However, we will see in the next section via a numerical example that by choosing appropriate $(d,T)$ parameters, the difference becomes negligible. In Section~\ref{sec:firaprox} we will derive bounds on the gap between the optimal solution and an FIR approximation. 

\subsubsection{Performance Comparison}\label{sec:performance}
In this section, we synthesize our LLQR optimal controller for a specific plant, and compare the performance with different classes of $\mathcal{H}_2$ optimal controllers from \cite{2013_Lamperski_H2}. In particular, we consider the centralized, delayed centralized, and optimal distributed (with quadratically invariant (QI) information sharing constraints) $\mathcal{H}_2$ controllers. 

The plant model is given by $(A,B)$, where $A$ is a tridiagonal matrix with dimension $59$ given by
\begin{eqnarray}
A = \begin{bmatrix} 1 & 0.2 & \cdots & 0\\
-0.2 & \ddots & \ddots & \vdots\\
\vdots & \ddots & \ddots & 0.2\\
0 & \cdots & -0.2 & 1
\end{bmatrix}. \nonumber
\end{eqnarray}
The instability of the plant is quantified by the spectral radius of $A$, which is $\rho(A) = 1.0768 > 1$. $B$ is a $59 \times 20$ matrix with the $(6n+1,2n+1)$-th and $(6n+2,2n+2)$-th entries being $1$ and zero elsewhere, $n = 0, \dots, 9$. Even though our method can be applied to an arbitrary plant topology, we use this simple plant model as it leads to easily visualized representations of the aforementioned space-time regions.

For the localized controller, we choose $(d,T) = (9,29)$. For all controllers with communication delay constraints, we assume that the communication network has the same topology as the physical network, but the speed is $t_c = 1.5$ times faster than the speed at which dynamics propagate through the plant. For our localized distributed controllers, we have
\begin{eqnarray}
\text{supp}\left(\Phi_x[t]\right) &=& \bigcup_{t=1}^T \frac{1}{z^t} \text{supp}\left({A}^{\min(d, \lfloor t_c(t-1) \rfloor)}\right) \nonumber\\
\text{supp}\left(\Phi_u[t]\right) &=& \bigcup_{t=2}^T \frac{1}{z^t} \text{supp}\left({B^\top}\right)\text{supp}\left({A}^{\min(d+1, \lfloor t_c(t-2) \rfloor)}\right)
\end{eqnarray}
for each value of $t$ up to $T$. The union operation for binary matrices is taken as the element-wise OR operation.
We then solve \eqref{eq:llqr} with $C_1$ and $D_{12}$ set to identity.

We illustrate the difference between the different control schemes by plotting the space-time evolution of a single disturbance hitting the middle state. A concrete realization of a localized forward space-time region is in Figures \ref{fig:subfigure8} and \ref{fig:subfigure9}: the effect of the disturbance is limited in both state and control action in time and space. 

\begin{figure}[h!]
      \centering
      \subfigure[State for Open loop]{%
      \includegraphics[width=0.22\textwidth]{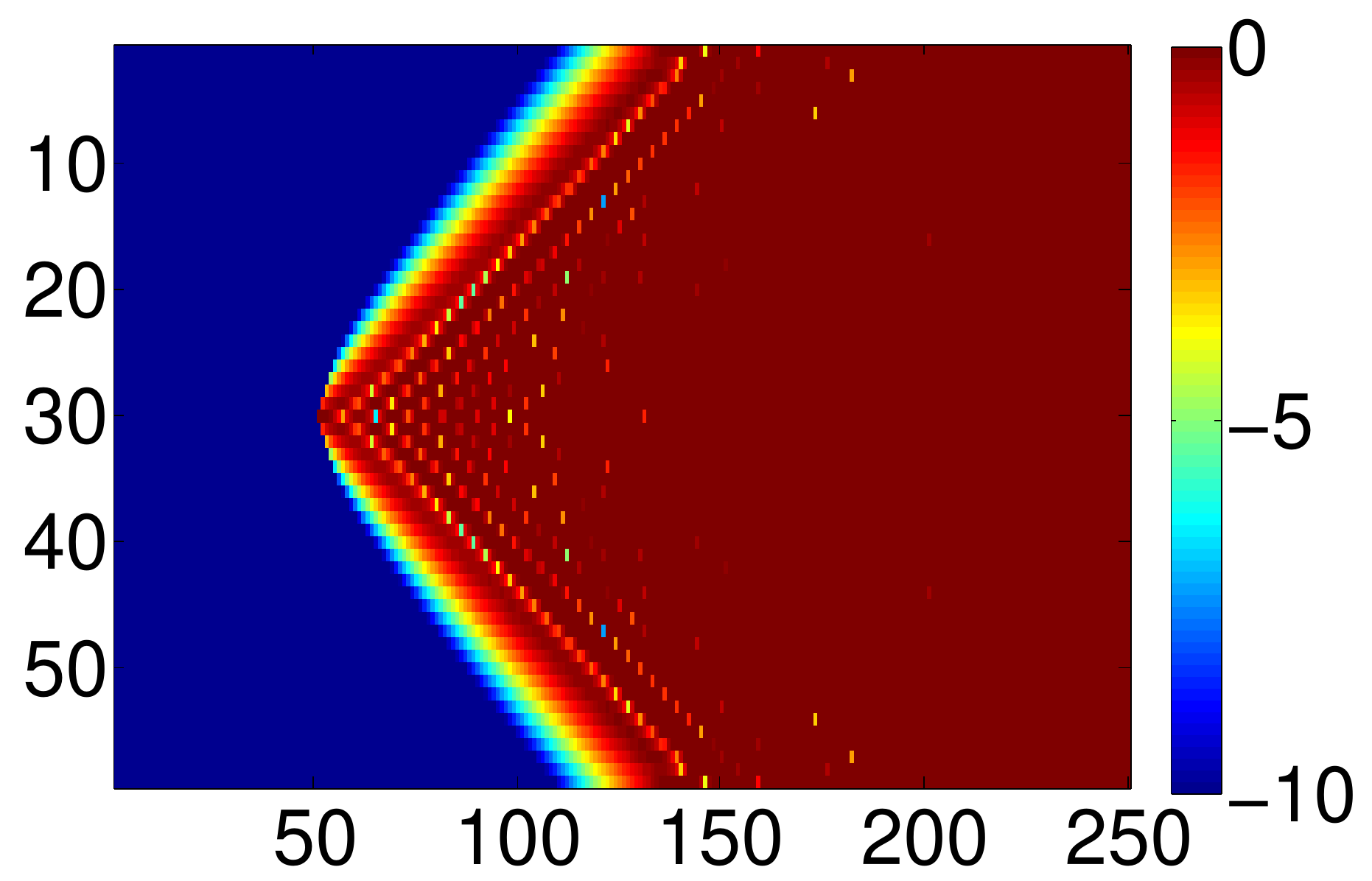}
      \label{fig:open}}
      
        \subfigure[State for Ideal $\mathcal{H}_2$]{%
          \includegraphics[width=0.22\textwidth]{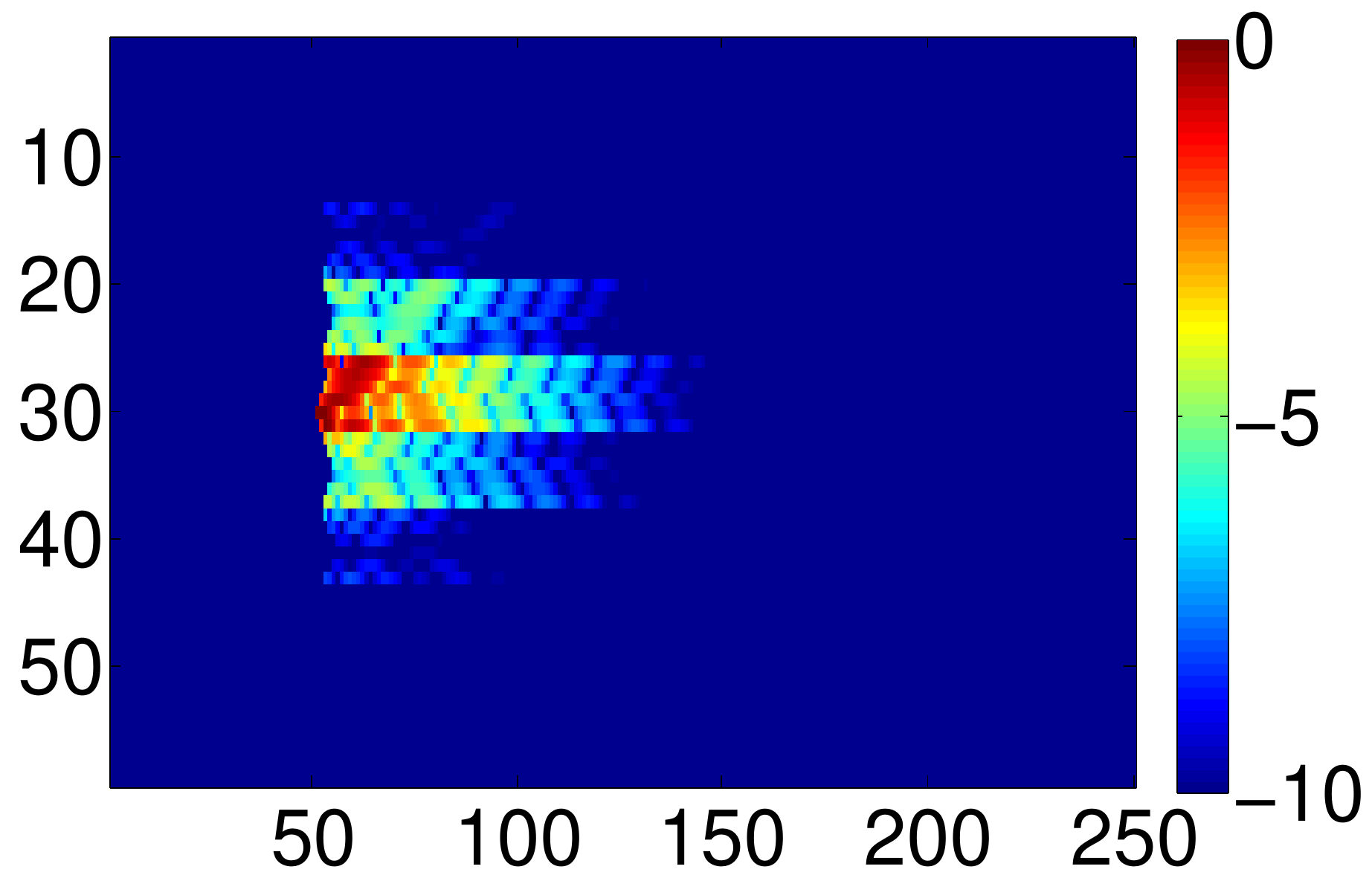}
          \label{fig:subfigure2}}
        \quad
        \subfigure[Control for Ideal $\mathcal{H}_2$]{%
          \includegraphics[width=0.22\textwidth]{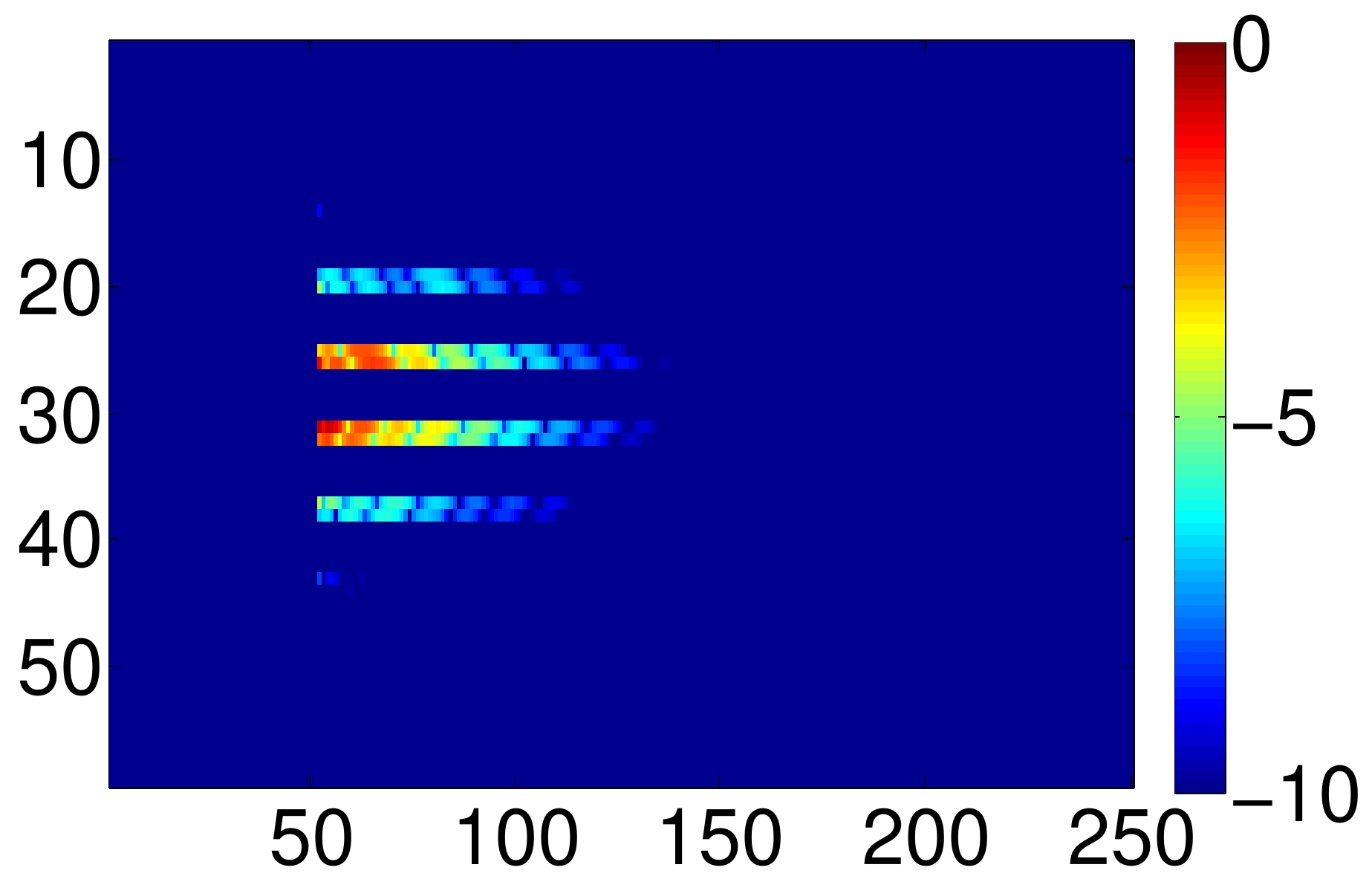}
          \label{fig:subfigure3}}
        
        \subfigure[State for Delayed $\mathcal{H}_2$]{%
          \includegraphics[width=0.22\textwidth]{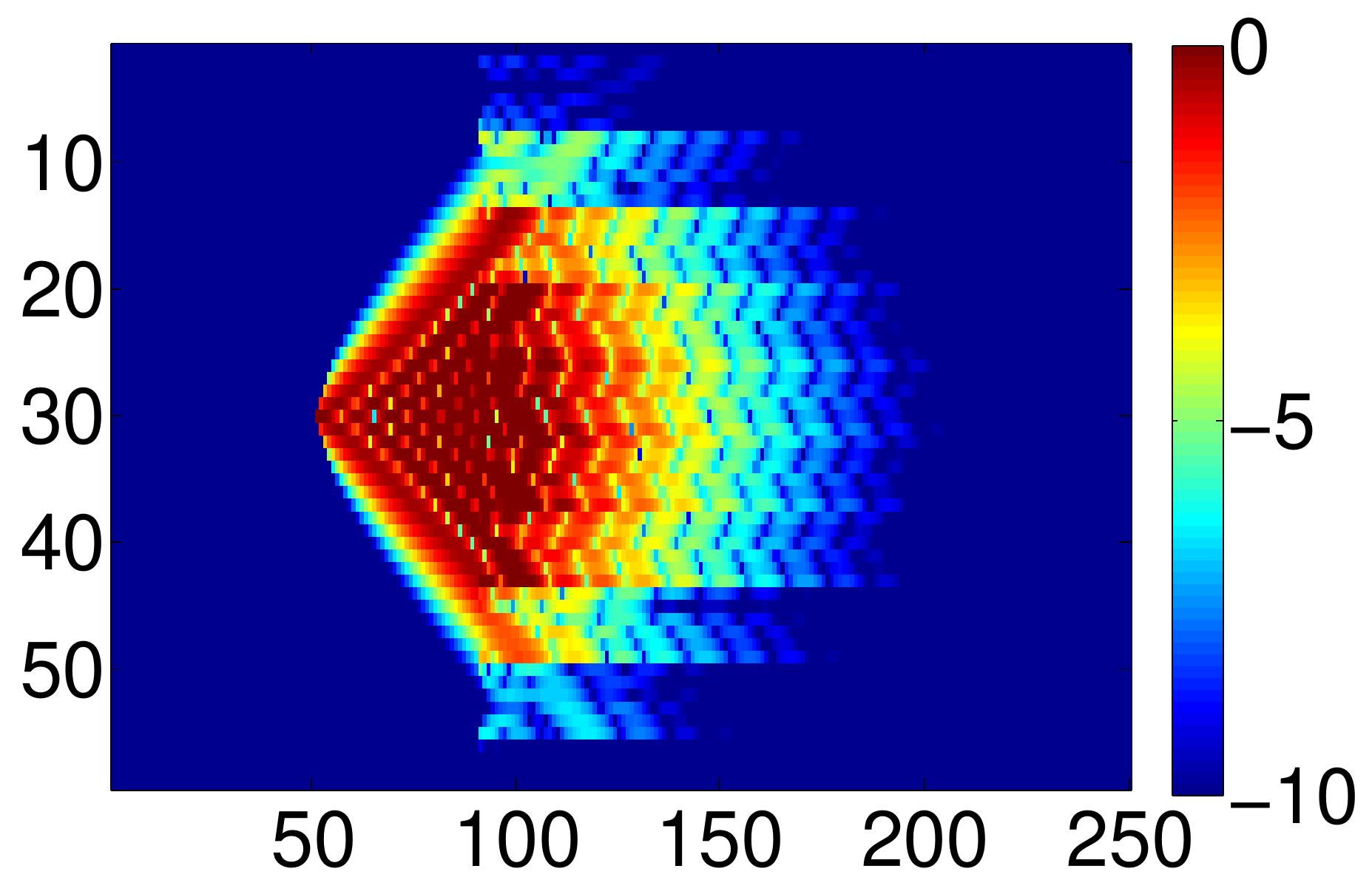}
          \label{fig:subfigure4}}
        \quad
        \subfigure[Control for Delayed $\mathcal{H}_2$]{%
          \includegraphics[width=0.22\textwidth]{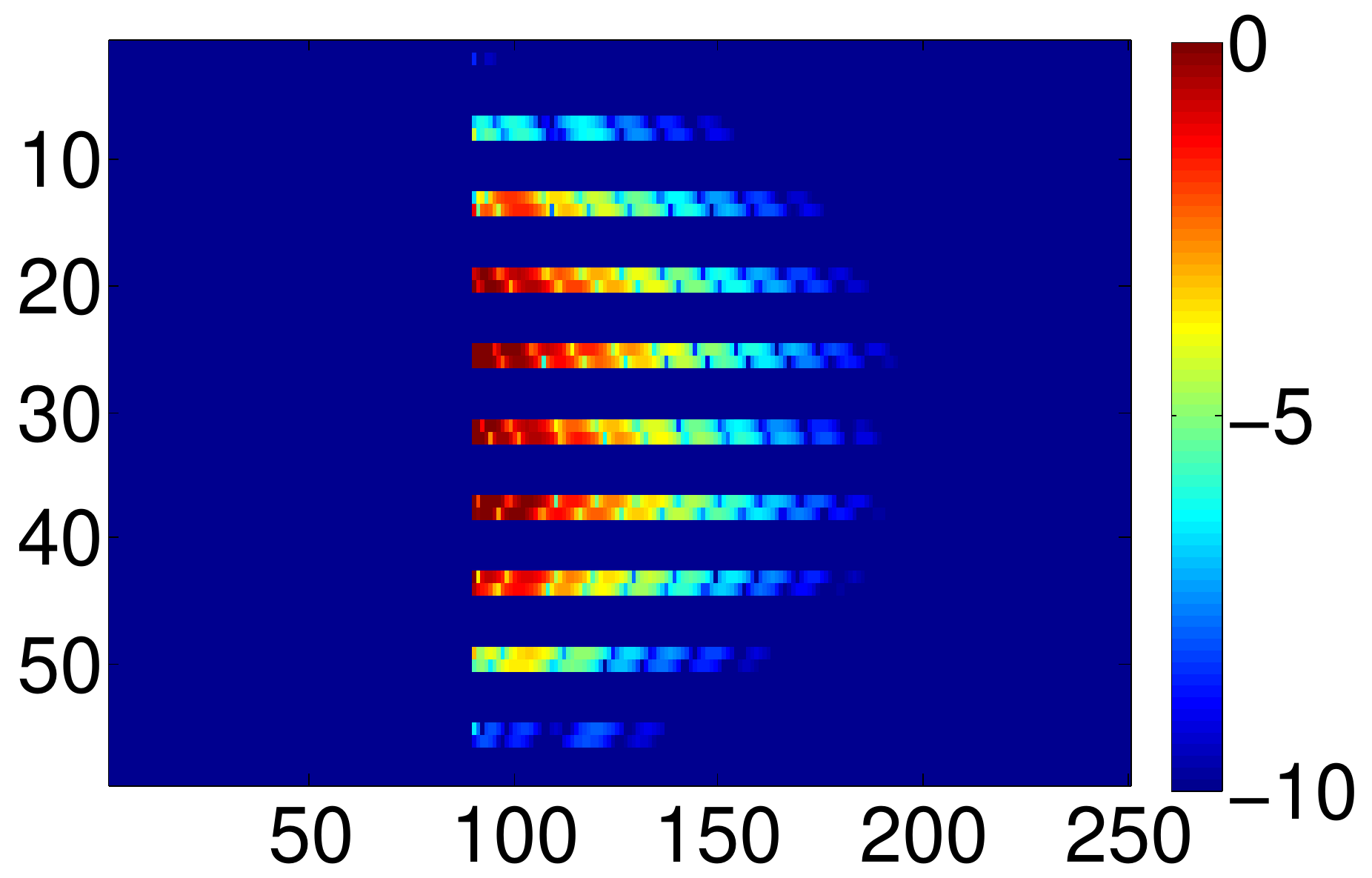}
          \label{fig:subfigure5}}
          
        \subfigure[State for Distributed]{%
          \includegraphics[width=0.22\textwidth]{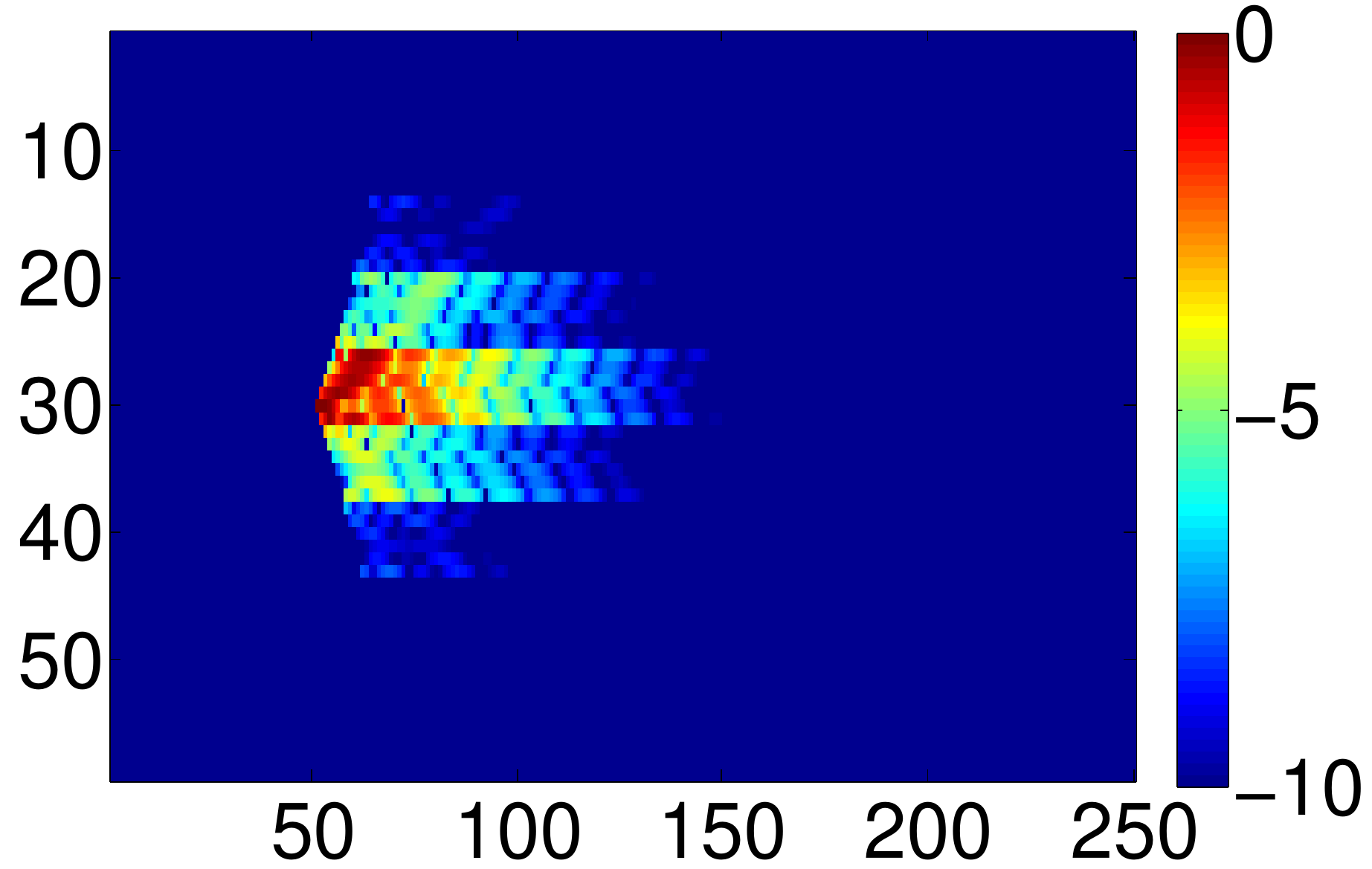}
          \label{fig:subfigure6}}
          \quad
        \subfigure[Control for Distributed]{%
          \includegraphics[width=0.22\textwidth]{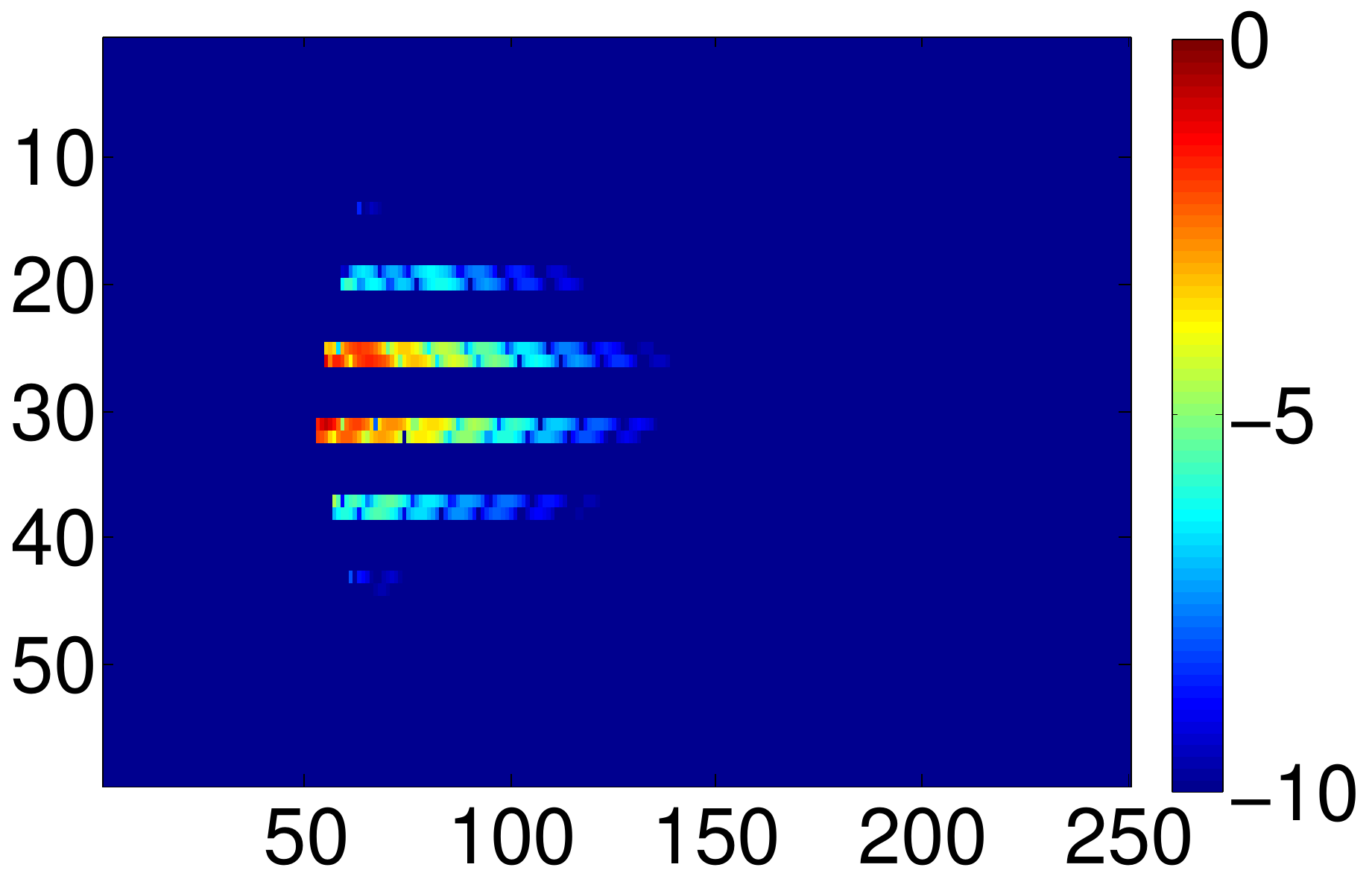}
          \label{fig:subfigure7}}          

        \subfigure[State for LLQR]{%
          \includegraphics[width=0.22\textwidth]{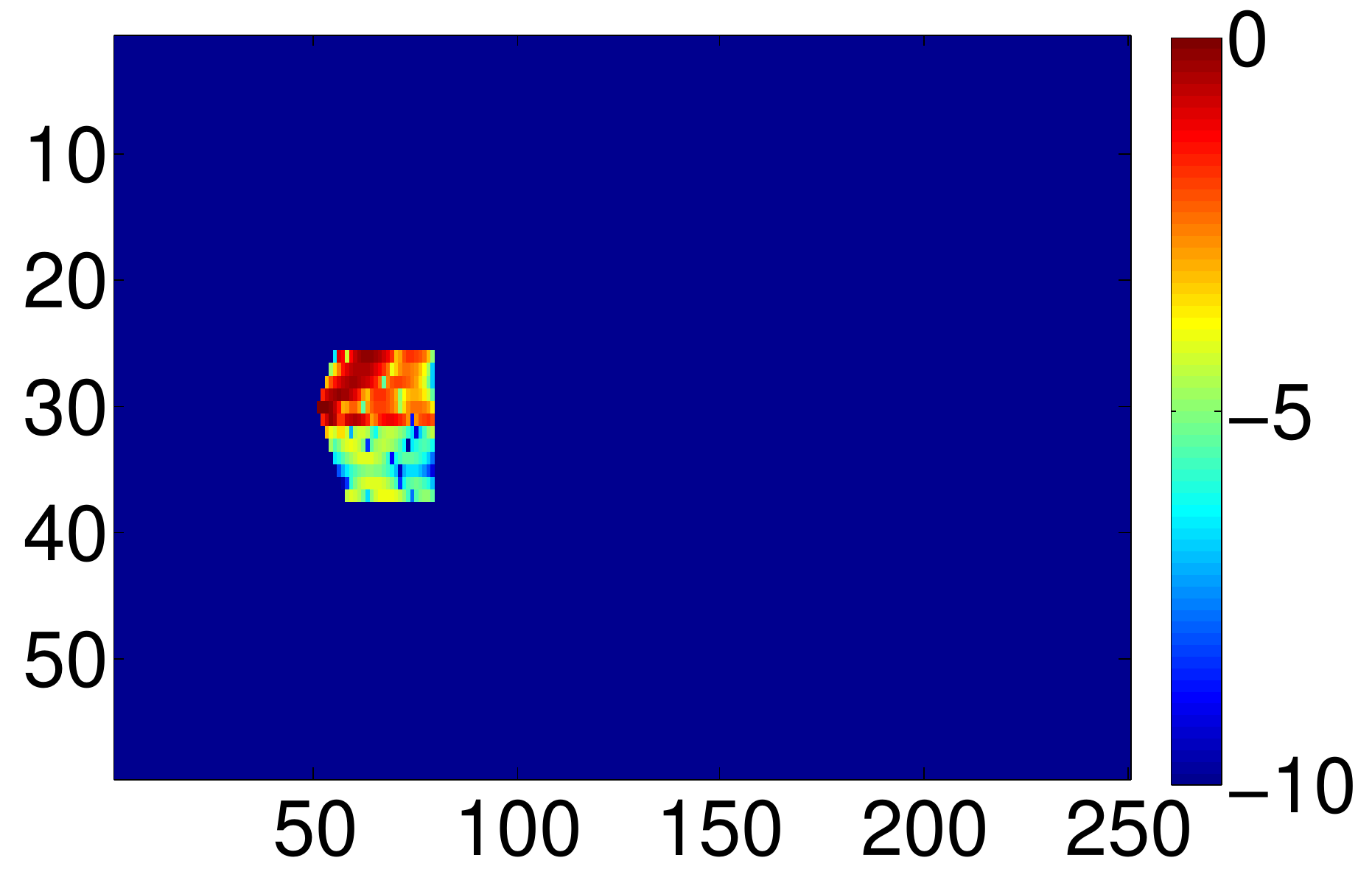}
          \label{fig:subfigure8}}
        \quad
        \subfigure[Control for LLQR]{%
          \includegraphics[width=0.22\textwidth]{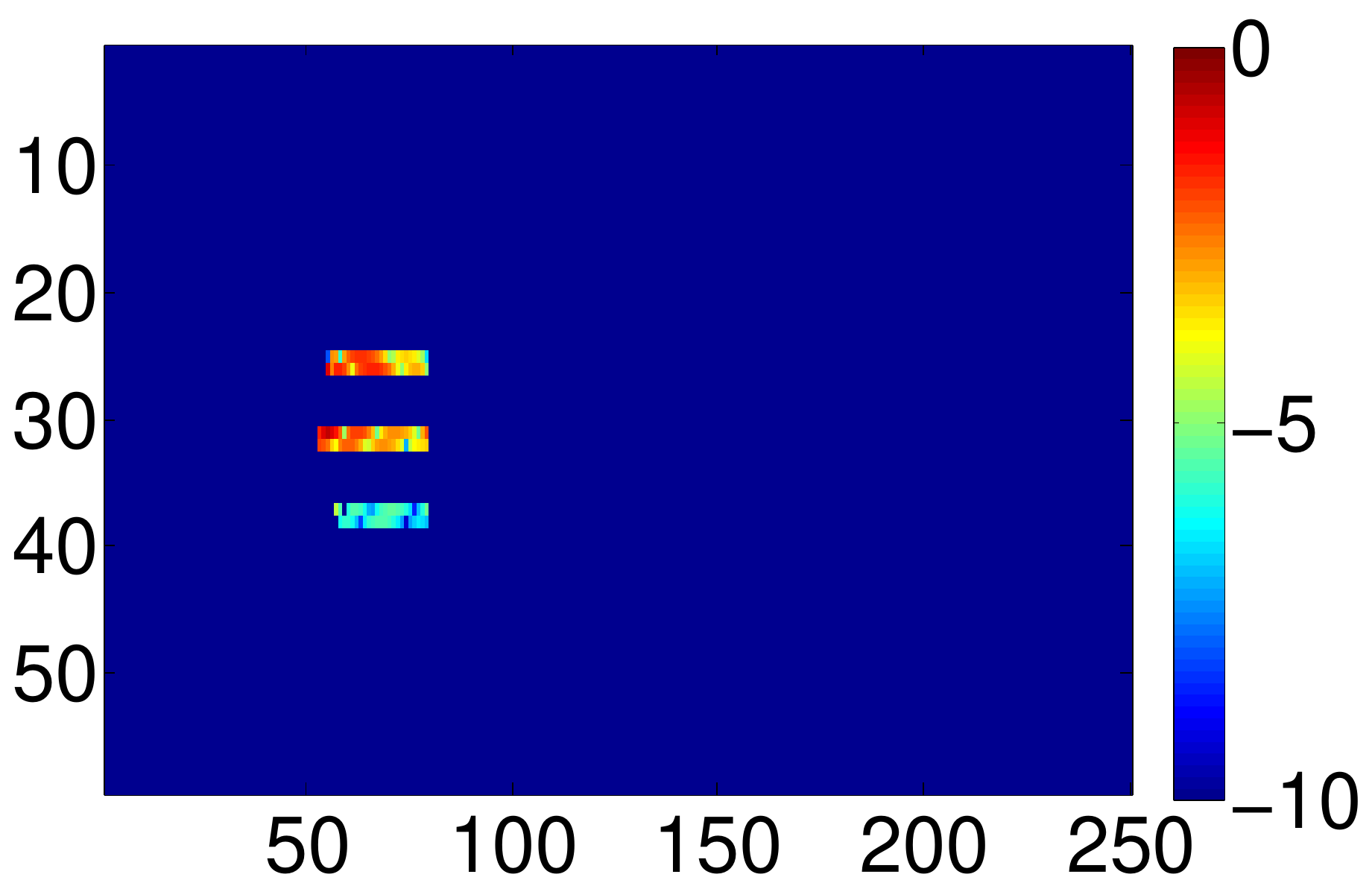}
          \label{fig:subfigure9}}
        \caption{ The log absolute value of state and control for a given disturbance at time $t = 50$. The horizontal axis represents time and the vertical axis represents state in space. The legend on the right shows the meaning of the colors.}\label{fig:schemes}
\end{figure}

Next, we calculate the optimal value for each controller and summarize the results in Table \ref{Table:1}. The objective is normalized with respect to that of the centralized $\mathcal{H}_2$ controller. Clearly, our localized control scheme can achieve similar performance to that of the centralized one. Numerical evidence seems to suggest that this property holds for most plants that are localizable. For an ideal $\mathcal{H}_2$ controller, the closed loop response decays exponentially in both time and space, as indicated by Figures \ref{fig:subfigure2} and \ref{fig:subfigure3}. Therefore, it is generally possible to find a favorable $(d,T)$ to synthesize the localized controller such that the closed loop transient response does not degrade much -- this is akin to the insight in \cite{2005_Bamieh_spatially_invariant} used to localize the implementation of funnel-causal systems. 

In summary, our result demonstrates that the LLQR optimal controller can be synthesized and implemented in a localized manner, but can achieve transient response close to that of an unconstrained optimal controller.

\begin{table}[h]
 \caption{Comparison between Different Control Schemes}
 \label{Table:1}
\begin{center}
    \begin{tabular}{ | l | l | l | l | l |}
    \hline
     & Ideal $\mathcal{H}_2$ & Delayed & Distributed& LLQR \\ \hline
    Comm Speed & Inf & 1.5 & 1.5 & 1.5 \\ \hline
    Control Time & Inf & Inf & Inf & 29 \\ \hline
    Locality & Max(58) & Max(58) & Max(58) & 9 \\ \hline
    Objective & 1 & 126.7882 & 1.1061 & 1.1142 \\ \hline
    \end{tabular}
\end{center}
\end{table}

{

\begin{figure}[ht!]
      \centering
      \includegraphics[width=0.45\textwidth]{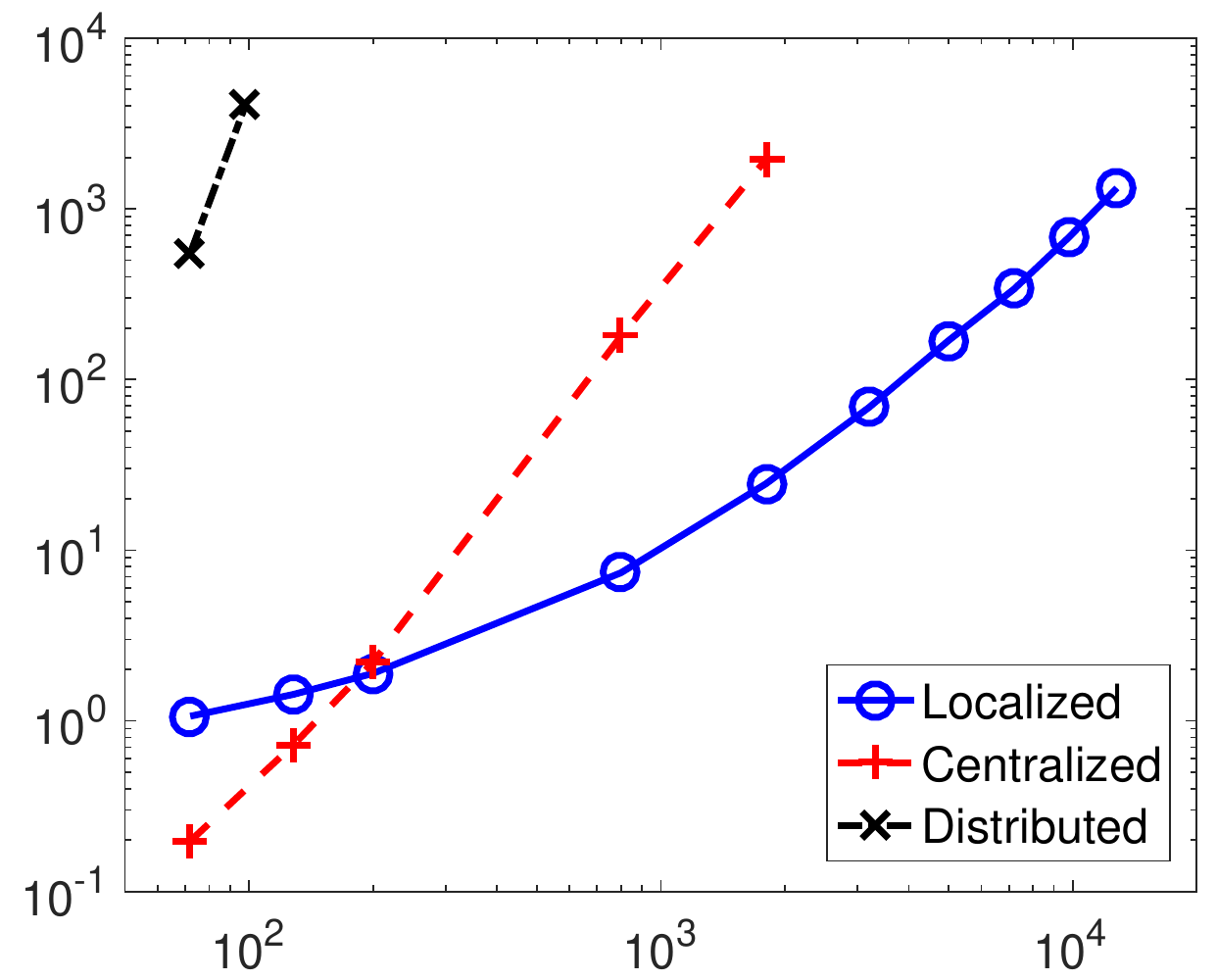}
      \caption{The horizontal axis denotes the number of states of the system, and the vertical axis is the computation time in seconds. The FIR horizon length is $T=7$.}
      \label{fig:Time}
\end{figure}

Figure \ref{fig:Time} clearly shows the computation time needed for the distributed controller grows rapidly when the size of problem increases. The gradient of the curve corresponding to  the localized controller is about $1.4$ which is slightly larger than the theoretical value of $1$. This issue is likely due to memory management issues and not related to the fundamental optimization problems. We further note that each of the distributed problems were solved in series on a single laptop, the performance would be significantly improved with a true parallel implementation.

}

\subsection{Robustness}
\label{sec:robust}
The previous sections highlighted the benefits of spatiotemporal locality: imposing FIR constraints trivially lead to finite-dimensional optimization problems (and simple filter-bank based controller implementations), and spatial locality leads to huge wins in terms of controller synthesis and implementation complexity.  However, both of these constraints can often be fairly restrictive in practice: indeed, if a system is stabilizable, but not controllable, we do not expect any such constraints to hold exactly.  This section extends the robustness results presented in Section \ref{sec:robust-time} to the infinite horizon setting, allowing us to circumvent these issues.  Further, we show that this allows us to formalize the folk theorem that ``good controllers are easy to compute.''

\subsubsection{A robustness result}

We begin with the infinite horizon analogue of the robustness result presented in Section \ref{sec:robust-time} that provides necessary and sufficient conditions under which a controller implemented using transfer matrices that only approximately satisfy the achievability constraint \eqref{eq:affine_cons} is stabilizing.  We then leverage this result and the use of ``virtual'' actuation and communication resources at design time to define a principled approach to approximate controller synthesis.

\begin{theorem}
Let $(\Phixh, \Phiuh, \ttf{\Delta})$ be a solution to
\begin{equation}
\begin{bmatrix} zI - A & -B_2 \end{bmatrix} \begin{bmatrix} \Phixh \\ \Phiuh \end{bmatrix} = I + \ttf{\Delta}, \quad \Phixh, \Phiuh \in \frac{1}{z}\RHinf.  \label{eq:near_sf1}
\end{equation}

Then, the controller implementation 
\begin{subequations}\label{eq:qf}
\begin{align}
\ttf{\hat{\delta}_x} &= \tf x - \tf{\hat{x}} \label{eq:qf1}\\
\tf u &= z \Phiuh \ttf{\hat{\delta}_x} \label{eq:qf2}\\
\tf{\hat{x}} &= (z \Phixh - I) \ttf{\hat{\delta}_x}. \label{eq:qf3} 
\end{align}
\end{subequations}
internally stabilizes the system $(A,B_2)$ if and only if $(I + \ttf{\Delta})^{-1}$ is stable.  Furthermore, the actual system responses achieved are given by
\begin{equation}
\begin{bmatrix} \tf x \\ \tf u \end{bmatrix} = \begin{bmatrix} \Phixh \\ \Phiuh \end{bmatrix}(I+ \ttf\Delta)^{-1}\ttf{\delta_x}.
\label{eq:robust-response}
\end{equation}
\label{thm:robust}
\end{theorem}
\begin{proof}
For a state-feedback system \eqref{eq:LTI} with controller implementation \eqref{eq:qf} using transfer matrices $\{\Phixh, \Phiuh\}$, the following holds
\begin{subequations}
\begin{align}
(zI-A) \tf x &= B_2 \tf u + \ttf{\delta_x} \label{eq:qf4}\\
\tf u &= z \Phiuh \ttf{\hat{\delta}_x} + \ttf{\delta_u} \label{eq:qf5}\\
\tf x &= z \Phixh \ttf{\hat{\delta}_x} - \ttf{\delta_y}. \label{eq:qf6},
\end{align}
\end{subequations}
where $\ttf{\delta_y}$ and $\ttf{\delta_u}$ are respectively perturbations on the measurement and control action (as illustrated in Figure \ref{fig:sf_controller}) introduced to verify the internal stability of the resulting closed loop system.

Substituting \eqref{eq:qf5} and \eqref{eq:qf6} into \eqref{eq:qf4}, we have
\begin{equation}
z(zI-A) \Phixh \ttf{\hat{\delta}_x} - (zI-A)\ttf{\delta_y} = z B_2 \Phiuh \ttf{\hat{\delta_x}} + B_2 \ttf{\delta_u} + \ttf{\delta_x}. \nonumber
\end{equation}
Moving $\ttf{\hat{\delta_x}}$ to the left-hand-side and using relation \eqref{eq:near_sf1}, we have 
\begin{equation}
z (I+\ttf{\Delta}) \ttf{\hat{\delta}_x} = (zI-A)\ttf{\delta_y} + B_2 \ttf{\delta_u} + \ttf{\delta_x}, \nonumber
\end{equation}
Let $\ttf{I_\Delta} := (I+\ttf{\Delta})^{-1}$. We then have that the closed loop transfer matrices from $(\ttf{\delta_x},$ $\ttf{\delta_y},$ $\ttf{\delta_u})$ to $\ttf{\hat{\delta}_x}$ are given by
\begin{equation}
\ttf{\hat{\delta}_x} = \frac{1}{z} \ttf{I_\Delta} \ttf{\delta_x} + \ttf{I_\Delta} (I-\frac{1}{z}A)\ttf{\delta_y} + \frac{1}{z} \ttf{I_\Delta} B_2 \ttf{\delta_u}. \label{eq:qf7}
\end{equation}
Substituting \eqref{eq:qf7} into \eqref{eq:qf5} and \eqref{eq:qf6}, we have the closed loop transfer matrices from $(\ttf{\delta_x},$ $\ttf{\delta_y},$ $\ttf{\delta_u})$ to $(\tf x,$ $\tf u$, $\ttf{\hat{\delta}_x})$ summarized in Table \ref{Table:200}. Clearly, if $\ttf{I_\Delta}$ is stable, then all the transfer matrices in Table \ref{Table:200} are stable. If $\ttf{I_\Delta}$ is unstable, then the closed loop maps from $\ttf{\delta_x}$ to $\ttf{\hat{\delta}_x}$ will be unstable, and the controller does not internally stabilize the system. Therefore, the stability of $\ttf{I_\Delta} = (I + \ttf{\Delta})^{-1}$ is necessary and sufficient condition for the controller implementation \eqref{eq:qf} to internally stabilize the system $(A,B_2)$.  Finally, Table \ref{Table:200} shows that the actual response achieved from $\ttf{\delta_x} \to (\tf x, \tf u)$ is given by \eqref{eq:robust-response}.
\end{proof}

\begin{table}[h!]
 \caption{Closed Loop Maps With Non-localizability}
 \label{Table:200}
\begin{center}
\renewcommand{\arraystretch}{2}
    \begin{tabular}{| c | c | c | c | c |}
    \hline 
    & $\ttf{\delta_x}$ & $\ttf{\delta_y}$ & $\ttf{\delta_u}$ \\ \hline
    $\tf x$ & $\Phixh \ttf{I_\Delta}$ & $\Phixh \ttf{I_\Delta} (zI - A) - I$ & $\Phixh \ttf{I_\Delta} B_2$ \\ \hline
    $\tf u$ & $\Phiuh \ttf{I_\Delta}$ & $\Phiuh \ttf{I_\Delta} (zI - A)$ & $I + \Phiuh \ttf{I_\Delta} B_2$ \\ \hline
    $\ttf{\hat{\delta}_x}$ & $\frac{1}{z} \ttf{I_\Delta}$ & $\ttf{I_\Delta} (I - \frac{1}{z} A)$ & $\frac{1}{z} \ttf{I_\Delta} B_2$ \\ \hline
    \end{tabular}
\end{center}
\end{table}

Theorem \ref{thm:robust} can now be combined with small gain theorems to provide simple sufficient conditions for robust stability.

\begin{coro}[Sufficient conditions for robustness]
Under the conditions of Theorem \ref{thm:robust}, the closed loop system is stable if $\max\{\|\ttf{\Delta}\|_{\mathcal{H}_\infty}, \norm{\ttf\Delta}_{\mathcal{L}_1},\norm{\ttf\Delta^\top}_{\mathcal{L}_1}\} <1$.
\label{coro:robust}
\end{coro}
%
%


It therefore follows that if a set of transfer matrices $\{\hat\Phix, \hat\Phiu, \ttf \Delta\}$ satisfy the following constraints:
\begin{equation}
\begin{array}{l}
\begin{bmatrix} zI - A & -B_2 \end{bmatrix} \begin{bmatrix} \Phixh \\ \Phiuh \end{bmatrix} = I + \ttf{\Delta} , \ \Phixh,\Phiuh \in \frac{1}{z}\RHinf, \ \ \|\ttf{\Delta}\|_{\bullet} <1,
\end{array}
\label{eq:nearly}
\end{equation}
for $\bullet \in \{ \mathcal{H}_\infty, \mathcal{\Ell}_1, \mathcal{E}_1\}$, then a controller \eqref{eq:qf} implemented using the transfer matrices $\{\Phixh, \Phiuh \}$ is globally stabilizing for the system dynamics \eqref{eq:LTI}.

%

From Theorem \ref{thm:robust}, we see that the actual closed loop map from the disturbance $\ttf \delta_x$ to state $\tf x$ and control action $\tf u$ achieved by a controller implemented using the approximate system responses $\{ \Phixh, \Phiuh \}$  is given by
\begin{equation}
\begin{bmatrix} \tf x \\ \tf u \end{bmatrix} = \begin{bmatrix} \Phixh \\ \Phiuh \end{bmatrix} (I + \ttf \Delta)^{-1} \ttf \delta_x.
\label{eq:actual_maps}
\end{equation}

Integrating the approximate achievability constraint \eqref{eq:near_sf1}, the sufficient conditions for stability imposed in by Corollary \ref{coro:robust}, and the system responses that they achieve \eqref{eq:actual_maps} into a SLS optimization problem, we obtain the following:

\begin{equation}
\begin{array}{rl}
\underset{\Phixh, \Phiuh, \ttf \Delta }{\text{minimize}} & \left\| \begin{bmatrix} C_1 & D_{12}\end{bmatrix} \begin{bmatrix} \Phixh \\ \Phiuh \end{bmatrix} (I + \ttf \Delta)^{-1}B_1  \right\| \\
\text{s.t.} & \begin{bmatrix} zI - A & - B_2 \end{bmatrix}\begin{bmatrix}\Phixh \\ \Phiuh\end{bmatrix} = I + \Delta \\
& \begin{bmatrix} \Phixh \\ \Phiuh \end{bmatrix} \in \s, \ \norm{\ttf\Delta} < 1\\
& \Phixh,\Phiuh\in \frac{1}{z}\RHinf.
\end{array}
\label{eq:opt1}
\end{equation}

Optimization problem \eqref{eq:opt1} is non-convex: however, if the norm chosen in the objective function is sub-multiplicative, we can upper bound the objective function by 
\begin{equation}
\frac{\|B_1\|}{1- \| \ttf \Delta\|}\left\| \begin{bmatrix} C_1 & D_{12}\end{bmatrix} \begin{bmatrix} \Phixh \\ \Phiuh \end{bmatrix}\right\|,
\label{eq:bound}
\end{equation}
where we have used the sub-multiplicative property, that $\|\ttf{\Delta}\|<1$ and the power-series expansion of the the inverse $(I+\ttf{\Delta})^{-1} = \sum_{k=0}^\infty \ttf{\Delta}^k$.  As the next lemma shows, this bound is indeed quasi-convex and hence can be effectively optimized.

\begin{lemma} \label{lem:innerouter}
For a convex function $f:\mathcal{D}\to\mathbb{R}$, a non-negative convex function $g:\mathcal{D}\to\mathbb{R}_+$, both defined on some domain $\mathcal{D}$, and a convex set $C\subseteq \mathcal{D}$, it holds that 
%

\begin{equation}
\min_{x\in C} \frac{f(x)}{1-g(x)} = \min_{\gamma\in[0,1)} \tfrac{1}{1-\gamma} \min_{x\in C} \{ f(x) ~|~ g(x)\leq\gamma\}
\end{equation}
\end{lemma}

\begin{proof}(Sketch)
As $g(x)$ is non-negative, we can write
\begin{align*}
\min_{x\in C} \frac{f(x)}{1-g(x)} &=  \min_{\gamma,x\in C} \frac{f(x)}{1-\gamma} ~~\text{s.t.}~~ g(x)\leq\gamma\\
&=  \min_{\gamma}\tfrac{1}{1-\gamma}\min_{x\in C} \{f(x) ~|~ g(x)\leq\gamma{}\},
\end{align*}
which proves the result.  Finally, it is easily verified that the above problem is quasi-convex, and hence uni-modal in $\gamma$.
\end{proof}

Thus applying the bound \eqref{eq:bound} and Lemma \ref{lem:innerouter}, we can compute an upper bound to optimization problem \eqref{eq:opt1} by solving the following quasi-convex problem:
\begin{equation}
\begin{array}{rl}
\min_{\gamma \in [0,1)} \frac{1}{1-\gamma}\underset{\Phixh, \Phiuh, \ttf \Delta }{\text{minimize}} & \left\| \begin{bmatrix} C_1 & D_{12}\end{bmatrix} \begin{bmatrix} \Phixh \\ \Phiuh \end{bmatrix} \right\|\\
\text{s.t.}  & \begin{bmatrix} zI  - A & - B_2 \end{bmatrix}\begin{bmatrix}\Phixh \\ \Phiuh\end{bmatrix} = I + \Delta \\
& \begin{bmatrix} \Phixh \\ \Phiuh \end{bmatrix} \in \s, \ \norm{\ttf\Delta} \leq \gamma\\
& \Phixh,\Phiuh\in \frac{1}{z}\RHinf.
\end{array}
\label{eq:opt_relaxed}
\end{equation}

We now give two examples of how this robustness result can be used.  The first is an infinite horizon analogue to the modeling error result presented in Section \ref{sec:robust-time}, and the second shows how spatiotemporal approximations to the infinite horizon SLS problem can be taken and still achieve provably near optimal performance while still benefiting from the favorable computational properties of localized synthesis.

\subsubsection{Robust control with sub-optimality guarantees}
We return to the problem setting where estimates $(\Ah, \Bh)$ of a true system $(A,B)$ satisfy
\[\|\Delta_A\|_{2\to 2}\leq\epsilon_A,~~\|\Delta_B\|_{2\to 2}\leq\epsilon_B\]
with $\Delta_A := \Ah-A$ and $\Delta_B := \Bh-B$ and where we wish to minimize the worst-case LQR cost over the parametric uncertainty.

Leveraging the results of the previous section, and following a similar argument as that in Section \ref{sec:robust-time}, we can cast a robust LQR problem as follows

\begin{align}\label{eq:robustLQRbnd}
\begin{split}
 \minimize_{\gamma\in[0,1)}\frac{1}{1 - \gamma}&\min_{\Phixh,  \Phiuh} \left\|\begin{bmatrix} Q^\frac{1}{2} & 0 \\ 0 & R^\frac{1}{2}\end{bmatrix}\begin{bmatrix} \Phixh \\  \Phiuh \end{bmatrix}\right\|_{\htwo}\\
& \text{s.t.} \begin{bmatrix}zI-\Ah&-\Bh\end{bmatrix}\begin{bmatrix} \Phixh \\  \Phiuh \end{bmatrix} = I,~~\sqrt{2}\left\|\begin{bmatrix} {\epsilon_A}  \Phixh \\ {\epsilon_B} \Phiuh \end{bmatrix} \right\|_{\hinf}\leq \gamma\\
&\qquad \Phixh,  \Phiuh  \in\frac{1}{z}\mathcal{RH}_\infty.
 \end{split}
\end{align}
We note that this optimization objective is jointly quasi-convex in $(\gamma,  \Phixh, \Phiuh)$. Hence, as a function of $\gamma$ alone the objective is quasi-convex, and furthermore is smooth in the feasible domain. Therefore, the outer optimization with respect to $\gamma$ can effectively be solved with methods like golden section search. We remark that the inner optimization is a convex problem, though an infinite dimensional one, but as we show in the next subsection, a simple finite dimensional approximation to this problem leads to nearly equivalent performance.

We further remark that because $\gamma \in [0,1)$, any feasible solution $(\Phixh,  \Phiuh)$ to optimization problem \eqref{eq:robustLQRbnd} generates a controller $\tf \Kh =  \Phiuh  \Phixh^{-1}$ satisfying the conditions of Corollary \ref{coro:robust}, and hence stabilizes the true system $(A,B)$.  Therefore, even if the solution is approximated, as long as it is feasible, it will be stabilizing.  We show now that for sufficiently small estimation error bounds $\epsilon_A$ and $\epsilon_B$, we can further bound the sub-optimality of the performance achieved by our robustly stabilizing controller relative to that achieved by the optimal LQR controller $\trueK$.

\begin{theorem}[Theorem 4.1, \cite{learning-lqr}]
\label{thm:lqr_cost}
Let $J_\star$ denote the minimal LQR cost achievable by any controller for the dynamical system with transition matrices $(A,B)$,  $J(A,B,\tf K)$ be the cost achieved by any controller $\tf K$ on the system defined by transition matrices $(A,B)$, and let $\trueK$ denote the optimal contoller. Let $(\Ah,\Bh)$ be estimates of the transition matrices such that $\norm{\Delta_A}_{2\to 2} \leq \epsilon_A$, $\norm{\Delta_B}_{2\to 2} \leq \epsilon_B$. Then, if $ \tf\Kh= \Phiuh\Phixh^{-1}$ is synthesized via \eqref{eq:robustLQRbnd} the relative error in the LQR cost is
\begin{align} \label{eq:lqr_bound}
\frac{J(\trueA, \trueB, \tf\Kh) - J_\star }{J_\star} \leq 5 (\epsilon_A + \epsilon_B\norm{\trueK}_{2 \to 2})\hinfnorm{(zI-(A+B\trueK))^{-1}} \:,
\end{align}
as long as $(\epsilon_A + \epsilon_B\norm{\trueK}_{2 \to 2})\|(zI-(A+B\trueK))^{-1}\|_{\hinf}\leq 1/5$.
\end{theorem}

This result offers a guarantee on the performance of the SLS synthesized
controller as a function of the uncertainty in our model parameters -- as far as we are aware, this is the first such result that provides a quantitative measure of performance degradation incurred by a robust controller as a function system uncertainty size.  We refer the reader to \cite{learning-lqr} for the proof this result, but remark here that Theorem \ref{thm:lqr_cost} allows one to transparently connect statistical estimation results (which provide non-asymptotic guarantees on the size of the estimation errors $\epsilon_A$ and $\epsilon_B$) with robust control, allowing for end-to-end sample complexity bounds to be obtained.  In this sense, the robust version of SLS provides a bridge between machine learning and robust control.

\subsubsection{Spatiotemporal approximations}\label{sec:firaprox}
We begin with the simpler setting of considering only an FIR approximation.  In particular, consider the general SLS optimization problem \eqref{eq:SLSsyn_inf}, where for simplicity, we assume that there are no additional constraints $\Ss$, resulting in a convex but infinite-horizon optimal control problem.  For concreteness, we will consider the following mixed $\mathcal{H}_2/\mathcal{H}_\infty$ optimal control problem -- these methods extend to the more general setting in a natural way:
\begin{equation}
\begin{array}{rl}
 \min_{\Phix, \Phiu}&  \bignorm{\begin{bmatrix} Q^{1/2} & \\ & R^{1/2}\end{bmatrix}\begin{bmatrix}\Phix \\ \Phiu\end{bmatrix}}_{\mathcal{H}_2} + \lambda \bignorm{\begin{bmatrix} Q_\infty^{1/2} & \\ & R_\infty^{1/2}\end{bmatrix}\begin{bmatrix}\Phix \\ \Phiu\end{bmatrix}}_{\mathcal{H}_\infty}\\
\st &  \begin{bmatrix} zI-A & - B_2 \end{bmatrix}\begin{bmatrix} \Phix \\ \Phiu \end{bmatrix} = I \\
& \Phix, \, \Phiu \in \frac{1}{z}\RHinf,
\end{array}\label{eq:multiobj}
\end{equation}
where $\lambda>0$ is a user-specified weighting parameter.

Although there is no known exact solution to this problem, we show here that the robustness results of the previous section can be used to formulate a finite-dimensional approximation to \eqref{eq:multiobj} that has provably near optimal performance.  In particular, we consider the following approximate SLS problem:

\begin{equation}
\begin{array}{rl}
 \min_{\gamma\in[0,1)}\frac{1}{1-\gamma}&\left(\displaystyle\min_{\Phix, \Phiu}\bignorm{\begin{bmatrix} Q^{1/2} & \\ & R^{1/2}\end{bmatrix}\begin{bmatrix}\Phix \\ \Phiu\end{bmatrix}}_{\mathcal{H}_2} + \lambda \bignorm{\begin{bmatrix} Q_\infty^{1/2} & \\ & R_\infty^{1/2}\end{bmatrix}\begin{bmatrix}\Phix \\ \Phiu\end{bmatrix}}_{\mathcal{H}_\infty}\right)\\
& \st \  \begin{bmatrix} zI-A & - B_2 & - I\end{bmatrix}\begin{bmatrix} \Phix \\ \Phiu \\ \frac{1}{z^T}V\end{bmatrix} = I \\
& \quad \quad \quad \quad \quad \norm{V}_{2\to 2} \leq \gamma, \ \Phix, \, \Phiu \in \fir.
\end{array}\label{eq:multiobj-approx}
\end{equation}

Before analyzing performance of the controller implemented using the solution to optimization problem \eqref{eq:multiobj-approx}, let us highlight the changes that have been made.  First, we have imposed that the synthesized system responses $\SFpair$ lie in $\fir$, i.e., that they be FIR of horizon $T$.  Although we argued that this was a desirable constraint to impose, as it immediately yields a finite-dimensional optimization problem, in practice, the system $(A,B_2)$ may not be controllable, meaning that such constraints are infeasible for any value of $T$.  Further, it is well known that imposing such FIR constraints, resulting in deadbeat control, can lead to robustness issues in implementation.

To circumvent these issues, we also introduce an additional term $\frac{1}{z^{T-1}}V$ to the achievability constraint: this term can be thought of as the introduction of a centralized virtual actuator that can only act at time $T-1$.   This virtual actuator ensures that the FIR constraint is in fact feasible for any value of $T\geq 1$ (simply set $V = - A\Phi_x[T-1] - B\Phi_u[T-1]$).  However, as this virtual actuator does not in fact exist, we must account for this in our analysis.  To do so, notice that the modified achievability constraint is equivalent to
\[
\begin{bmatrix} zI-A & - B_2\end{bmatrix}\begin{bmatrix} \Phix \\ \Phiu \end{bmatrix} = I + \frac{1}{z^{T-1}}V,
\]
allowing us to leverage Theorem \ref{thm:robust} and Corollaroy \ref{coro:robust}  with $\ttf \Delta = z^{-(T-1)}V$. Following a similar argument to that used above, so long as $\hinfnorm{z^{-(T-1)}V} = \norm{V}_{2\to2}<1$ (which is enforced by the constraints of the optimization problem), we can then upper bound the achieved objective function by their robust counterparts, leading to the modified objective function.

We now show that the performance of the controller synthesized using optimization problem \eqref{eq:multiobj-approx} converges exponentially quickly to that achieved by the optimal solution to the infinite horizon problem \eqref{eq:multiobj} with respect to the approximation horizon $T$.

In particular, let $\Phix^\star$ and $\Phiu^\star$ denote the optimal solutions to problem \eqref{eq:multiobj}, and let $\Kstar = \Phiu^\star(\Phix^\star)^{-1}$ be the resulting optimal controller.  As $\Phix^\star = \sum_{t=1}^\infty z^{-t}\Phi_x[t] \in \frac{1}{z}\RHinf$, there exist constants $C_\star$ and $\rho_\star$ such that $\norm{\Phi_x^\star[t]}_{2\to 2}\leq C_\star \rho_\star^t$ for all $t\geq 0$.  We leverage this observation to construct a feasible solution to \eqref{eq:multiobj-approx} using $\Phix^\star$ and $\Phiu^\star$ to obtain a sub-optimality bound in terms of the horizon $T$.

\begin{theorem}
\label{thm:fir-approx}
If the approximation horizon $T$ is sufficiently large to ensure that $C_\star\rho_\star^T <1$,
then the approximate SLS problem \eqref{eq:multiobj-approx} is feasible, and achieves the relative performance bound:
\begin{equation}
J_T(\Phix,\Phiu,V,\gamma)\leq \frac{1}{1-C_\star\rho_\star^T}J_\star + \frac{\lambda C_\star\rho_\star^T}{(1-C_\star\rho_\star^T)(1-\rho_\star^T)}(\norm{Q_\infty^{1/2}}_{2 \to 2} + \norm{R_\infty^{1/2}}_{2 \to 2}\hinfnorm{\Kstar}).
\end{equation}
Here $J_T(\Phix,\Phiu,V,\gamma)$ is the performance achieved by the controller synthesized using the approximate SLS problem \eqref{eq:multiobj-approx} with horizon $T$, and $J_\star$ is the performance achieved by the optimal solution to \eqref{eq:multiobj}.
\end{theorem}

Before presenting the proof of this result, we comment on its implications.  First, we see very clearly that the gap between $J_T$ and $J_\star$ decays exponentially with the horizon $T$, and further, this rate is specified by $\rho_\star$, which governs the spectral norm decay of true system response elements.  Informally, this theorem shows that if the optimal controller leads to good system performance (i.e., to rapid decay of the system response elements), then it is easy to approximate.  

\begin{proof}
Let $(\Phix,\Phiu,V,\gamma)$ denote the optimal solutions to the approximate SLS problem \eqref{eq:multiobj-approx}, and denote by $J_T(\Phix,\Phiu,V,\gamma)$ the value of the objective function over a finite horizon $T$. 

We now construct a feasible solution to to the problem using the optimal system responses $\Phix^\star$ and $\Phiu^\star$.  In particular, we claim that
\begin{align*}
\Phixtil &:= \Phix^\star(1:T-1) \\
 \Phiutil &:= \Phiu^\star(1:T-1) \\
\tilde V & :=  -A\Phi^\star_x[T-1] - B\Phi^\star_u[T-1] = -\Phi^\star_x[T] \\
\tilde \gamma & := C_\star \rho_\star^T,
\end{align*}
is a feasible solution.  First, notice that for $t=0,...,T-2$, the achievability constraint~\eqref{eq:affine_cons} is equivalent to
\[
\Phi_x[t+1] = A\Phi_x[t] + B_2 \Phi_u[t].
\]
Since $\{\Phixtil,\Phiutil\}$ are simply truncations of $\{\Phix^\star,\Phiu^\star\}$ to a horizon of $T-1$, they satisfy these constraints by construction.  Finally, for $t=T-1$, we have that
\[
\Phi_x[T] = A\Phi^\star_x[T-1] + B\Phi^\star_u[T-1] + \tilde V = 0,
\]
where the final equality follows from our choice of $\tilde V$ to effectively cancel out the tail of optimal response.  Finally, note that $\norm{\tilde V}_{2\to 2} = \norm{\Phi_x^\star[T]}_{2\to 2} \leq C_\star\rho_\star^{T} < 1$, where the final inequality follows from our assumption on $T$.

Therefore, by optimality, it follows that
\begin{align*}
J_T(\Phix,\Phiu,V,\gamma) &\leq J_T(\Phixtil,\Phiutil,\tilde V,\tilde\gamma)\\
& = \frac{1}{1-C_\star\rho_\star^T}\left(\bignorm{\begin{bmatrix} Q^{1/2} & \\ & R^{1/2}\end{bmatrix}\begin{bmatrix}\Phixtil \\ \Phiutil \end{bmatrix}}_{\mathcal{H}_2} + \lambda \bignorm{\begin{bmatrix} Q_\infty^{1/2} & \\ & R_\infty^{1/2}\end{bmatrix}\begin{bmatrix}\Phixtil \\ \Phiutil \end{bmatrix}}_{\mathcal{H}_\infty}\right).
\end{align*}
Our final step is to relate the norms achieved by the FIR approximations $\Phix^\star(0:T-1)$ and $\Phiu^\star(0:T-1)$ to those achieved by $\Phix^\star$ and $\Phiu^\star$.  We begin by noting that for any stable transfer matrix $\tf M \in z^{-1}\RHinf$, we have that $\htwonorm{\tf M(1:T)} \leq \htwonorm{\tf M}$, for any $T\geq 0$.  The $\mathcal{H}_\infty$ case is slightly more cumbersome, as such an inequality does not hold exactly, but notice that we can write
\begin{align*}
\bignorm{\begin{bmatrix} Q_\infty^{1/2} & \\ & R_\infty^{1/2}\end{bmatrix}\begin{bmatrix}\Phix^\star(1:T-1) \\ \Phiu^\star(1:T-1)\end{bmatrix}}_{\mathcal{H}_\infty} & \leq \bignorm{\begin{bmatrix} Q_\infty^{1/2} & \\ & R_\infty^{1/2}\end{bmatrix}\begin{bmatrix}\Phix^\star \\ \Phiu^\star\end{bmatrix} }_{\mathcal{H}_\infty} + \bignorm{\begin{bmatrix} Q_\infty^{1/2} & \\ & R_\infty^{1/2}\end{bmatrix}\begin{bmatrix}\Phix^\star(T:\infty) \\ \Phiu^\star(T:\infty),\end{bmatrix} }_{\mathcal{H}_\infty}
\end{align*}
where the inequality follows from the triangle inequality.  Finally, notice that last term can further be be bounded as
\begin{align*}
\bignorm{\begin{bmatrix} Q_\infty^{1/2} & \\ & R_\infty^{1/2}\end{bmatrix}\begin{bmatrix}\Phix^\star(T:\infty) \\ \Phiu^\star(T:\infty),\end{bmatrix} }_{\mathcal{H}_\infty} & \leq (\norm{Q_\infty^{1/2}}_{2\to 2} + \norm{R_\infty^{1/2}}_{2\to 2}\hinfnorm{\Kstar})\hinfnorm{\Phix^\star(T:\infty)} \\
&\leq (\norm{Q_\infty^{1/2}}_{2\to 2} + \norm{R_\infty^{1/2}}_{2\to 2}\hinfnorm{\Kstar})\sum_{t=T}^\infty \norm{\Phi_x^\star(t)}_{2\to 2} \\
&\leq (\norm{Q_\infty^{1/2}}_{2\to 2} + \norm{R_\infty^{1/2}}_{2\to 2}\hinfnorm{\Kstar})C_\star\sum_{t=T}^\infty \rho_\star^t \\
& = (\norm{Q_\infty^{1/2}}_{2\to 2} + \norm{R_\infty^{1/2}}_{2\to 2}\hinfnorm{\Kstar})\frac{C_\star\rho_\star^T}{1-\rho_\star^T}.
\end{align*}
Combining this bound with the previous inequalities then yield the result.
\end{proof}

We conclude this section by remarking that the above proof contains a general recipe for generating such suboptimality bounds.  Begin with the desired infinite horizon problem, and then write out its FIR approximation by augmenting the system with virtual actuation, and then exploiting the robustness results in Theorem \ref{thm:robust} and Corollary \ref{coro:robust}.  Then, construct a feasible solution to the approximate problem using the optimal solution to the infinite horizon problem and follow the general approach of the proof above to obtain a suboptimality bound by exploiting the exponential decay of stable LTI systems.  Finally, to extend these results from temporal (FIR) approximations to spatiotemporal (localized) approximations, it suffices to impose a localized SLC on the system responses, and further to assume that the optimal solutions also suitably decay in space according to some function $f(d)$-- the resulting sub-optimality bounds will be similar, but now decaying exponentially as a function of the horizon $T$, and according to $f(d)$ with respect to the localized region diameter $d$.


{
\subsubsection{Examples}
Now that we have established a theoretical robustness result we shall use this subsection to illustrate the effect of Theorem~\ref{thm:robust} when $\|\ttf \Delta\|\neq 0$. Consider the symmetric bi-directional chain topology, where each node in the chain is a scalar LTI system evolving according to the dynamics
\begin{equation*}
x_i[t+1] = \alpha(x_i[t] + \kappa x_{i-1}[t]+\kappa x_{i+1}[t])+b_iu_i[t]+w_i[t],
\end{equation*}
where $\alpha$ is chosen to make the open-loop system unstable with a spectral radius of $1.1$. The coupling strength $\kappa$ is set to 1.  The value $b_i$ is given by 1 if there is an actuator at subsystem $i$, and $0$ otherwise. We place $40$ actuators in the  network at locations specified by $i = 5j - 4$ and $5j$ for $j = 1,\hdots ,20$. The objective function is the quadratic term $\|\tf x\|_2^2 + \gamma \|\tf u\|_2^2$. 

We begin by showing the system response for a system with actuation at every other node. By choosing $(d,T)$-locality constraints with $d=4$, $T=15$ and restricting the communication speed to be equal to the speed at which the plant dynamics propagate, the system response cannot be perfectly localized. However, by employing Theorem~\ref{thm:robust}, the robust SLS problem finds a solution such that $\|\ttf \Delta\|<1$. The state and control action from the resulting controller are shown in Figure~\ref{fig:eg3}.

\begin{figure*}[h!]
      \centering
      \includegraphics[width=0.6\textwidth]{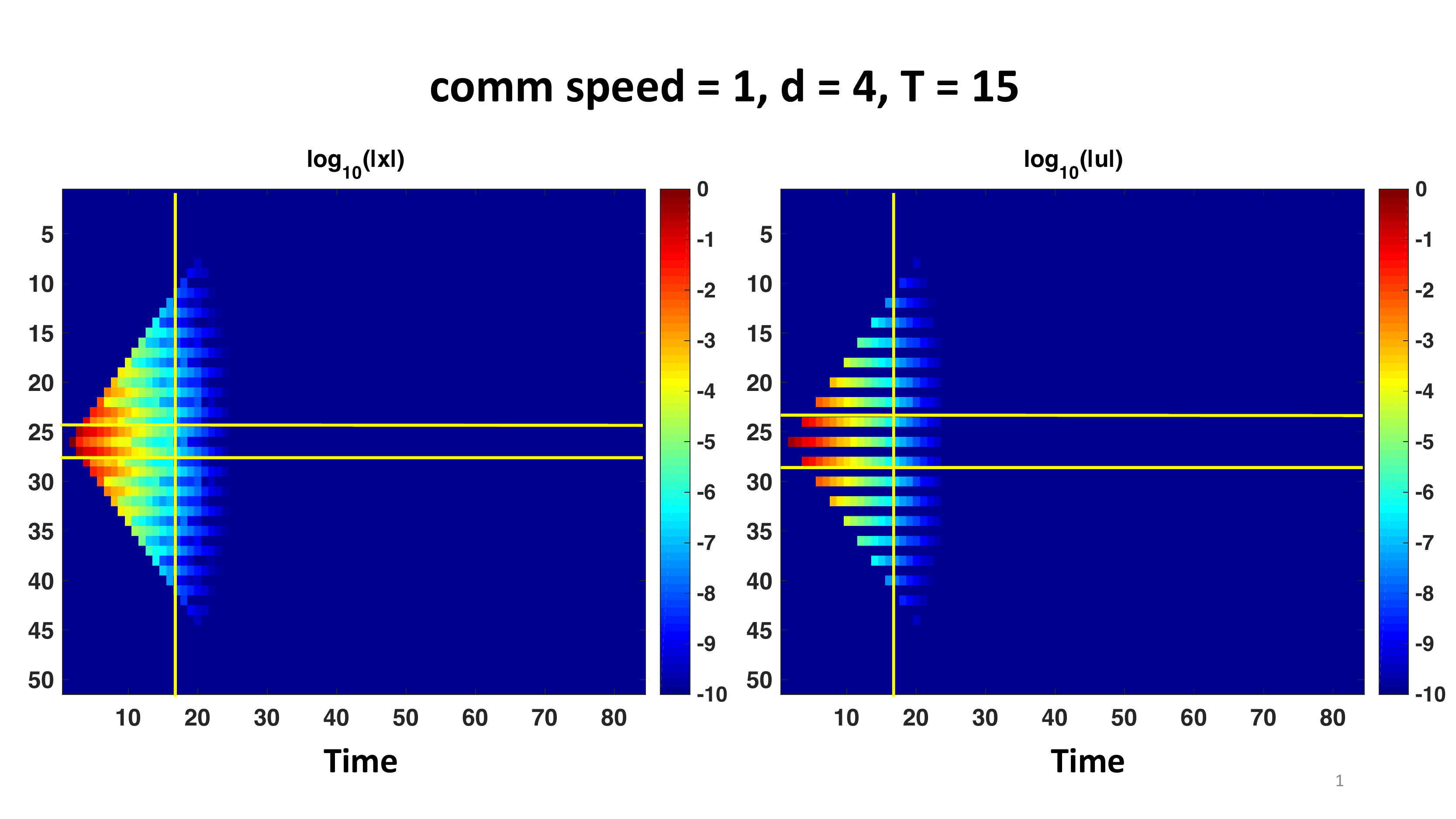}
      \caption{With the dynamics propagating at the same speed as the controllers communicate and with such sparse actuation, we are now unable to localize in space and time as desired. However, the control action and state deviation is reasonably well contained.}
      \label{fig:eg3}
\end{figure*}

In subsequent experiments shown in Figures~\ref{fig:eg5}--\ref{fig:eg7} the system responses are shown for systems where we trade off the actuator density (Fig~\ref{fig:eg5}), $d$-hop locality (Figure~\ref{fig:eg6}), and FIR horizon length (Figure~\ref{fig:eg7}). In each case we show a parameter value that results in a localizable system response (shown on the left) and the virtually localizable response (shown on the right).

\begin{figure*}[h!]
      \centering
      \includegraphics[width=0.9\textwidth]{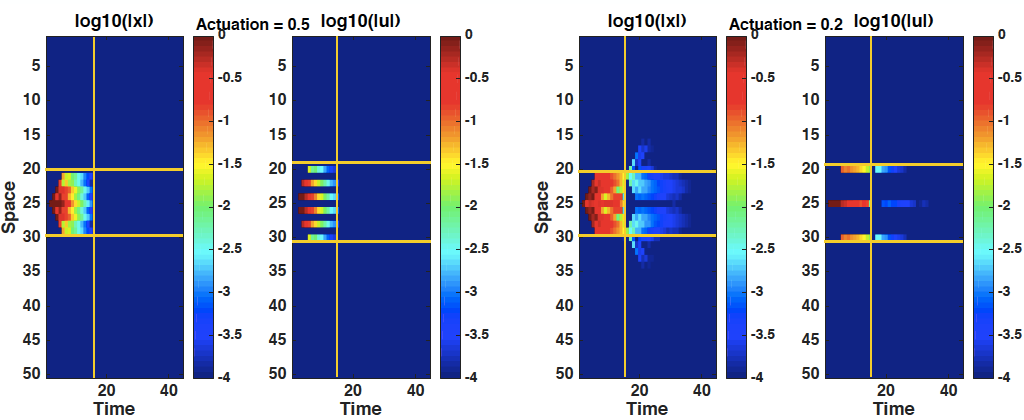}
      \caption{As the number of actuators available for control decreases it becomes more difficult to obtain a localized response. The $\mathcal{H}_2$-norm for the actuation = 0.5 is $79.5$, and for actuation = $0.2$ it is $97.52$ with, $\|\Delta\|_{\mathcal{E}_1}=0.25$. }
      \label{fig:eg5}
\end{figure*}

\begin{figure*}[h!]
      \centering
      \includegraphics[width=0.9\textwidth]{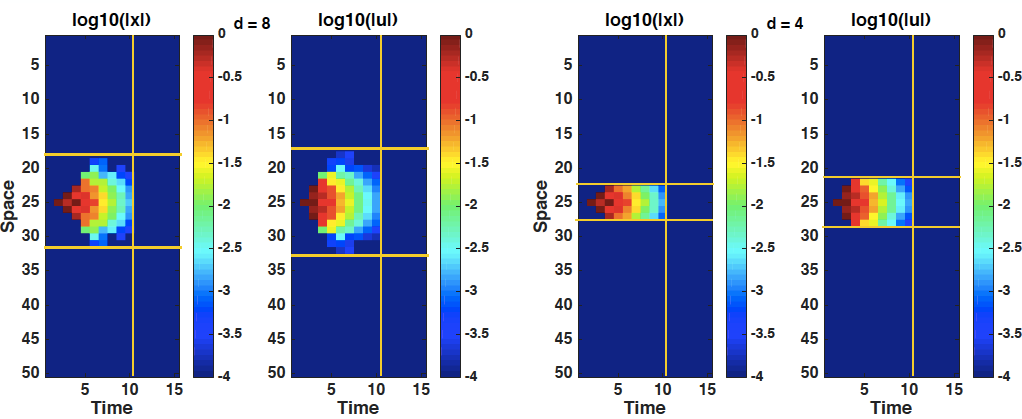}
      \caption{The $d$-hop sparsity is decreased from $d=8$ to $d=4$. The closed loop $\mathcal{H}_2$-norm is $72.57$ and $72.59$ for $d= 8$ and $d=4$ respectively. No virtual localization was required.}
      \label{fig:eg6}
\end{figure*}

\begin{figure*}[h!]
      \centering
      \includegraphics[width=0.9\textwidth]{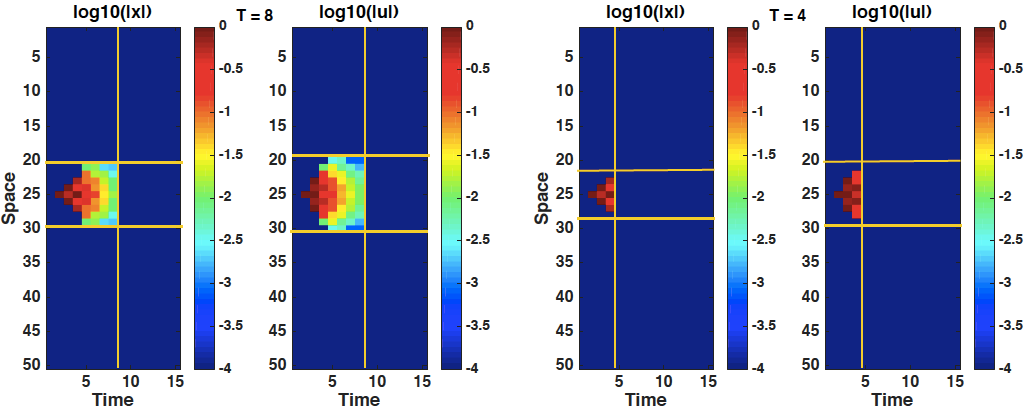}
      \caption{The FIR Horizon length is decreased from $T=8$ to $T=4$. In neither case was virtual actuation required. The $\mathcal{H}_2$-norm was 72.57 for $T=8$ and 72.96 for $T=4$.}
      \label{fig:eg7}
\end{figure*}

}



\subsection{Summary}
In this section, we extended the state-feedback finite-time horizon results from Section \ref{sec:sys_resp_intro} to the infinite horizon settings.  We further showed that the SLS framework allows us to impose convex locality constraints on distributed controllers, greatly improving the scalability of synthesis methods via a decomposition and dimensionality reduction based method.  We also showed how robust counterparts to SLS problems can be used to transparently accommodate modeling errors, as well as how they can be used to derive finite dimensional approximations to infinite horizon problems with provable performance guarantees. In addition to the results presented here, there has also been interesting work looking at linking the SLS state-feedback formulation to classical notions such as passivity~\cite{Ran18} and infinite-dimensional system theory and the spatial invariance framework~\cite{JenB18}. In~\cite{AndMC19} the $\ttf \Delta$-block is  used to model discretization errors when a sparse approximate model is constructed from a structured continuous-time dynamics.

\section{Output Feedback}
\label{sec:output}
This section extends many of the ideas presented in the previous section to the output feedback setting.  As we will see, many of the core concepts developed in the state-feedback setting (working directly with system responses, affine achievability constraints) extend in a natural way to the output feedback setting.  However, these extensions often come at the expense of more complex formulae and computational procedures.

\subsection{Output Feedback with $D_{22} = 0$} \label{sec:of}

We begin by extending Theorem \ref{thm:param} to the output feedback setting.  Without loss of generality (see later in the section), we consider a strictly proper plant described by
\begin{equation}
\tf P = \left[ \begin{array}{c|cc} A & B_1 & B_2 \\ \hline C_1 & D_{11} & D_{12} \\ C_2 & D_{21} & 0 \end{array} \right]. \label{eq:ofplant}
\end{equation}

Letting $\delta_x [t] = B_1 w[t]$ denote the disturbance on the state, and $\delta_y [t] = D_{21} w[t]$ denote the disturbance on the measurement, the dynamics defined by plant \eqref{eq:ofplant} can be written as
\begin{eqnarray}
x[t+1] &=& A x[t] + B_2 u[t] + \delta_x [t] \nonumber \\
y[t] &=& C_2 x[t] + \delta_y [t]. \label{eq:sys_out}
\end{eqnarray}

Analogous to the state-feedback case, we define a system response $\{\Phixx, \Phiux, \Phixy, \Phiuy\}$ from perturbations $(\ttf{\delta_x}, \ttf{\delta_y})$ to state and control inputs $(\tf x,\tf u)$ via the following relation:
\begin{equation}
\begin{bmatrix} \tf x \\ \tf u \end{bmatrix} = \begin{bmatrix} \Phixx & \Phixy \\ \Phiux & \Phiuy \end{bmatrix} \begin{bmatrix} \ttf{\delta_x} \\ \ttf{\delta_y} \end{bmatrix}. \label{eq:cltm_df}
\end{equation}

Substituting the output feedback control law $\tf u = \tf K \tf y$ into the z-transform of system equation \eqref{eq:sys_out}, we obtain
\begin{equation}
(zI - A - B_2 \tf K C_2) \tf x = \ttf{\delta_x} + B_2 \tf K \ttf{\delta_y}. \nonumber
\end{equation}
For a proper controller $\tf K$, the transfer matrix $(zI-A-B_2 \tf K C_2)$ is always invertible, hence we obtain the following \rev{equivalent} expressions for the system response \eqref{eq:cltm_df} in terms of an output feedback controller $\tf K$:
\begin{align}
\Phixx &= (zI - A - B_2 \tf K C_2)^{-1} \nonumber\\
\Phiux &= \tf K C_2 \Phixx \nonumber\\
\Phixy &= \Phixx B_2 \tf K \nonumber\\
\Phiuy  &= \tf K + \tf K C_2 \Phixx B_2 \tf K. \label{eq:K_relation}
\end{align}

We now present one of the main results of this section: an algebraic characterization of the set $\{\Phixx, \Phiux, \Phixy, \Phiuy\}$ of output-feedback system responses that are achievable by an internally stabilizing controller $\tf K$.
\begin{figure}[ht!]
      \centering
      \includegraphics[width=0.48\textwidth]{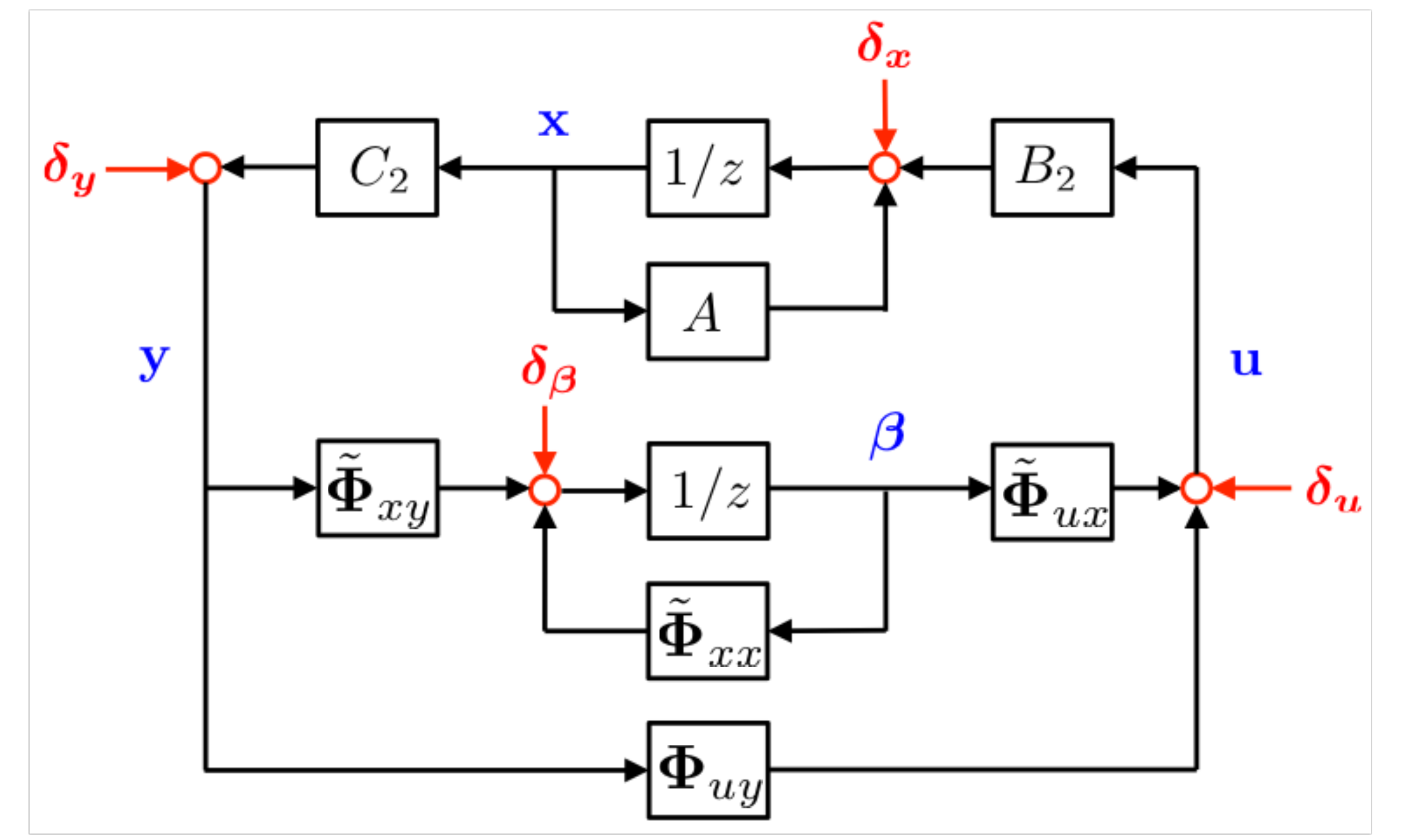}
      \caption{The proposed output feedback controller structure, with $\tf{\tilde\Phi}_{xx} =  z(I - z \Phixx)$, $\tf{\tilde\Phi}_{ux} = z \Phiux$, and $\tf{\tilde\Phi}_{xy} = -z \Phixy$.}
      \label{fig:of}
\end{figure}
\begin{theorem}
For the output feedback system \eqref{eq:ofplant}, the following are true:
\begin{enumerate}[(a)]
  \item The affine subspace described by:
  \begin{subequations} \label{eq:output_fb}
\begin{align}
\begin{bmatrix} zI - A & -B_2 \end{bmatrix}
\begin{bmatrix} \Phixx & \Phixy \\ \Phiux & \Phiuy \end{bmatrix} &= 
\begin{bmatrix} I & 0 \end{bmatrix} \label{eq:output_fb1}\\
\begin{bmatrix} \Phixx & \Phixy \\ \Phiux & \Phiuy \end{bmatrix}
 \begin{bmatrix} zI - A \\ -C_2 \end{bmatrix} &= 
\begin{bmatrix} I \\ 0 \end{bmatrix} \label{eq:output_fb2} \\
\Phixx, \Phiux, \Phixy \in \frac{1}{z} \mathcal{RH}_\infty, \quad & \Phiuy \in \mathcal{RH}_\infty \label{eq:output_fb3}
\end{align}
\end{subequations} 
 parameterizes all system responses \eqref{eq:K_relation} achievable by an internally stabilizing controller $\tf K$. \label{thm2-4}
  \item For any transfer matrices $\{\Phixx, \Phiux, \Phixy, \Phiuy\}$ satisfying \eqref{eq:output_fb}, the controller $\tf K = \Phiuy - \Phiux \Phixx^{-1} \Phixy$ achieves the desired response \eqref{eq:K_relation}.\rev{\footnote{\rev{Note that for any transfer matrices $\{\Phixx, \Phiux, \Phixy, \Phiuy\}$ satisfying \eqref{eq:output_fb}, the transfer matrix $\Phixx$ is always invertible because its leading spectral component $\frac{1}{z} I$ is invertible.  The same holds true for the transfer matrices defined in equation \eqref{eq:K_relation}.}}} \revsecond{Further, if the controller is implemented as in Fig. \ref{fig:of}, then it is internally stabilizing.} \label{thm2-3} 
  \end{enumerate} \label{thm:of}
\end{theorem}

As in the state-feedback case, although not explicitly enforced, one cannot circumvent the need for a stabilizable and detectable system, as formalized in the following lemma.

\begin{lemma}\label{lem:stab_det}
The triple $(A, B_2, C_2)$ is stabilizable and detectable if and only if the affine subspace described by \eqref{eq:output_fb} is non-empty.
\end{lemma}
\begin{proof}
We begin by noting the analysis for the state feedback problem in Section \ref{sec:state_feedback_SLS} can be applied to the state estimation problem by considering the dual to a full control system (c.f.,  \S 16.5 in \cite{Zhou1996robust}). For instance, the following corollary to Lemma \ref{lem:stabilizable} gives an alternative definition of the detectability of pair $(A,C_2)$ \cite{2015_Wang_LDKF}.
\begin{coro}
The pair $(A, C_2)$ is detectable if and only if the following conditions are feasible:
\begin{subequations} \label{eq:state_est}
\begin{align}
& \begin{bmatrix} \Phixx & \Phixy \end{bmatrix} \begin{bmatrix} zI - A \\ -C_2 \end{bmatrix} = I \label{eq:state_est1}\\
& \Phixx, \Phixy \in \frac{1}{z} \mathcal{RH}_\infty. \label{eq:state_est2}
\end{align}
\end{subequations} \label{cor:1}.
\end{coro}

Now assume the triple $(A,B_2,C_2)$ is stabilizable and detectable.  As argued in Lemma \ref{lem:stabilizable} and Corollary \ref{cor:1}, this implies that there exist stable transfer matrices $(\Phixx^1,\Phiux^1)$ satisfying \eqref{eq:affine_cons}, and stable transfer matrices $(\Phixx^2,\Phixy^2)$ satisfying \eqref{eq:state_est}.  Using these transfer matrices, we construct the  stable transfer matrices, which can be verified to lie in the affine space \eqref{eq:output_fb} 
\begin{subequations}
\begin{align}
\Phixx &= \Phixx^1 + \Phixx^2 - \Phixx^1 (zI-A) \Phixx^2 \label{eq:sfest1}\\
\Phiux &= \Phiux^1 - \Phiux^1 (zI-A) \Phixx^2 \label{eq:sfest2}\\
\Phixy &= \Phixy^2 - \Phixx^1 (zI-A) \Phixy^2 \label{eq:sfest3}\\
\Phiuy &= -\Phiux^1 (zI-A) \Phixy^2. \label{eq:sfest4}
\end{align}\label{eq:sfest}
\end{subequations}
\end{proof}

This lemma provides a sanity check, and in particular tells us that the set of achievable stable system responses for the triple $(A,B_2,C_2)$ is non-empty if and only if the system is stabilizable and detectable.

\begin{proof}[Proof of Theorem \ref{thm:of}]

Proof of 1. 
Consider the output feedback system \eqref{eq:ofplant}. Let $\{ \Phixx, \Phiux, \Phixy, \Phiuy\}$, with $\tf x = \Phixx \ttf{\delta_x} + \Phixy \ttf{\delta_y}$ and  $\tf u = \Phiux \ttf{\delta_x} + \Phiuy \ttf{\delta_y}$, be the system response achieved by an internally stabilizing control law $\tf u = \tf K \tf y$. Then $\{\Phixx, \Phiux, \Phixy, \Phiuy\}$ lies in the affine subspace described by \eqref{eq:output_fb}.

To see this, let the internally stabilizing controller $\tf K$ have state space realization \eqref{eq:Kss}.
Combining \eqref{eq:Kss} with the system equation \eqref{eq:sys_out}, we obtain the closed loop dynamics
\begin{equation}
\begin{bmatrix} z \tf x \\ z \ttf{\xi} \end{bmatrix} = \begin{bmatrix} A + B_2 D_k C_2 & B_2 C_k \\ B_k C_2 & A_k \end{bmatrix} \begin{bmatrix} \tf x \\ \ttf{\xi} \end{bmatrix} + \begin{bmatrix} I & B_2 D_k \\ 0 & B_k \end{bmatrix} \begin{bmatrix} \ttf{\delta_x} \\ \ttf{\delta_y} \end{bmatrix}. \nonumber
\end{equation}
From the assumption that $\tf K$ is internally stabilizing, we know that the state matrix of the above equation is a stable matrix (Lemma $5.2$ in \cite{Zhou1996robust}). The system response achieved by $\tf u = \tf K \tf y$ is given by 
\begin{equation}
\begin{bmatrix} \Phixx & \Phixy \\ \Phiux & \Phiuy \end{bmatrix} = \left[ \begin{array}{cc | cc} A + B_2 D_k C_2 & B_2 C_k & I & B_2 D_k \\ B_k C_2 & A_k & 0 & B_k \\ \hline I & 0 & 0 & 0 \\ D_k C_2 & C_k & 0 & D_k \end{array} \right], \label{eq:ss_relation}
\end{equation}
which satisfies \eqref{eq:output_fb3}.
In addition, it can be shown by routine calculation that \eqref{eq:ss_relation} satisfies both \eqref{eq:output_fb1} and \eqref{eq:output_fb2}.

Proof of 2.
Here we show that for any system response $\{\Phixx, \Phiux, \Phixy, \Phiuy\}$ lying in the affine subspace \eqref{eq:output_fb}, we can construct an internally stabilizing controller $\tf K$ that leads to the desired system response \eqref{eq:K_relation}. From the relations in \eqref{eq:K_relation}, we notice the identity $\tf K = \Phiuy - \tf K C_2 \Phixx B_2 \tf K = \Phiuy - \Phiux \Phixx^{-1} \Phixy$. This relation leads to the controller structure given in Figure \ref{fig:of}, with $\tf{\tilde\Phi}_{xx}= z (I - z \Phixx)$, $\tf{\tilde\Phi}_{ux} = z \Phiux$, and $\tf{\tilde\Phi}_{xy} = -z \Phixy$. As was the case for the state feedback setting, it can be verified that all three of these transfer matrices lie in $\RHinf$. Therefore, the structure given in Figure \ref{fig:of} is well defined. In addition, all of the blocks in Figure \ref{fig:of} are stable filters -- thus, as long as the origin $(x,\beta) = (0,0)$ is asymptotically stable, all signals internal to the block diagram will decay to zero. To check the internal stability of the structure, we introduce external perturbations $\ttf{\delta_x}, \ttf{\delta_y}$, $\ttf{\delta_u}$, and $\ttf{\delta_\beta}$ to the system. The perturbations appearing on other links of the block diagram can all be expressed as a combination of the perturbations $(\ttf{\delta_x}, \ttf{\delta_y}, \ttf{\delta_u}, \ttf{\delta_\beta})$ being acted upon by some stable transfer matrices, and so it suffices to check the input-output stability of the closed loop transfer matrices from perturbations $(\ttf{\delta_x}, \ttf{\delta_y}, \ttf{\delta_u}, \ttf{\delta_\beta})$ to controller signals $(\tf x, \tf u, \tf y, \ttf{\beta})$ to determine the internal stability of the structure \cite{Zhou1996robust}.

For any system response $\{\Phixx, \Phiux, \Phixy, \Phiuy\}$ lying in the affine subspace defined by \eqref{eq:output_fb}, we construct a controller using the structure given in Figure \ref{fig:of}. 
We now check the stability of the closed loop transfer matrices from the perturbations $(\ttf{\delta_x}, \ttf{\delta_y}, \ttf{\delta_u}, \ttf{\delta_\beta})$ to the internal variables $(\tf x, \tf u, \tf y, \ttf{\beta})$. We have the following equations from Figure \ref{fig:of}:
\begin{eqnarray}
z \tf x &=& A \tf x + B_2 \tf u + \ttf{\delta_x} \nonumber\\
\tf y &=& C_2 \tf x + \ttf{\delta_y} \nonumber\\
z \ttf{\beta} &=& \tf{\tilde\Phi}_{xx} \ttf{\beta} +\tf{\tilde\Phi}_{xy}  \tf y + \ttf{\delta_\beta} \nonumber\\
\tf u &=& \tf{\tilde\Phi}_{ux} \ttf{\beta} + \Phiuy \tf y + \ttf{\delta_u}. \nonumber
\end{eqnarray}
Combining these equations with the relations in \eqref{eq:output_fb1}--\eqref{eq:output_fb2}, we summarize the closed loop transfer matrices from $(\ttf{\delta_x}, \ttf{\delta_y}, \ttf{\delta_u}, \ttf{\delta_\beta})$ to $(\tf x, \tf u, \tf y, \ttf{\beta})$ in Table \ref{Table:1}.

\begin{table}[t!]
 \caption{Closed Loop Maps from Perturbations to Internal Variables}
 \label{Table:1}
\begin{center}
\renewcommand{\arraystretch}{2}
    \begin{tabular}{| c | c | c | c | c | c |}
    \hline 
    & $\ttf{\delta_x}$ & $\ttf{\delta_y}$ & $\ttf{\delta_u}$ & $\ttf{\delta_\beta}$ \\ \hline
    $\tf x$ & $\Phixx$ & $\Phixy$ & $\Phixx B_2$ & $\frac{1}{z} \Phixy C_2$ \\ \hline
    $\tf u$ & $\Phiux$ & $\Phiuy$ & $I + \Phiux B_2$ & $\frac{1}{z} \Phiuy C_2$ \\ \hline
    $\tf y$ & $C_2 \Phixx$ & $I + C_2 \Phixy$ & $C_2 \Phixx B_2$ & $\frac{1}{z} C_2 \Phixy C_2$ \\ \hline
    $\ttf{\beta}$ & $-\frac{1}{z} B_2 \Phiux$ & $-\frac{1}{z} B_2 \Phiuy$ & $-\frac{1}{z} B_2 \Phiux B_2$ & $\frac{1}{z} I - \frac{1}{z^2} (A + B_2 \Phiuy C_2)$ \\ \hline
    \end{tabular}
\end{center}
\end{table}

Equation \eqref{eq:output_fb3} implies that all sixteen transfer matrices in Table \ref{Table:1} are stable, so the implementation in Figure \ref{fig:of} is internally stable. Furthermore, the desired system response from $(\ttf{\delta_x}, \ttf{\delta_y})$ to $(\tf x,\tf u)$ is achieved.
\end{proof}

The controller implementation of Figure \ref{fig:of} is governed by the following equations:
\begin{eqnarray}
z \ttf{\beta} &=&  \tf{\tilde\Phi}_{xx} \ttf{\beta} +\tf{\tilde\Phi}_{xy} \tf y \nonumber\\
\tf u &=& \tf{\tilde\Phi}_{ux}\ttf{\beta} + \Phiuy \tf y, \label{eq:ss_like}
\end{eqnarray}
which can be informally interpreted as an extension of the state-space realization \eqref{eq:Kss} of a controller $\tf K$. In particular, the realization equations \eqref{eq:ss_like} can be viewed as a state-space like implementation where the constant matrices $A_K, B_K, C_K, D_K$ of the state-space realization \eqref{eq:Kss} are replaced with stable proper transfer matrices $ \tf{\tilde\Phi}_{xx} ,  \tf{\tilde\Phi}_{ux}, \tf{\tilde\Phi}_{xy}, \Phiuy$.  The benefit of this implementation is that arbitrary convex constraints imposed on the transfer matrices $ \tf{\tilde\Phi}_{xx} ,  \tf{\tilde\Phi}_{ux}, \tf{\tilde\Phi}_{xy}, \Phiuy$. carry over directly to the controller implementation.  We show in Section \ref{sec:output-locality}, as in the state-feedback setting, that this transparency allows for a class of structural (locality) constraints to be imposed on the system response (and hence the controller) that are crucial for extending controller synthesis methods to large-scale systems. \rev{In contrast, we recall that imposing general convex constraints on the controller $\tf K$ or directly on its state-space realization $A_K, B_K, C_K, D_K$ do not lead to convex optimal control problems.}

\revsecond{\begin{remark}
The controller implementation \eqref{eq:ss_like} admits the following equivalent representation
\begin{equation}
\begin{bmatrix} \Phixx & \Phixy \\ \Phiux & \Phiuy \end{bmatrix} \begin{bmatrix}z \ttf{\beta} \\ \tf y \end{bmatrix} = \begin{bmatrix} 0 \\ \tf u\end{bmatrix}, \label{eq:ros_rep}
\end{equation}
allowing for an interesting interpretation of the controller $\tf K = \Phiuy - \Phiux \Phixx^{-1} \Phixy$ in terms of Rosenbrock system matrix representations \cite{Rosenbrock}. In particular, the system response \eqref{eq:cltm_df} specifies a Rosenbrock system matrix representation of the controller that achieves it.
\end{remark}}

\rev{\subsection{Specialized Implementations for Open-loop Stable Systems}}
\rev{In this subsection, we propose two specializations of the controller implementation in Figure \ref{fig:of} for open loop stable systems.} From Table \ref{Table:1}, if we set $\ttf{\delta_u}$ and $\ttf{\delta_\beta}$ to $0$, it follows that $\ttf{\beta} = -\frac{1}{z}B_2 \tf u$. 
This leads to a simpler controller implementation given by $\tf u = \Phiuy \tf y - \Phiux B_2 \tf u$, with the corresponding controller structure shown in Figure \ref{fig:alt1}. 
This implementation can also be obtained from the identity $ \tf K = (I + \Phiux B_2)^{-1} \Phiuy$, which follows from the relations in \eqref{eq:K_relation}. 
Unfortunately, as shown below, this implementation is internally stable only when the open loop plant is stable.  

For the controller implementation and structure shown in Figure \ref{fig:alt1}, the closed loop transfer matrices from perturbations to the internal variables are given by
\begin{equation}
\begin{bmatrix} \tf x \\ \tf u \end{bmatrix} = \begin{bmatrix} \Phixx & \Phixy & \Phixx B_2 & (zI-A)^{-1} B_2 \\ \Phiux & \Phiuy & I + \Phiux B_2 & I \end{bmatrix} \begin{bmatrix} \ttf{\delta_x} \\ \ttf{\delta_y} \\ \ttf{\delta_u} \\ \ttf{\delta_\beta} \end{bmatrix}. \label{eq:alt1}
\end{equation}
When $A$ defines a stable system, the implementation in Figure \ref{fig:alt1} is internally stable. However, when the open loop plant is unstable (and the realization $(A,B_2)$ is stabilizable), the transfer matrix $(zI-A)^{-1} B_2$ is unstable. From \eqref{eq:alt1}, the effect of the perturbation $\ttf{\delta_\beta}$ can lead to instability of the closed loop system. This structure thus shows the necessity of introducing and analyzing the effects of perturbations $\ttf{\delta_\beta}$ on the controller internal state.

Alternatively, if we start with the identity $\tf K = \Phiuy (I + C_2 \Phixy)^{-1}$, which also follows from \eqref{eq:K_relation}, we obtain the controller structure shown in Figure \ref{fig:alt2}. The closed loop map from perturbations to internal signals is then given by
\begin{equation}
\begin{bmatrix} \tf x \\ \tf u \\ \ttf{\beta} \end{bmatrix} = \begin{bmatrix} \Phixx & \Phixy & \Phixx B_2 \\ \Phiux & \Phiuy & I + \Phiux B_2 \\ C_2 (zI-A)^{-1} & I & C_2 (zI-A)^{-1} B_2 \end{bmatrix} \begin{bmatrix} \ttf{\delta_x} \\ \ttf{\delta_y} \\ \ttf{\delta_u} \end{bmatrix}. \nonumber
\end{equation}
As can be seen, the controller implementation is once again internally stable only when the open loop plant is stable (if the realization $(A,C_2)$ is detectable). This structure thus shows the necessity of introducing and analyzing the effects of perturbations on the controller internal state $\ttf{\beta}$.


Of course, when the open loop system is stable, the controller structures illustrated below may be appealing as they are simpler and easier to implement. \rev{In fact, we can show that the controller structure in Figure \ref{fig:alt1} is an alternative realization of the internal model control principle (IMC) \cite{IMC1, IMC2} as applied to the Youla parameterization. Specifically, for open loop stable systems, the Youla parameter is given by $\tf Q = \tf K (I - \tf{P_{22}} \tf K)^{-1}$. As is show in Lemma 5 of \cite{SysLevelSyn1}, the Youla parameter $\tf Q$ is equal to the system response $\Phiuy$ for open loop stable systems. We then have}
\rev{
\begin{subequations} \label{eq:imc}
\begin{align}
\tf u &= \Phiuy \tf y - \Phiux B_2 \tf u \label{eq:imc0}\\
&= \tf Q \tf y - \Phiuy  C_2 (zI-A)^{-1} B_2 \tf u \label{eq:imc1} \\
&= \tf Q \tf y - \tf Q \tf{P_{22}} \tf u \label{eq:imc2} \\
& = \tf Q (\tf y - \tf{P_{22}} \tf u), \label{eq:imc3}
\end{align}
\end{subequations} 
where \eqref{eq:imc1} is obtained by substituting $\Phiux = \Phiuy  C_2 (zI-A)^{-1}$ from \eqref{eq:output_fb2} into \eqref{eq:imc0}. Equation \eqref{eq:imc3} is exactly IMC. Thus, we see that IMC is equivalent to our proposed parameterization (and the simplified representation shown in Figure \ref{fig:alt1}) for open loop stable systems. } 

\begin{figure}[ht!]
      \centering
      \subfigure[\rev{Internal Model Control}]{%
      \includegraphics[width=0.5\textwidth]{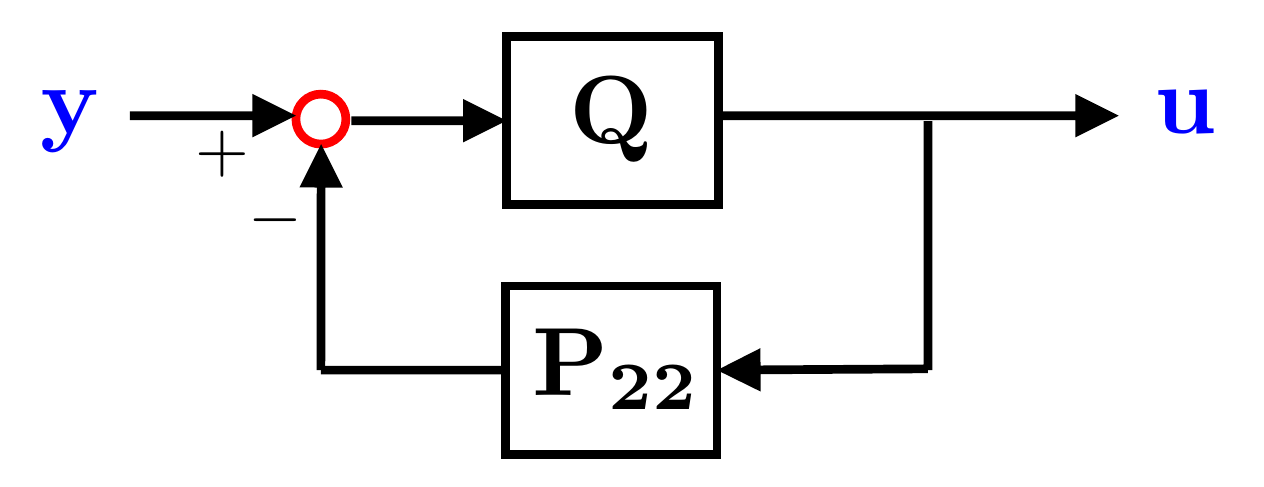}
      \label{fig:alt0}}
      
      \subfigure[Structure 1]{%
      \includegraphics[width=0.5\textwidth]{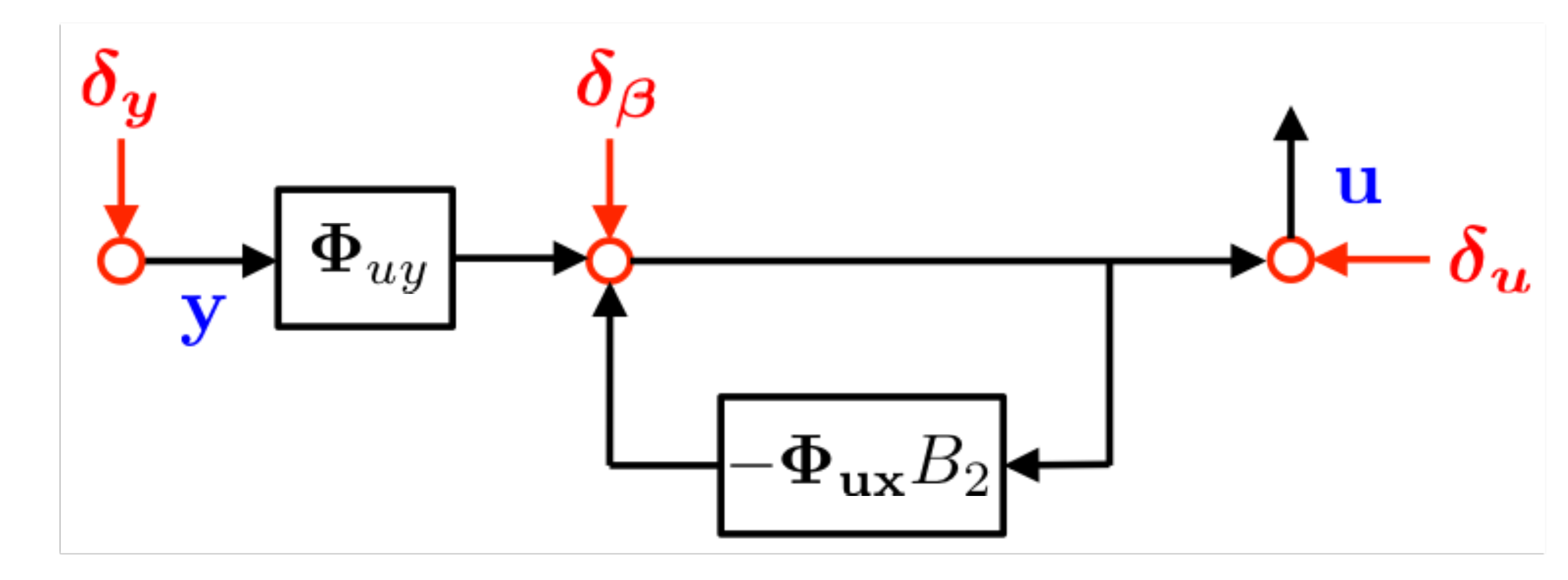}
      \label{fig:alt1}}
      
      \subfigure[Structure 2]{%
      \includegraphics[width=0.5\textwidth]{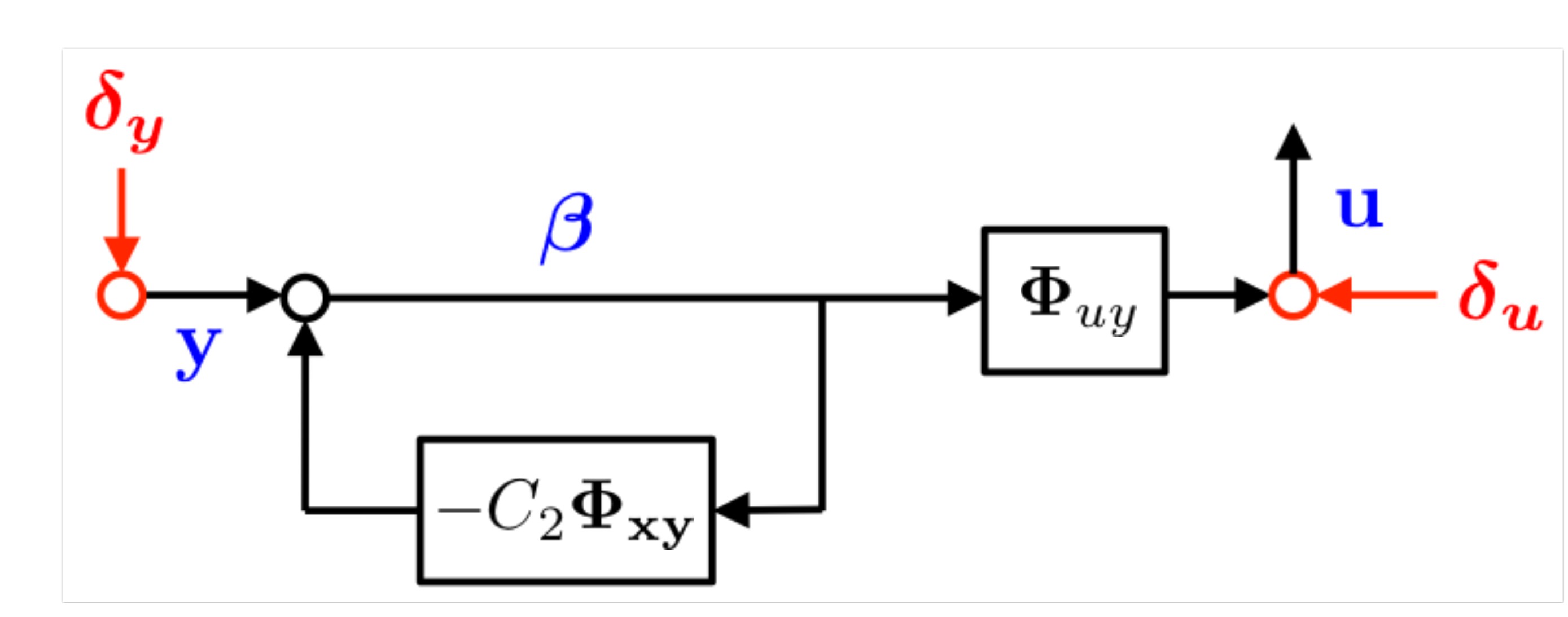}
      \label{fig:alt2}}
      
      \caption{Alternative controller structures for stable systems.}
\end{figure}

\subsection{Output Feedback with $D_{22} \not = 0$} \label{sec:non-sp}
Finally, for a general proper plant model \eqref{eq:sys-dynamics} with $D_{22} \not = 0$, we define a new measurement $\bar{y}[t] = y[t] - D_{22} u[t]$. This leads to the controller structure shown in Figure \ref{fig:pof}. In this case, the closed loop transfer matrices from $\ttf{\delta_u}$ to the internal variables become
\begin{equation}
\begin{bmatrix} \tf x \\ \tf u \\ \tf y \\ \ttf{\beta} \end{bmatrix} = \begin{bmatrix} \Phixx B_2 + \Phixy D_{22} \\ I + \Phiux B_2 + \Phiuy D_{22} \\ C_2 \Phixx B_2 + D_{22} + C_2 \Phixy D_{22} \\ -\frac{1}{z}B_2 (\Phiux B_2 + \Phiuy D_{22}) \end{bmatrix} \ttf{\delta_u}. \nonumber
\end{equation}
The remaining entries of Table \ref{Table:1} remain the same. Therefore, the controller structure shown in Figure \ref{fig:pof} internally stabilizes the plant.
\begin{figure}[ht!]
      \centering
      \includegraphics[width=0.4\textwidth]{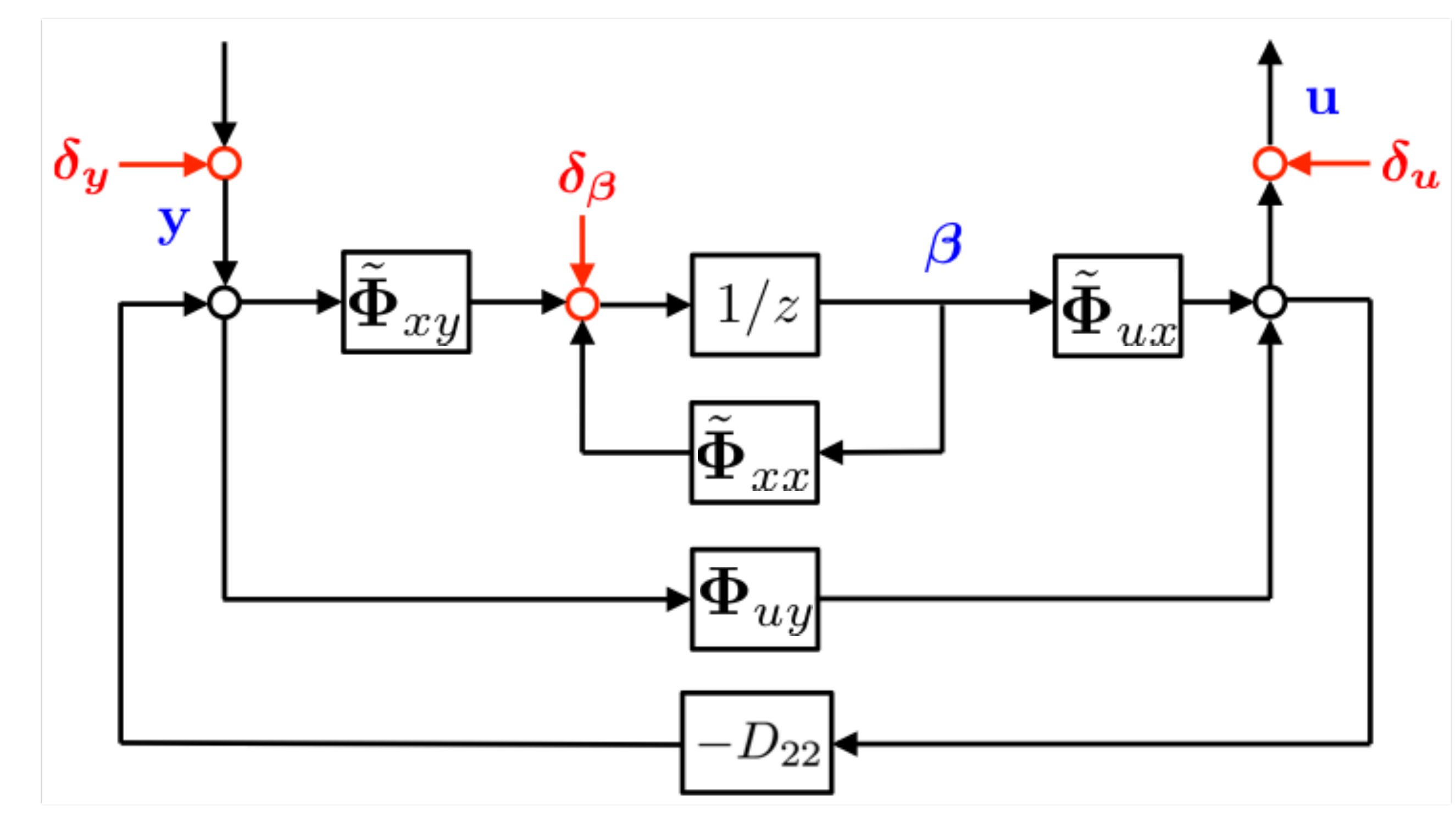}
      \caption{The proposed output feedback controller structure for $D_{22} \not = 0$.}
      \label{fig:pof}
\end{figure}

\subsubsection*{\revsecond{System Level and Youla Parameterizations}}
\revsecond{A key difference between the SL and Youla parameterizations is the manner in which they characterize the achievable closed loop responses of a system.  The Youla parameterization provides an image space representation of the achievable system responses, parameterized explicitly by the free Youla parameter.  This parameterization lends itself naturally to efficient computation via the standard and theoretically supported approach \cite{Boyd_closed} of restricting the Youla parameter and objective function to be FIR.  However, despite this ease of computation, as alluded to in Section \ref{sec:prior}, imposing sparsity constraints on the controller via the Youla parameter in general intractable.}

\revsecond{In contrast, the proposed SL parameterization specifies a kernel space representation of achievable system responses, parameterized implicitly by the affine space \eqref{eq:output_fb1} - \eqref{eq:output_fb2}.  While our discussion highlights the benefits and flexibility of the SL approach, there is the important caveat that the affine constraints \eqref{eq:output_fb1} - \eqref{eq:output_fb2} are in general infinite dimensional. Hence, although the parameterization is a convex one, it does not immediately lend itself to efficient computation.  However, much as in the state-feedback setting, imposing (approximate) FIR constraints on the system responses leads to a finite-dimensional optimization problem.}

\subsection{Output Feedback System Level Synthesis}
We conclude this subsection by formulating the general output feedback SLS problem.  Proceeding as in Section \ref{sec:state_feedback_SLS}, we use Theorem~\ref{thm:of} to pose the following optimal control problem
\begin{equation}
\begin{array}{rl}
 \min_{\Phixx, \Phiux,\Phixy,\Phiuy}&  g(\Phixx,\Phiux,\Phixy,\Phiuy) \\
\st &  \text{constraints \eqref{eq:output_fb}}\\
& \Phix, \Phiu \in \frac{1}{z}\mathcal{R}\hinf \cap \Ss,
\end{array}\label{eq:of-opt}
\end{equation}
where as before, $g(\cdot)$ is an appropriately specified objective function, constraints \eqref{eq:output_fb} ensure achievability of the resulting system responses, and $\Ss$ can be used to impose SLCs.  A typical choice for the objective function $g(\cdot)$ is
\begin{equation}
g(\Phixx,\Phiux,\Phixy,\Phiuy) = \bignorm{\begin{bmatrix}C_1 & D_{12}\end{bmatrix}\begin{bmatrix} \Phixx & \Phixy \\ \Phiux & \Phiuy\end{bmatrix}\begin{bmatrix}B_1 \\ D_{21}\end{bmatrix}},
\label{eq:of-obj}
\end{equation}
for $\norm{\cdot}$ a suitably chosen system norm such as $\htwo$, $\hinf$, or $\Ell_1$.

In the next subsection, we show how the imposition of localized SLCs once again leads to scalable controller synthesis in the distributed setting.

\subsection{Distributed Control, Locality, and Scalability}
\label{sec:output-locality}
In this section, we extend the separable objectives and constraints introduced previously to encompass a class of \emph{partially} separable objectives and constraints, and show how these properties can be exploited to solve the general form of the output feedback SLS problem \eqref{eq:of-opt} when locality constraints are imposed in a scalable manner by using distributed optimization techniques such as ADMM.\footnote{Note that \eqref{eq:of-opt} is neither column-wise nor row-wise separable due to the coupling constraints \eqref{eq:output_fb1} and \eqref{eq:output_fb2}.} 
We further show that many constrained optimal control problems are partially separable. Examples include the (localized) $\mathcal{H}_2$ optimal control problem with joint sensor and actuator regularization and the (localized) mixed $\mathcal{H}_2 / \mathcal{L}_1$ optimal control problem. 

We begin by defining the partial separability of a generic optimization problem. We then specialize our discussion to partially separable instances of the SLS problem \eqref{eq:of-opt}, before ending this section with a discussion of computational techniques that can be used to accelerate the ADMM algorithm as applied to solving partially separable SLS problems.

\rev{Before proceeding, we remark that although, as we show below, many SLOs and SLCs of interest satisfy our definition of partial separability, some classical measures of performance do not.  For instance the $\mathcal{H}_\infty$ norm of the closed loop response is not partially separable. This however hints at an interesting direction for future work: can suitable notions of robustness to $\ell_2$ bounded disturbances be defined for the large-scale setting that are partially separable?}

\subsection{Partially Separable Problems} \label{sec:gopt-o}
We begin with a generic optimization problem given by
\begin{subequations} \label{eq:gopt-o}
\begin{align} 
\underset{\ttf{\Phi}}{\text{minimize }} \quad & g(\ttf{\Phi}) \label{eq:gopt-o1}\\
\text{subject to } \quad & \ttf{\Phi} \in \s, \label{eq:gopt-o2}
\end{align}
\end{subequations}
where $\ttf{\Phi}$ is a $m \times n$ transfer matrix, $g(\cdot)$ a functional objective, and $\s$ a set constraint. We assume that the optimization problem \eqref{eq:gopt-o} can be written in the form
\begin{subequations} \label{eq:gopt-od}
\begin{align} 
\underset{\ttf{\Phi}}{\text{minimize }} \quad & \rowdec{g}(\ttf{\Phi}) + \coldec{g}(\ttf{\Phi}) \label{eq:gopt-od1}\\
\text{subject to } \quad & \ttf{\Phi} \in \rowdec{\s} \cap \coldec{\s}, \label{eq:gopt-od2}
\end{align}
\end{subequations}
where $\rowdec{g}$ and $\rowdec{\s}$ are row-wise separable\footnote{Row-wise separability of functions and constraint sets is defined in the obvious manner, analogous to column-wise separability from Definition~\ref{def:cols}.} and $\coldec{g}$ and $\coldec{\s}$ are column-wise separable. If optimization problem \eqref{eq:gopt-o} can be written in the form of problem \eqref{eq:gopt-od}, then optimization problem \eqref{eq:gopt-o} is called \emph{partially separable}.

Due to the coupling between the row-wise and column-wise separable components, optimization problem \eqref{eq:gopt-od} admits neither a row-wise nor a column-wise decomposition. Our strategy is to instead use distributed optimization techniques such as ADMM to decouple the row-wise separable component and the column-wise separable components: we show that this leads to iterate subproblems that are column/row-wise separable, allowing the results of the previous section to be applied to solve the iterate subproblems using localized and parallel computation. 

\begin{definition} \label{dfn:do}
Let $\{ \rowps_1, \dots, \rowps_{q}\}$ be a partition of the set $\{1, \cdots, m\}$ and $\{ \colps_1, \dots, \colps_{p}\}$ be a partition of the set $\{1, \cdots, n\}$. The functional objective $g(\ttf{\Phi})$ in \eqref{eq:gopt-o1} is said to be \emph{partially separable} with respect to the row-wise partition $\{ \rowps_1, \dots, \rowps_{q}\}$ and the column-wise partition $\{ \colps_1, \dots, \colps_{p}\}$ if $g(\ttf{\Phi})$ can be written as the sum of $\rowdec{g}(\ttf{\Phi})$ and $\coldec{g}(\ttf{\Phi})$, where $\rowdec{g}(\ttf{\Phi})$ is row-wise separable with respect to the row-wise partition $\{ \rowps_1, \dots, \rowps_{q}\}$ and $\coldec{g}(\ttf{\Phi})$ is column-wise separable with respect to the column-wise partition $\{ \colps_1, \dots, \colps_{p} \}$. Specifically, we have 
\begin{eqnarray}
g(\ttf{\Phi}) &=& \rowdec{g}(\ttf{\Phi}) + \coldec{g}(\ttf{\Phi}) \nonumber\\
\rowdec{g}(\ttf{\Phi}) &=& \sum_{j=1}^{q} \rowdec{g}_j(\Sub{\ttf{\Phi}}{\rowps_j}{:}) \nonumber\\
\coldec{g}(\ttf{\Phi}) &=& \sum_{j=1}^{p} \coldec{g}_j(\Sub{\ttf{\Phi}}{:}{\colps_j}) \label{eq:obj_decom2}
\end{eqnarray}
for some functionals $\rowdec{g}_j(\cdot)$ for $j = 1, \dots, q$, and $\coldec{g}_j(\cdot)$ for $j = 1, \dots, p$.  
\end{definition}

\begin{example}
Let $\Phi$ be the optimization variable, and $E$ and $F$ be two matrices of compatible dimension. The objective function
\begin{equation}
\| E \Phi \|_{\mathcal{F}}^2 + \| \Phi F \|_{\mathcal{F}}^2 \nonumber
\end{equation}
is then partially separable. Specifically, the first term is column-wise separable, and the second term is row-wise separable.
\end{example}

\begin{definition} \label{dfn:ss}
The set constraint $\s$ in \eqref{eq:gopt-o2} is said to be \emph{partially separable} with respect to the row-wise partition $\{ \rowps_1, \dots, \rowps_{q}\}$ and the column-wise partition $\{ \colps_1, \dots, \colps_{p}\}$ if $\s$ can be written as the intersection of two sets $\rowdec{\s}$ and $\coldec{\s}$, where $\rowdec{\s}$ is row-wise separable with respect to the row-wise partition $\{ \rowps_1, \dots, \rowps_{q}\}$ and $\coldec{\s}$ is column-wise separable with respect to the column-wise partition $\{ \colps_1, \dots, \colps_{p} \}$. Specifically, we have 
\begin{eqnarray}
\s &=& \rowdec{\s} \cap \coldec{\s} \nonumber\\
\ttf{\Phi} \in \rowdec{\s} \quad &\Longleftrightarrow& \quad \Sub{\ttf{\Phi}}{\rowps_j}{:} \in \rowdec{\s}_j, \forall j \nonumber\\
\ttf{\Phi} \in \coldec{\s} \quad &\Longleftrightarrow& \quad \Sub{\ttf{\Phi}}{:}{\colps_j} \in \coldec{\s}_j, \forall j \label{eq:sdecom}
\end{eqnarray}
for some sets $\rowdec{\s}_j$ for $j = 1, \dots, q$ and $\coldec{\s}_j$ for $j = 1, \dots, p$.
\end{definition}

\begin{example} [Affine Subspace]
The affine subspace constraint
\begin{equation}
\tf G \ttf{\Phi} = \tf H \quad \text{and} \quad \ttf{\Phi} \tf E = \tf F \nonumber
\end{equation}
is partially separable. Specifically, the first constraint is column-wise separable, and the second is row-wise separable.
\end{example}
\begin{example} [Induced Matrix Norm]
The induced $1$-norm of a matrix $\Phi$ is given by
\begin{equation}
\| \Phi \|_{1} = \underset{ 1 \leq j \leq n}{\text{max}} \,\, \sum_{i=1}^m | \Phi_{ij} |, \nonumber
\end{equation}
and the induced $\infty$-norm of a matrix is given by
\begin{equation}
\| \Phi \|_{\infty} = \underset{ 1 \leq i \leq m}{\text{max}} \,\, \sum_{j=1}^n | \Phi_{ij} |. \nonumber
\end{equation}
The following constraint is partially separable:
\begin{equation}
\| \Phi \|_{1} \leq \gamma_1 \quad \text{and} \quad \| \Phi \|_{\infty} \leq \gamma_{\infty}. \nonumber
\end{equation}
Specifically, the first term is column-wise separable, and the second term is row-wise separable.
\end{example}

Assume now that the objective and the constraints in optimization problem \eqref{eq:gopt-o} are both partially separable with respect to a row-wise partition $\{ \rowps_1, \dots, \rowps_{q}\}$ and a column-wise partition $\{ \colps_1, \dots, \colps_{p}\}$. In this case, the optimization problem \eqref{eq:gopt-o} is partially separable and can be rewritten in the form specified by equation \eqref{eq:gopt-od}. We now propose an ADMM based algorithm that exploits the partially separable structure of the optimization problem.  As we show in the next subsection, when this method is applied to the SLS problem \eqref{eq:of-opt}, the iterate subproblems are column/row-wise separable, allowing us to apply the methods described in Section \ref{sec:scalability}.  In particular, if locality constraints are imposed, iterate subproblems can be solved using localized and parallel computation.

Let $\ttf{\Psi}$ be a duplicate of the optimization variable $\ttf{\Phi}$. We define the extended-real-value functionals $\rowdec{h}(\ttf{\Phi})$ and $\coldec{h}(\ttf{\Psi})$ by
\begin{eqnarray}
&& \rowdec{h}(\ttf{\Phi}) = \left\{ \begin{array}{ll}
\rowdec{g}(\ttf{\Phi}) & \textrm{if } \ttf{\Phi} \in \rowdec{\s}\\
\infty & \textrm{otherwise}
\end{array} \right. \nonumber\\
&& \coldec{h}(\ttf{\Psi}) = \left\{ \begin{array}{ll}
\coldec{g}(\ttf{\Psi}) & \textrm{if } \ttf{\Psi} \in \coldec{\s}\\
\infty & \textrm{otherwise.}
\end{array} \right. \label{eq:extended}
\end{eqnarray}
Problem \eqref{eq:gopt-od} can then be reformulated as
\begin{eqnarray}
\underset{\{ \ttf{\Phi}, \ttf{\Psi} \}}{\text{minimize    }} && \rowdec{h}(\ttf{\Phi}) + \coldec{h}(\ttf{\Psi}) \nonumber\\
\text{subject to }  && \ttf{\Phi} = \ttf{\Psi}. \label{eq:sub3}
\end{eqnarray}
Problem \eqref{eq:sub3} can be solved using ADMM \cite{BoydADMM}:
\begin{subequations}\label{eq:admm}
\begin{align}
\ttf{\Phi}^{k+1} &= \underset{\ttf{\Phi}}{\text{argmin }} \Big( \rowdec{h}(\ttf{\Phi}) + \frac{\rho}{2} || \ttf{\Phi} - \ttf{\Psi}^{k} + \ttf{\Lambda}^{k}||_{\mathcal{H}_2}^2 \Big) \nonumber \\
\label{eq:admm1} \\
\ttf{\Psi}^{k+1} &= \underset{\ttf{\Psi}}{\text{argmin }} \Big( \coldec{h}(\ttf{\Psi}) + \frac{\rho}{2} || \ttf{\Psi} - \ttf{\Phi}^{k+1} - \ttf{\Lambda}^{k}||_{\mathcal{H}_2}^2 \Big) \nonumber \\
\label{eq:admm2} \\
\ttf{\Lambda}^{k+1} &= \ttf{\Lambda}^{k} + \ttf{\Phi}^{k+1} - \ttf{\Psi}^{k+1} \label{eq:admm3}
\end{align}
\end{subequations}
where the square of the $\mathcal{H}_2$ norm is computed as in \eqref{eq:h2_comp}. 



In Appendix B of \cite{SysLevelSyn2}, we provide stopping criterion and prove convergence of the ADMM algorithm \eqref{eq:admm1} - \eqref{eq:admm3} to an optimal solution of \eqref{eq:gopt-od} (or equivalently, \eqref{eq:sub3}) under the following assumptions. 
\begin{assumption}
Problem \eqref{eq:gopt-od} has a feasible solution in the relative interior of the set $\s$. \label{as:1}
\end{assumption}
\begin{assumption} 
The functionals $\rowdec{g}(\cdot)$ and $\coldec{g}(\cdot)$ are closed, proper, and convex. \label{as:2}
\end{assumption}
\begin{assumption} 
The sets $\rowdec{\s}$ and $\coldec{\s}$ are closed and convex. \label{as:3}
\end{assumption}

\rev{We note that these assumptions are satisfied by typical SLCs that one may wish to apply.  For instance, all of the examples considered in this paper satisfy these conditions so long as they specify a non-empty feasible set.}

\subsection{Partially Separable SLS Problems} \label{sec:ps}
We now specialize our discussion to the output feedback SLS problem \eqref{eq:of-opt}. We use
\begin{equation}
\ttf{\Phi} = \begin{bmatrix} \Phixx & \Phixy \\ \Phiux & \Phiuy \end{bmatrix} \nonumber
\end{equation}
to represent the system response we want to optimize for. Denote by $\tf Z_{AB}$ the transfer matrix $\begin{bmatrix} zI-A & -B_2 \end{bmatrix}$, and by $\tf Z_{AC}$ the transfer matrix $\begin{bmatrix} zI-A^\top & -C_2^\top \end{bmatrix}^\top$. Let $J_B$ be the matrix in the right-hand-side of \eqref{eq:output_fb1} and $J_C$ be the matrix in the right-hand-side of \eqref{eq:output_fb2}. The localized SLS problem \eqref{eq:of-opt} can then be written as
\begin{subequations} \label{eq:of_lop}
\begin{align} 
\underset{\ttf{\Phi}}{\text{minimize }} \quad & g(\ttf{\Phi}) \label{eq:of-1}\\
\text{subject to } \quad & \tf Z_{AB} \ttf{\Phi} = J_B \label{eq:of-2} \\
& \ttf{\Phi} \tf Z_{AC} = J_C \label{eq:of-3} \\
& \ttf{\Phi} \in \s \label{eq:of-4}
\end{align}
\end{subequations}
with $\s = \Ell \cap \FT \cap \Sother$. Note that as already discussed, the affine subspace constraints \eqref{eq:of-2} and \eqref{eq:of-3} are partially separable with respect to arbitrary column-wise and row-wise partitions, respectively. Thus the SLS problem \eqref{eq:of_lop} is partially separable if the SLO \eqref{eq:of-1} and the SLC $\ttf \Phi \in \Sother$ are partially separable.

In particular, if $\Sother$ is partially separable, we can express the original SLC $\s$ as an intersection of the sets $\rowdec{\s} = \Ell \cap \FT \cap \rowdec{\Sother}$ and $\coldec{\s} = \Ell \cap \FT \cap \coldec{\Sother}$, where $\rowdec{\Sother}$ is a row-wise separable component of $\Sother$ and $\coldec{\Sother}$ a column-wise separable component of $\Sother$.
Note that the locality constraint and the FIR constraint are included in both $\rowdec{\s}$ and $\coldec{\s}$. 
This is the key point to allow the subroutines \eqref{eq:admm1} - \eqref{eq:admm2} of the ADMM algorithm to be solved using the techniques described in Section \ref{sec:scalability}.

To illustrate how the iterate subproblems of the ADMM algorithm \eqref{eq:admm}, as applied to the SLS problem \eqref{eq:of_lop}, can be solved using localized and parallel computation, we analyze in more detail the iterate subproblem \eqref{eq:admm2}, which is given by
\begin{subequations} \label{eq:admm_lop}
\begin{align} 
\underset{\ttf{\Psi}}{\text{minimize }} \quad & \coldec{g}(\ttf{\Psi}) + \frac{\rho}{2} || \ttf{\Psi} - \ttf{\Phi}^{k+1} - \ttf{\Lambda}^{k}||_{\mathcal{H}_2}^2 \label{eq:admmlop1}\\
\text{subject to } \quad & \ttf{\Psi} \in \coldec{\s} \label{eq:admmlop3}
\end{align}
\end{subequations}
with $\coldec{\s} = \Ell \cap \FT \cap \coldec{\Sother}$.
The $\mathcal{H}_2$ norm regularizer in \eqref{eq:admmlop1} is column-wise separable with respect to arbitrary column-wise partition.
As the objective $\coldec{g}(\cdot)$ and the constraint $\coldec{\s}$ are column-wise separable with respect to a given column-wise partition, we can use the column-wise separation technique described in the previous section to decompose \eqref{eq:admm_lop} into parallel subproblems. 
Then, as in the state-feedback setting, we can then exploit the locality constraint $\Ell$ to reduce the dimension of each subproblem from global to local scale.  Similarly, iterate subproblem \eqref{eq:admm1} can also be solved using parallel and localized computation by exploiting its row-wise separability.  Finally, update equation \eqref{eq:admm3} trivially decomposes element-wise since it is a matrix addition. 

We now give some examples of partially separable SLOs.

\begin{example}[Norm Optimal Control] \label{ex:b1}
Consider the SLO of the optimal control problem in \eqref{eq:of-opt} with objective function as in \eqref{eq:of-obj} and norm given by the square of the $\mathcal{H}_2$ norm. Suppose that there exists a permutation matrix $\Pi$ such that the matrix $\begin{bmatrix} B_1^\top & D_{21}^\top \end{bmatrix} \Pi$ is block diagonal. It is then possible to find a column-wise partition to decompose the SLO in a column-wise manner. Similarly, suppose that there exists a permutation matrix $\Pi$ such that the matrix $\begin{bmatrix} C_1 & D_{12} \end{bmatrix} \Pi$ is block diagonal. We can find a row-wise partition to decompose the SLO in a row-wise manner. In both cases, the SLO is column/row-wise separable and thus partially separable.
\end{example}
\begin{example}[Sensor and Actuator Norm]
Consider the weighted actuator norm defined in \cite{MC_CDC14, 2015_Wang_Reg}, which is given by
\begin{equation}
|| \mu \begin{bmatrix} \Phiux & \Phiuy \end{bmatrix} ||_{\mathcal{U}} = \sum_{i=1}^{n_u} \mu_i || e_i^\top \begin{bmatrix} \Phiux & \Phiuy \end{bmatrix} ||_{\mathcal{H}_2} \label{eq:act_norm} 
\end{equation}
where $\mu$ is a diagonal matrix with \rev{$\mu_i \geq 0$} being its $i^{th}$ diagonal entry, and $e_i$ is a unit vector $i^{th}$ entry set to $1$, and $0$ elsewhere. The actuator norm \eqref{eq:act_norm} can be viewed as an infinite dimensional analog to the weighted $\ell_1 / \ell_2$ norm, also known as the group lasso \cite{2006_grouplasso} in the statistical learning literature. Adding this norm as a regularizer to a SLS problem induces row-wise sparse structure in the transfer matrix $\begin{bmatrix} \Phiux & \Phiuy \end{bmatrix}$. Recall from Theorem \ref{thm:of} that the controller achieving the desired system response can be implemented via \eqref{eq:ss_like}. If the $i$th row of the transfer matrix $\begin{bmatrix} \Phiux & \Phiuy \end{bmatrix}$ is identically zero, then the $i^{th}$ component of the control action $u_i$ is always equal to zero, and therefore the actuator at node $i$ corresponding to control action $u_i$ can be removed without changing the closed loop response. It is clear that the actuator norm defined in \eqref{eq:act_norm} is row-wise separable with respect to arbitrary row-wise partition. This still holds true when the actuator norm is defined by the $\ell_1 / \ell_\infty$ norm. Similarly, consider the weighted sensor norm given by
\begin{equation}
|| \begin{bmatrix} \Phixy \\ \Phiuy \end{bmatrix} \lambda ||_{\mathcal{Y}} = \sum_{i=1}^{n_y} \lambda_i || \begin{bmatrix} \Phixy \\ \Phiuy \end{bmatrix} e_i ||_{\mathcal{H}_2} \label{eq:sen_norm} 
\end{equation}
where $\lambda$ is a diagonal matrix with \rev{$\lambda_i \geq 0$} being its $i$th diagonal entry. The sensor norm \eqref{eq:sen_norm}, when added as a regularizer, induces column-wise sparsity in the transfer matrix $\begin{bmatrix} \Phixy^\top & \Phiuy^\top \end{bmatrix}^\top$. Using the controller implementation \eqref{eq:ss_like}, the sensor norm can therefore be viewed as regularizing the number of sensors used by a controller. For instance, if the $i^{th}$ column of the transfer matrix $\begin{bmatrix} \Phixy^\top & \Phiuy^\top \end{bmatrix}^\top$ is identically zero, then the sensor at node $i$ and its corresponsing measurement $y_i$ can be removed without changing the closed loop response. The sensor norm defined in \eqref{eq:act_norm} is column-wise separable with respect to any column-wise partition.
\end{example}
\begin{example} [Combination]
From Definition \ref{dfn:do}, it is straightforward to see that the class of partially separable SLOs with respect to the same partitions are closed under summation.
Therefore, we can combine the partially separable SLOs described above, and the resulting SLO is still partially separable. For instance, consider the SLO given by
\begin{eqnarray}
g(\Phixx,\Phiux,\Phixy,\Phiuy) &=& || \begin{bmatrix} C_1 & D_{12} \end{bmatrix} \begin{bmatrix} \Phixx & \Phixy \\ \Phiux & \Phiuy \end{bmatrix} \begin{bmatrix} B_1 \\ D_{21} \end{bmatrix} ||_{\mathcal{H}_2}^2 \nonumber\\
&&+ || \mu \begin{bmatrix} \Phiux & \Phiuy \end{bmatrix} ||_{\mathcal{U}} + || \begin{bmatrix} \Phixy \\ \Phiuy \end{bmatrix} \lambda ||_{\mathcal{Y}} \label{eq:h2_sar}
\end{eqnarray}
where \rev{$\mu$ and $\lambda$ are non-negative diagonal matrices specifying the relative weighting between the closed loop $\mathcal{H}_2$ performance, and the actuator and sensor norm penalties, respectively. }
If there exists a permutation matrix $\Pi$ such that the matrix $\begin{bmatrix} C_1 & D_{12} \end{bmatrix} \Pi$ is block diagonal, then the SLO \eqref{eq:h2_sar} is partially separable. 
Specifically, the $\mathcal{H}_2$ norm and the actuator regularizer belong to the row-wise separable component, and the sensor regularizer belongs to the column-wise separable component.\footnote{Note that an alternative penalty is proposed in \cite{MC_CDC14} for the design of joint actuation and sensing architectures; it is however more involved to define, and hence we restrict ourself to this alternative penalty for the purposes of illustrating the concept of partially separable SLOs.}
\end{example}

We now provide some examples of partially separable SLCs.
\begin{example} [$\mathcal{L}_1$ Constraint]
The $\mathcal{L}_1$ norm of a transfer matrix is given by its worst case $\ell_\infty$ to $\ell_\infty$ gain.  In particular, the $\mathcal{L}_1$ norm \cite{Munther_L1} of a FIR transfer matrix $\tf G \in \FT$ is given by
\begin{equation}
|| \tf G ||_{\mathcal{L}_1} = \underset{i}{\text{max}} \quad \sum_{j} \sum_{t=0}^T | g_{ij}[t] |.\label{eq:L1_defn}
\end{equation}
We can therefore add the constraint
\begin{equation}
|| \begin{bmatrix} \Phixx & \Phixy \\ \Phiux & \Phiuy \end{bmatrix} \begin{bmatrix} B_1 \\ D_{21} \end{bmatrix}||_{\mathcal{L}_1} \leq \gamma \label{eq:L1_cons}
\end{equation}
to the optimization problem \eqref{eq:of_lop} for some $\gamma$ to control the worst-case amplification of $\ell_\infty$ bounded signals.  From the definition \eqref{eq:L1_defn}, the SLC \eqref{eq:L1_cons} is row-wise separable with respect to any row-wise partition. 
\end{example}
\begin{example} [Combination]
From Definition \ref{dfn:ss}, the class of partially separable SLCs with respect to the same row and column partitions are closed under intersection, allowing for partially separable SLCs to be combined.
For instance, the combination of a locality constraint $\Ell$, a FIR constraint $\FT$, and an $\mathcal{L}_1$ constraint as in equation \eqref{eq:L1_cons} is partially separable.  This property is extremely useful as it provides a unified framework for dealing with several partially separable constraints at once.
\end{example}

 Using the previously described examples of partially separable SLOs and SLCs, we now consider two partially separable SLS problems: (i) the localized $\mathcal{H}_2$ optimal control problem with joint sensor and actuator regularization, and (ii) the localized mixed $\mathcal{H}_2 / \mathcal{L}_1$ optimal control problem. These two problems are used in Section \ref{sec:simu} as case study examples.

\begin{example}
The localized $\mathcal{H}_2$ optimal control problem with joint sensor and actuator regularization is given by
\begin{subequations} \label{eq:h2SA}
\begin{align} 
\underset{\{\Phixx, \Phiux, \Phixy, \Phiuy\}}{\text{minimize }} \quad & \eqref{eq:h2_sar} \label{eq:h2SA-1} \\
\text{subject to } \quad & \eqref{eq:output_fb1} - \eqref{eq:output_fb3}  \label{eq:h2SA-2} \\
& \begin{bmatrix} \Phixx & \Phixy \\ \Phiux & \Phiuy \end{bmatrix} \in \mathcal{C} \cap \Ell \cap \FT. \label{eq:h2SA-3}
\end{align}
\end{subequations}
where $\mathcal{C}$ encodes the information sharing constraints of the distributed controller. If there exists a permutation matrix $\Pi$ such that the matrix $\begin{bmatrix} C_1 & D_{12} \end{bmatrix} \Pi$ is block diagonal, then \eqref{eq:h2SA} is partially separable. 
\end{example}
\begin{remark}
\rev{When the actuator and sensor norm penalties are zero}, problem \eqref{eq:h2SA} reduces to the localized LQG optimal control problem defined and solved in \cite{2015_Wang_H2}.  Further, if the locality and FIR constraints are removed, we recover the standard LQG optimal control problem, which is also seen to be partially separable.
\end{remark}


Next we consider the localized mixed $\mathcal{H}_2/\mathcal{L}_1$ optimal control problem.
\begin{example}
The localized mixed $\mathcal{H}_2/\mathcal{L}_1$ optimal control problem is given by 
\begin{subequations} \label{eq:h2l1}
\begin{align} 
\underset{\{\Phixx, \Phiux, \Phixy, \Phiuy\}}{\text{minimize }} \quad & || \begin{bmatrix} \Phixx & \Phixy \\ \Phiux & \Phiuy \end{bmatrix} \begin{bmatrix} B_1 \\ D_{21} \end{bmatrix}||_{\mathcal{H}_2}^2 \label{eq:h2l1-1} \\
\text{subject to } \quad & \eqref{eq:output_fb1} - \eqref{eq:output_fb3}, \eqref{eq:L1_cons}, \eqref{eq:h2SA-3} \label{eq:h2l1-2}
\end{align}
\end{subequations}
which is partially separable.
\end{example}

The localized mixed $\mathcal{H}_2/\mathcal{L}_1$ optimal control problem can be used to design the tradeoff between average and worst-case performance, as measured by the $\mathcal{H}_2$ and $\mathcal{L}_1$ norms of the closed loop system, respectively. 


\subsection{Analytic Solution and Acceleration}
Suppose that the Assumptions \ref{as:1} - \ref{as:3} in Section \ref{sec:gopt-o} hold. The ADMM algorithm presented in \eqref{eq:admm} is a special case of the proximal algorithm \cite{proximal, proximalBoyd, BoydADMM}. For certain type of objective functionals $\rowdec{h}(\cdot)$ and $\coldec{h}(\cdot)$, the proximal operators can be evaluated analytically (see Ch. 6 of \cite{proximalBoyd}). In this situation, we only need to evaluate the proximal operators \emph{once}, and iterate \eqref{eq:admm1} and \eqref{eq:admm2} in closed form. This improves the overall computation time significantly.  In Appendix C of \cite{SysLevelSyn2}, it is shown how to express the solutions of \eqref{eq:admm1} and \eqref{eq:admm2} using proximal operators. Here we list a few examples for which the proximal operators can be evaluated analytically. 

\begin{example} \label{ex:analytic1}
Consider the LLQG problem proposed in \cite{2015_Wang_H2}.
When the global optimization problem is decomposed into parallel subproblems, each subproblem is a convex quadratic program restricted to an affine set. In this case, the proximal operator is an affine function \cite{2015_Wang_H2, proximalBoyd, BoydADMM}. We only need to calculate this affine function once. The iteration in \eqref{eq:admm1} - \eqref{eq:admm3} can then be carried out using multiple matrix multiplications in the reduced dimension, which significantly improves the overall computation time.
\end{example}
\begin{example} \label{ex:analytic2}
Consider the LLQR problem with actuator regularization \cite{2015_Wang_Reg}, which is the state feedback version of \eqref{eq:h2SA}. The column-wise separable component of this problem is identical to that of the LLQG example, and therefore the update subproblem \eqref{eq:admm2} can be solved using matrix multiplication, as described in the previous example.  We further showed in \cite{2015_Wang_Reg} that the row-wise separable component of this problem can be simplified to multiple unconstrained optimization problems, each with proximal operators given by vectorial soft-thresholding \cite{proximalBoyd}.  This offers an efficient way to solve the update supbroblem \eqref{eq:admm1}.
\end{example}

We end this section by noting that although our focus has been on how ADMM can be used to efficiently solve the localized SLS problem \eqref{eq:of-opt},  there exist other distributed algorithms that can exploit partial separability to improve computational efficiency.  For instance, if either $\rowdec{g}(\cdot)$ or $\coldec{g}(\cdot)$ is strongly convex, we can use the alternating minimization algorithm (AMA) \cite{AMA} to simplify the ADMM algorithm. 


\subsection{Case Studies}
\label{sec:simu}
In this section, we apply the localized $\mathcal{H}_2$ optimal control with sensor actuator regularization \eqref{eq:h2SA} problem and the localized mixed $\mathcal{H}_2 / \mathcal{L}_1$ optimal control problem \eqref{eq:h2l1} to a power system inspired example.
After introducing the power system model, we show that the localized $\mathcal{H}_2$ controller, with its additional locality, FIR and communication delay constraints can achieve comparable closed loop performance to that of a centralized $\mathcal{H}_2$ optimal controller.  We further demonstrate the scalability of the proposed method by synthesizing a localized $\mathcal{H}_2$ controller for a randomized heterogeneous networked system with $12800$ states using a single laptop; for such a large system, neither the centralized nor the distributed optimal controllers can be computed. We then solve the localized $\mathcal{H}_2$ optimal control problem with joint sensor and actuator regularization to co-design an output feedback controller and its actuation/sensing architecture. Finally, we solve the localized mixed $\mathcal{H}_2 / \mathcal{L}_1$ optimal control problem, thus identifying the achievable tradeoff curve between the average and worst-case performance of the closed loop system.  

\subsection{Power System Model}
We begin with a randomized spanning tree embedded on a $10 \times 10$ mesh network representing the interconnection between subsystems. The resulting interconnected topology is shown in Fig. \ref{fig:mesh} -- we assume that all edges are undirected.
\begin{figure}[ht!]
      \centering
      \subfigure[Interconnected topology]{%
      \includegraphics[width=0.36\columnwidth]{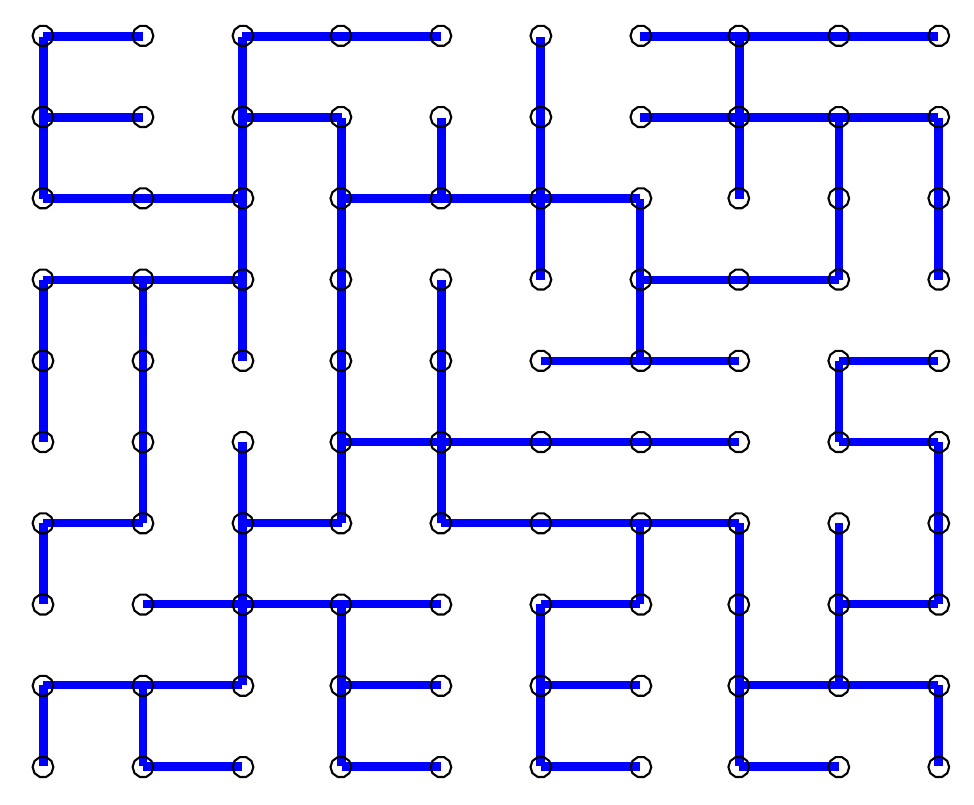}
      \label{fig:mesh}}      
      \subfigure[Interaction between neighboring subsystems]{%
      \includegraphics[width=0.45\columnwidth]{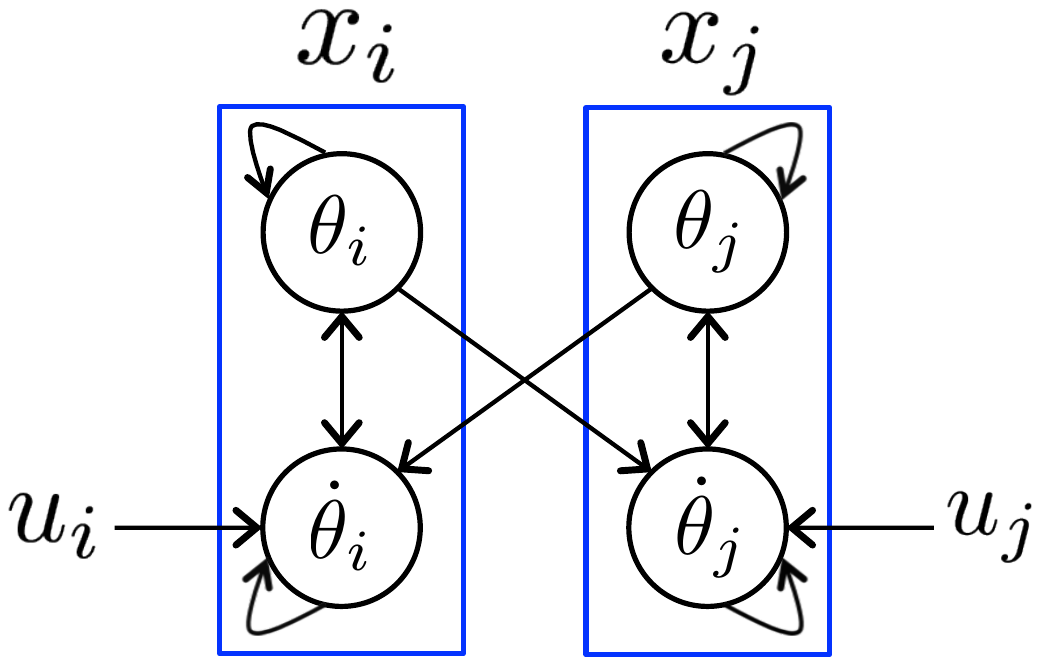}
      \label{fig:inter}} 
      \caption{Simulation example interaction graph.}
\end{figure}
The dynamics of each subsystem is given by the discretized swing equations. Consider the swing dynamics 
\begin{equation}
m_i \ddot{\theta_i} + d_i \dot{\theta_i} = -\sum_{j \in \mathcal{N}_i} k_{ij} (\theta_i - \theta_j) + w_i + u_i \label{eq:swing}
\end{equation}
where $\theta_i$, $\dot{\theta_i}$, $m_i$, $d_i$, $w_i$, $u_i$ are the phase angle deviation, frequency deviation, inertia, damping, external disturbance, and control action of the controllable load of bus $i$. The coefficient $k_{ij}$ is the coupling term between buses $i$ and $j$. We let $x_i := [\theta_i \quad \dot{\theta_i}]^\top$ be the state of bus $i$ and use $e^{A \Delta t} \approx I + A \Delta t$ to discretize the swing dynamics. Equation \eqref{eq:swing} can then be expressed in the form
\begin{eqnarray}
x_i[t+1] &=& A_{ii} x_i[t] + \sum_{j \in \mathcal{N}_i} A_{ij}x_j[t] + B_{ii} u_i[t] + \delta_{x_i}[t] \nonumber\\
y_i[t] &=& C_{ii} x_i[t] + \delta_{y_i}[t], \label{eq:interconnected}
\end{eqnarray}
with
\begin{equation}
A_{ii} = \begin{bmatrix} 1 & \quad \Delta{t} \\ -\frac{k_i}{m_i} \Delta{t} & \quad 1-\frac{d_i}{m_i} \Delta{t} \end{bmatrix}, \quad A_{ij} = \begin{bmatrix} 0 & 0 \\ \frac{k_{ij}}{m_i} \Delta{t} & 0 \end{bmatrix}, \nonumber
\end{equation}
and $B_{ii} = \begin{bmatrix} 0 & 1 \end{bmatrix}^\top$. We set $\Delta{t} = 0.2$ and $k_i = \sum_{j \in \mathcal{N}_i} k_{ij}$. In addition, the parameters $k_{ij}, d_i$, and $m_i^{-1}$ are randomly generated and uniformly distributed between $[0.5, 1]$, $[1, 1.5]$, and $[0, 2]$, respectively. 
The interaction between neighboring subsystems of the discretized model is illustrated in Figure \ref{fig:inter}. We initially assume that each subsystem in the power network has a phase measurement unit (PMU), a frequency sensor, and a controllable load that generates $u_i$.

From \eqref{eq:swing}, the external disturbance $w_i$ only directly affects the frequency deviation $\dot{\theta_i}$. 
To make the objective functional strongly convex, we introduce a small artificial disturbance on the phase deviation $\theta_i$ as well. 
We assume that the process noise on frequency and phase are uncorrelated and \rev{independently and identically distributed (i.i.d.) Gaussian random variables with zero mean,} and covariance matrices given by $I$ and $10^{-4}I$, respectively. In addition, we assume that both the phase deviation and the frequency deviation are measured with some sensor noise. The sensor noise of phase and frequency measurements are uncorrelated \rev{i.i.d. Gaussian random variables with zero mean, and} covariance matrices both given by $10^{-2}I$. We choose equal penalty on the state deviation and control effort, i.e., $\begin{bmatrix} C_1 & D_{12} \end{bmatrix} = I$. 

Based on the above setting, we formulate a $\mathcal{H}_2$ optimal control problem that minimizes the $\mathcal{H}_2$ norm of the transfer matrix from the process and sensor noises to the regulated output. The $\mathcal{H}_2$ norm of the closed loop is given by $13.3169$ when a proper centralized $\mathcal{H}_2$ optimal controller is applied, and $16.5441$ when a strictly proper centralized $\mathcal{H}_2$ optimal controller is applied. In the rest of this section, we normalize the $\mathcal{H}_2$ norm with respect to the proper centralized $\mathcal{H}_2$ optimal controller.



\subsection{Localized $\mathcal{H}_2$ Optimal Control}
The underlying assumption of the centralized optimal control scheme is that measurements can be  {instantaneously} transmitted among {all subsystems} in the network. 
To incorporate realistic communication delay constraints and facilitate the scalability of controller design, we impose additional communication delay constraints, locality constraints, and a FIR constraint on the system response. 

For the communication delay constraint $\mathcal C$, we assume that each subsystem takes one time step to transmit information to its neighboring subsystems.  Therefore, if subsystems $i$ and $j$ are $k$-hops apart (as defined by the interaction graph illustrated in Figure \ref{fig:mesh}), then the control action $u_i[t]$ at subsystem $i$ can only use the measurements $y_j[\tau]$ and internal controller state $\beta_j[\tau]$ of subsystem $j$ if $\tau \leq t-k$.

The interaction between subsystems illustrated in Fig. \ref{fig:inter} implies that it takes two discrete time steps for a disturbance at subsystem $j$ to propagate to its neighboring subsystems, and hence the communication speed between sub-controllers is twice as fast as the propagation speed of disturbances through the plant. For the given communication delay constraint $\mathcal C$, we use the method in \cite{2015_Wang_Reg} to design the tightest feasible locality constraint $\Ell$. In this example, we can localize the joint effect of the process and sensor noise at subsystem $j$ to a region defined by its two-hop neighbors (where one hop is as defined in terms of the interaction graph of the system). This implies that the sub-controller at node $j$ only needs to transmit its measurements $y_j$ and controller states $\beta_j$ within this localized region, and further only a restricted plant model (as defined by this localized region) is needed to synthesize its corresponding local control policy. 

Assuming a fixed communication delay constraint $\mathcal C$ and locality constraint $\Ell$, we first explore the tradeoff between the length $T$ of the FIR constraint $\FT$ and the transient performance of the closed loop system. Figure \ref{fig:T} shows the tradeoff curve between the transient performance of the localized $\mathcal{H}_2$ controller and the length $T$ of the FIR constraint.
For the given communication delay constraint $\mathcal C$ and the locality constraint $\Ell$, the localized $\mathcal{H}_2$ controller is feasible for FIR constraints $\FT$ whenever $T \geq 3$. When the length of the FIR constraint increases, the $\mathcal{H}_2$ norm of the closed loop converges quickly to the unconstrained optimal value. For instance, for FIR lengths of $T=7,\, 10,$ and $20$, the performance degradations of the localized FIR controllers with respect to the unconstrained $\mathcal{H}_2$ optimal controller are given by $3.8\%$, $1.0\%$, and $0.1\%$, respectively. This further means that the performance degradation due to the additional communication delay constraint $\mathcal C$ and the locality constraint $\Ell$ is less than $0.1\%$. From Figure \ref{fig:T}, we see that the localized $\mathcal{H}_2$ controller, with its additional communication delay, locality and FIR constraints can achieve similar transient performance to that of an unconstrained centralized (unimplementable) optimal $\mathcal{H}_2$ controller.

\begin{figure}[ht!]
      \centering
      \includegraphics[width=0.4\textwidth]{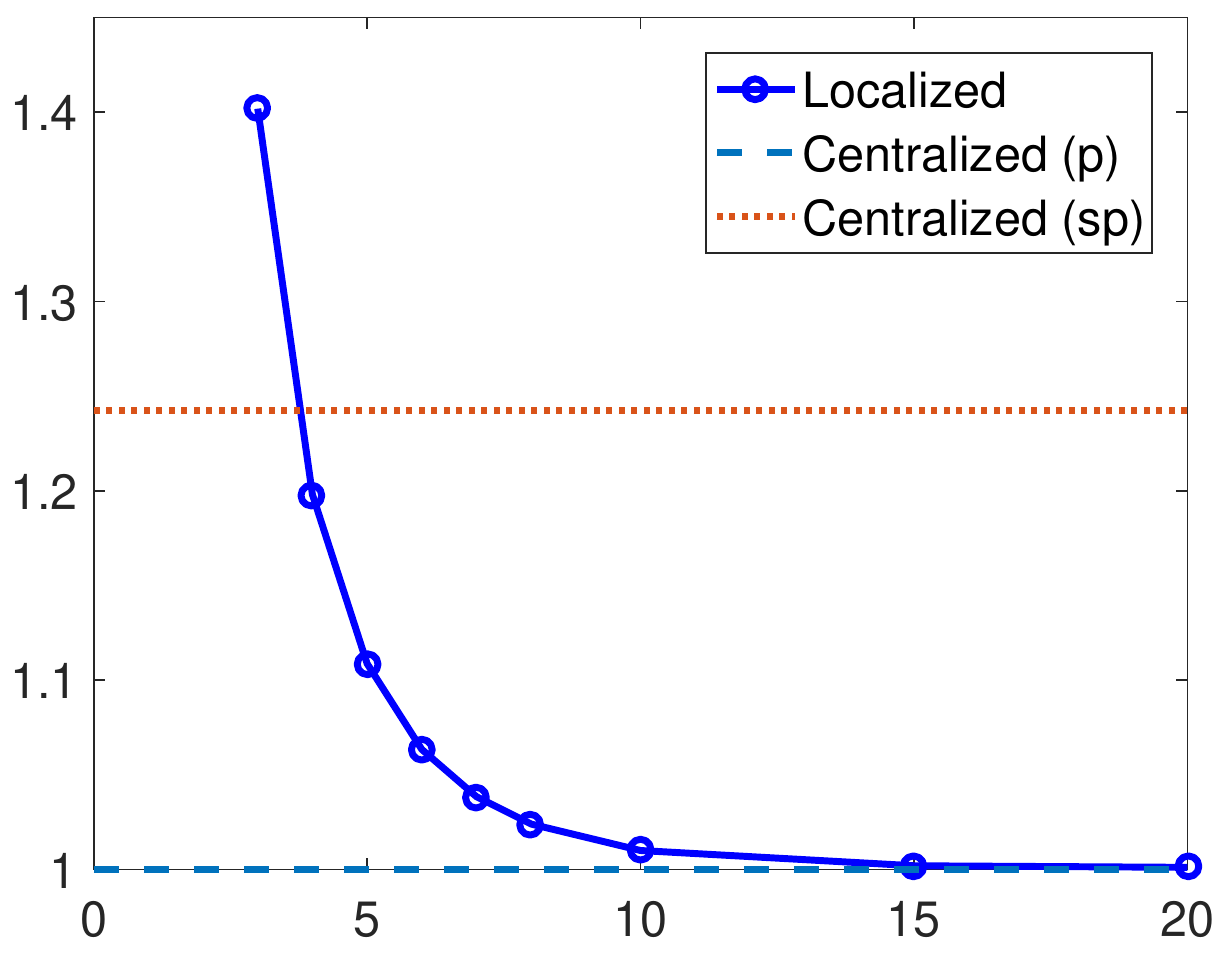}
      \caption{The vertical axis is the normalized $\mathcal{H}_2$ norm of the closed loop when the localized $\mathcal{H}_2$ controller is applied. The localized $\mathcal{H}_2$ controller is subject to the constraint $\mathcal C \cap \Ell \cap \FT$. The horizontal axis is the horizon $T$ of the FIR constraint $\FT$, which is also the settling time of an impulse disturbance. We plot the normalized $\mathcal{H}_2$ norm for the centralized unconstrained optimal controller (proper and strictly proper) in the same figure.}
      \label{fig:T}
\end{figure}

To further illustrate the advantages of the localized control scheme, we choose $T = 20$ and compare the localized optimal controller, distributed optimal controller, and the centralized optimal controller in terms of closed loop performance, complexity of controller synthesis and implementation, as well as ease of redesign,\footnote{\rev{The ability to locally redesign subsets of the control system has the appealing feature of allowing for real-time redesign in the case of local changes to the plant model.}} in Table \ref{Table:2}. The distributed optimal controller is computed using the method described in \cite{2014_Lamperski_H2_journal}, in which we assume the same communication constraint $\mathcal C$ as the localized optimal controller. It can be seen that the localized controller is vastly preferable in all aspects, except for a slight degradation in the closed-loop performance. 

\begin{table}[h]
 \caption{Comparison Between Centralized, Distributed, and Localized Control}
 \label{Table:2}
\begin{center}
    \begin{tabular}{| l | l | l | l | l |}
    \hline
     & & Cent. & Dist. & Local. \\ \hline
     & Affected region & Global & Global & $2$-hop \\ \cline{2-5}
     Closed Loop & Affected time & Long & Long & $20$ steps \\ \cline{2-5}
     & Normalized $\mathcal{H}_2$ & $1$ & $1.001$ & $1.001$ \\ \hline \hline
     & Complexity & $O(n^3)$ & $\geq O(n^3)$ & $O(n)$ \\ \cline{2-5}
     Synthesis & Plant model & Global & Global & $2$-hop \\ \cline{2-5}
     & Redesign & \rev{Global} & \rev{Global} & \rev{Local} \\ \hline \hline
     Implement. & Comm. Speed & $\infty$ & $2$ & $2$ \\ \cline{2-5}
     & Comm. Range & Global & Global & $2$-hop \\ \hline 
     \end{tabular}
\end{center}
\end{table}

We now allow the size of the problem to vary and compare the computation time needed to synthesize a centralized, distributed, and localized $\mathcal{H}_2$ optimal controller. We choose $T = 7$ for the localized controller. The empirical relationship obtained between computation time and problem size for different control schemes is illustrated in Figure \ref{fig:Time}. It is easily seen that the computation time needed for the distributed controller grows rapidly when the size of problem increases. The slope of the line describing the computation time of the centralized controller in the log-log plot of Figure \ref{fig:Time} is $3$, which matches the theoretical complexity $O(n^3)$. The slope for the localized $\mathcal{H}_2$ controller is about $1.4$, which is larger than the theoretical value of $1$. As  mentioned previously, we believe this overhead may be caused by other computational issues, such as memory management. We note that the computational bottleneck faced in computing our large-scale example arises from using a single laptop to compute the controller (and hence the localized subproblems were solved in serial) -- in practice, if each local subsystem is capable of solving its corresponding localized subproblem, our approach scales to systems of arbitrary size as all computations can be done in parallel. For the largest example that we have, we can compute the optimal localized $\mathcal{H}_2$ controller for a system with $12800$ states in $22$ minutes using \rev{MATLAB with an ASUS laptop with Intel i7-3610QM @ 2.3GHz CPU and 8 GB of RAM.} 
If the computation is parallelized across all $6400$ sub-systems, the synthesis algorithm can be completed within $0.2$ seconds. In contrast, the theoretical time to compute the centralized $\mathcal{H}_2$ optimal controller for the same example is more than a week.


\subsection{Localized $\mathcal{H}_2$ Optimal Control with Joint Sensor and Actuator Regularization}
We now move back to the $10 \times 10$ mesh example shown in Figure \ref{fig:mesh}. In the previous subsection, we assumed that each subsystem in the power network had a phase sensor, a frequency sensor, and a controllable load. In practice, the installation of these sensors and actuators is expensive: therefore we explore the tradeoff between the closed loop performance of the system and the number of sensors and actuators being used. A challenging problem is to determine the optimal locations of these sensors and actuators due to the combinatorial nature of this task. In this subsection, we apply the regularization for design (RFD) \cite{MC_CDC14} framework to jointly design the localized optimal controller and the optimal locations of sensors and actuators in the power network. This is achieved by solving the localized $\mathcal{H}_2$ optimal control problem with joint sensor and actuator regularization \eqref{eq:h2SA}.

In order to allow for more flexibility in sensor and actuator placement, we increase the localized region of each process and sensor noise from its two-hop neighbors to its four-hop neighbors. This implies that each subsystem $j$ needs to exchange information with up to its four-hop neighbors, and use the plant model restricted to this four-hop neighborhood to synthesize the localized $\mathcal{H}_2$ controller. Similarly, the length of the FIR constraint $\FT$ is increased to $T = 30$. The initial localized $\mathcal{H}_2$ cost is given by $13.3210$, which is a $0.03\%$ degradation with respect to the optimal centralized $\mathcal{H}_2$ controller. We assume that the relative prices between each frequency sensor, PMU, and controllable load are $1$, $100$, and $300$, respectively. This is to model the fact that actuators are typically more expensive than sensors, and that PMUs are typically more expensive than frequency sensors. The price for the same types of sensors and actuators at different locations remains constant.

\begin{figure}[ht!]
      \centering
      \includegraphics[width=0.25\textwidth]{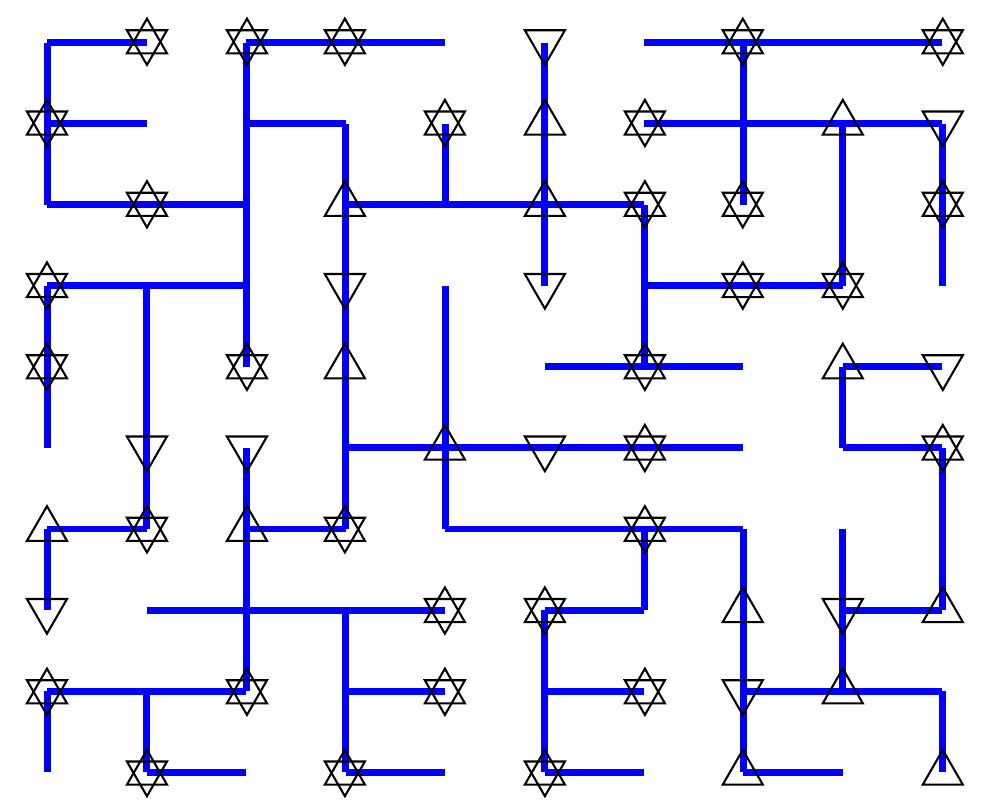}
      \caption{The upward-pointing triangles represent the subsystems in which the PMU is removed. The downward-pointing triangles represent the subsystems in which the controllable load (actuator) is removed.}
      \label{fig:SA}
\end{figure}

We begin with a dense controller architecture composed of $100$ controllable loads, $100$ PMUs, and $100$ frequency sensors, i.e., one of each type of sensor and actuator at each node.  Using optimization problem \eqref{eq:h2SA} to identify the sparsest possible architecture that still satisfies the locality, FIR and communication delay constraints, we are able to remove $43$ controllable loads and $46$ PMUs from the system (no frequency sensors were removed due to the chosen relative pricing). The locations of the removed sensors and actuators are shown in Figure \ref{fig:SA}. We argue that this sensing and actuation interface is very sparse. In particular, we only use $57$ controllable loads to control process noise from $200$ states and sensor noise from $154$ states, while ensuring that the system response to {all} process and sensor disturbances remains both localized and FIR.

The localized $\mathcal{H}_2$ cost for the system with reduced number of sensors and actuators is given by $17.8620$. In comparison, the cost achieved by a proper centralized optimal controller is $16.2280$, and the cost achieved by a strictly proper centralized optimal controller is $18.4707$. Note that as the sensing and actuation interface becomes sparser, the performance gap between the centralized and the localized controller becomes larger. 
Nevertheless, we note that the performance degradation is only $10\%$ compared to the proper centralized optimal scheme implemented using the same sparse controller architecture.

\subsection{Localized Mixed $\mathcal{H}_2/\mathcal{L}_1$ Optimal Control}
Finally, we solve the localized mixed $\mathcal{H}_2/\mathcal{L}_1$ optimal control problem in \eqref{eq:h2l1} on the $10 \times 10$ mesh example shown in Figure \ref{fig:mesh}. We progressively reduce the allowable $\mathcal{L}_1$ gain $\gamma$, as shown in equation \eqref{eq:h2l1}, to explore the tradeoff between average and worst-case performance of a system. We plot the normalized $\mathcal{H}_2$ norm and the normalized $\mathcal{L}_1$ norm in Figure \ref{fig:H2L1}.\footnote{We normalize the $\mathcal{H}_2$ performance with respect to that achieved by the optimal localized $\mathcal{H}_2$ controller, and likewise normalize the $\mathcal{L}_1$ performance with respect to that achieved by the optimal localized $\mathcal{L}_1$ controller.} The left-top point in Figure \ref{fig:H2L1} is the localized $\mathcal{H}_2$ solution. When we start reducing the $\mathcal{L}_1$ sublevel set, the $\mathcal{H}_2$ norm of the closed loop response gradually increases, thus tracing out a tradeoff curve. 


\begin{figure}[ht!]
      \centering
      \includegraphics[width=0.35\textwidth]{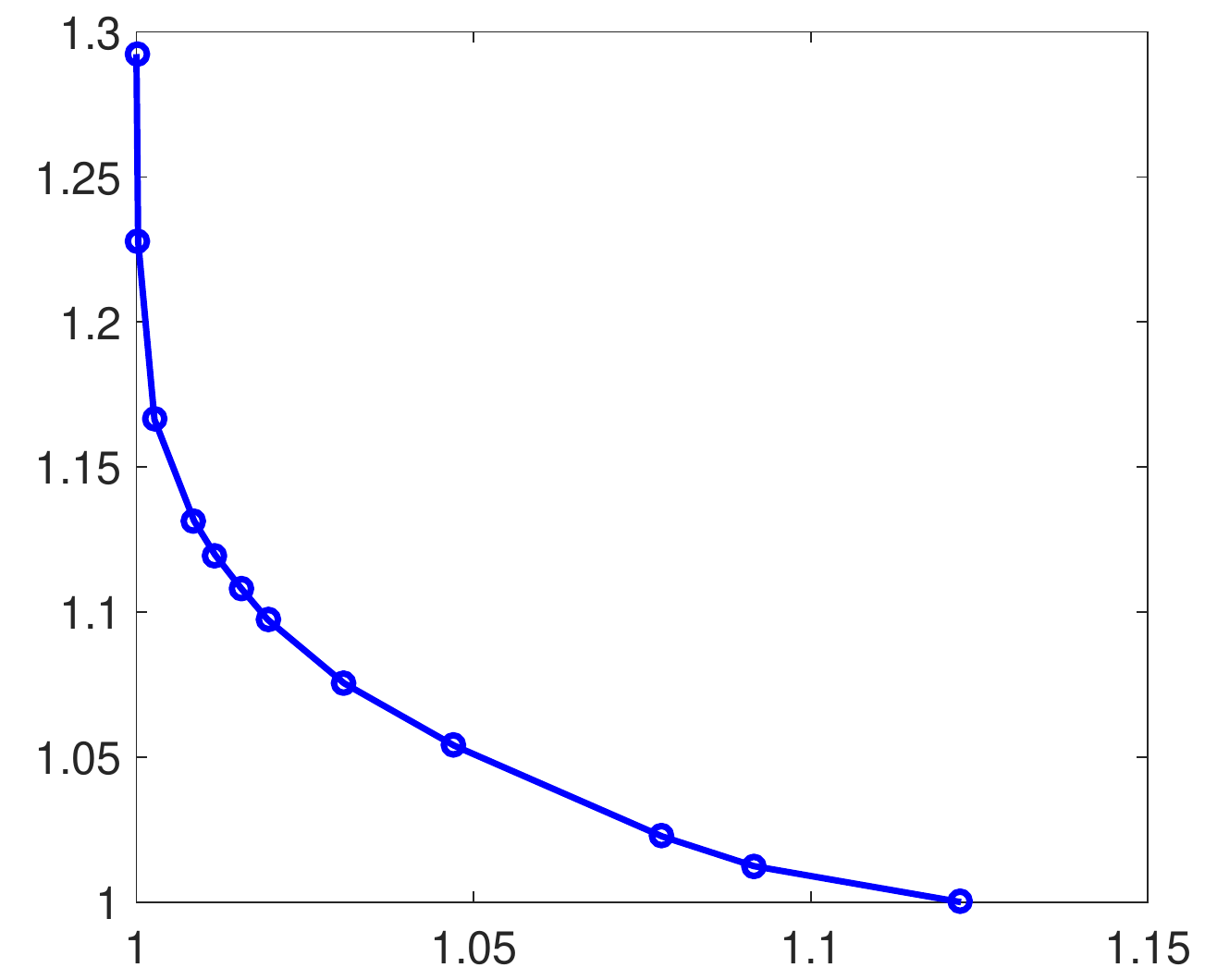}
      \caption{The vertical axis represents the normalized $\mathcal{L}_1$ norm of the closed loop, and the horizontal axis represents the normalized $\mathcal{H}_2$ norm of the closed loop.}
      \label{fig:H2L1}
\end{figure}

\subsection{Summary}
In this section, we extended many of the state-feedback results from Section \ref{sec:state_feedback_SLS} to the output feedback setting.  In particular, we showed that by exploiting the partial separability of many SLS problems of interest, output feedback problems also enjoy improved the scalability via a decomposition and dimensionality reduction based method combined with distributed optimization methods such as ADMM.  Lacking in the output feedback setting however is an analogous family of robustness results to those presented in Section \ref{sec:robust}.  Although preliminary results exist \cite{boczar2018finite}, this remains in important direction for future work.

\section{Conclusions}
\label{sec:conclusion}
This article reviewed the System Level Synthesis framework.  Our aim was to provide a self-contained and accessible resource that summarizes  progress in developing this novel approach to controller synthesis.  We highlighted the benefits that SLS provides in the context of distributed control and in robust control.  As we hope we were able to show, by working directly with system responses, there is a transparency in how system constraints, structure, and uncertainty affect controller synthesis, implementation, and performance, and it is this transparency that we exploit throughout to improve upon the state-of-the-art.  As lengthy as this article is, we believe that we have just begun to scratch the surface of what SLS can do.  Exciting current directions of research focus upon integrating SLS into model predictive control algorithms, further understanding the algebraic structure underlying localized controllers and their state-space realizations, further developing robust SLS theory (especially in the output feedback setting), and applying these tools to application areas spanning power-systems, computer networking, and machine learning/AI.

\bibliographystyle{IEEEtran}
\bibliography{Distributed_new,lqr}

\end{document}